\documentclass[preprint,12pt]{elsarticle}
\usepackage{amsmath,amssymb,amsthm,amscd, mathrsfs}

\usepackage[all,cmtip,ps]{xy}

\usepackage[english]{babel}

\usepackage{graphics}

\renewcommand{\iff}{if and only if }

\newcommand{\st}{such that }

\newcommand{\Z}{\mathbb{Z}}

\newcommand{\Q}{\mathbb{Q}}

\theoremstyle{plain}

\newtheorem{thm}{Theorem}[section]

\newtheorem{prop}[thm]{Proposition}

\newtheorem{lem}[thm]{Lemma}

\newtheorem{cor}[thm]{Corollary}

\newtheorem{ex}[thm]{Example}

\theoremstyle{definition}

\newtheorem{defn}[thm]{Definition}

\theoremstyle{remark}

\newtheorem{rem}[thm]{Remark}

\numberwithin{equation}{section}

\newcommand{\la}{\lambda}

\newcommand{\ra}{\rightarrow}

\newcommand{\lra}{\longrightarrow}

\newcommand{\ten}{\otimes}

\newcommand{\arc}{\ar@{-}@/_/}
\newcommand{\tra}{\ar@{-}}
\newcommand{\dta}{\ar@{.}}
\newcommand{\ltra}{\ar@{<-}}
\newcommand{\rtra}{\ar@{->}}
\newcommand{\larc}{\ar@{<-}@/_/}
\newcommand{\rarc}{\ar@{<-}@/_/}

\begin{document}
\begin{frontmatter}

\title{Cellular structure of $q$-Brauer algebras}

\author{Dung Tien Nguyen}

\address{\emph{Present address} : University of Stuttgart, Institute of Algebra and Number Theory,
Pfaffenwaldring 57, D-70569 Stuttgart, Germany}

\address{\emph{Permanent address} :University of Vinh, Mathematical Faculty,
Leduan 182, Vinh city, Vietnam}

\ead{nguyen@mathematik.uni-stuttgart.de}


{\begin{abstract}
In this paper we consider the $q$-Brauer algebra over $R$ a commutative noetherian domain.
We first construct a new basis for $q$-Brauer algebras,
and we then prove that it is a cell basis, and thus these algebras are cellular in the sense of Graham and Lehrer.
In particular, they are shown to be an iterated inflation of Hecke algebras of type $A_{n-1}.$
Moreover, when $R$ is a field of arbitrary characteristic, we determine for which parameters the $q$-Brauer algebras are quasi-heredity.
So the general theory of cellular algebras and quasi-hereditary algebras applies to $q$-Brauer algebras.
As a consequence, we can determine all irreducible representations of $q$-Brauer algebras by linear algebra methods.
\end{abstract}}

\end{frontmatter}

\section{Introduction}
Schur-Weyl duality relates the representation theory of infinite group $GL_{N}(R)$ with that of symmetric group $S_{n}$
via the mutually centralizing actions of two groups on the tensor power space $(R^{N})^\otimes{^n}$.
In 1937 Richard Brauer \cite{Br} introduced the algebras which are now called 'Brauer algebra'.
These algebras appear in an analogous situation where $GL_{N}(R)$ is replaced by either a symplectic or an orthogonal group
and the group algebra of the symmetric group is replaced by a Brauer algebra.
Afterwards, a q-deformation of these algebras has been found
in \cite{BW} by Birman and Wenzl, and independently by Murakami \cite{Mu}, which is referred nowadays as BMW-algebras, in connection with knot theory and quantum groups.

Recently a new algebra was introduced by Wenzl \cite{W3} via generators and relations who called it the $q$-Brauer algebra and gave a detailed description of its structure.
This algebra is another q-deformation of Brauer algebras and contains the Hecke algebra of type $A_{n-1}$ as a subalgebra.
In particular, he proved that over the field $\Q(r,q)$ it is semisimple and isomorphic to the Brauer algebra.
Subsequently Wenzl found a number of applications of $q$-Brauer algebras, such as in the study of the representation ring of an orthogonal or symplectic group \cite{W4},
of module categories of fusion categories of type A and corresponding subfactors of type II$_{1}$ in \cite{W5}.
Another quantum analogue of the Brauer algebra had been introduced by Molev in \cite{Mo1}. He has shown that Wenzl's $q$-Brauer algebra is a quotient of his algebra in \cite{Mo2}.

It is well known that algebras as the group algebra of symmetric groups, the Hecke algebra of type $A_{n-1}$,
the Brauer algebra and its $q$-deformation the BMW-algebra are cellular algebras in the sense of Graham and Lehrer (see \cite{GL, KX4, X2}).
Hence, a question arising naturally from this situation is whether the $q$-Brauer algebra is also a cellular algebra.

This question will be answered in the positive in this paper via using an equivalent definition with that of Graham and Lehrer defined by Koenig and Xi in \cite{KX2}
and their approach, called \text{ \it 'iterated inflation'}, to cellular structure as in \cite{KX4} (or \cite{KX2}).
Firstly, we will construct a basis for $q$-Brauer algebras which is labeled by a natural basis of Brauer algebras.
This basis allows us to find cellular structure on $q$-Brauer algebras. Also note that there is a general basis for the $q$-Brauer algebras introduced in \cite{W3} by Wenzl,
but it seems that his basis is not suitable to provide a cell basis. Subsequently, we prove that these $q$-Brauer algebras are cellular in the sense of Graham and Lehrer.
More precisely, we exhibit on this algebra an iterated inflation structure as that of the Brauer algebra in \cite{KX4} and of the BMW-algebra in \cite{X2}. These results are valid over a commutative noetherian ring.
These enable us to determine the irreducible representations of $q$-Brauer algebras over arbitrary field of any characteristic by using the standard methods in the general theory of cellular algebras.
As another application, we give the choices for parameters such that the corresponding $q$-Brauer algebras are quasi-hereditary in the sense of \cite{CPS}.
Thus, for those $q$-Brauer algebras, the finite dimensional left modules form a highest weight category with many important homological properties \cite{PS}.
Under a suitable choice of parameters all results of the $q$-Brauer algebra obtained in this paper recover those of the Brauer algebra with non-zero parameter given by Koenig and Xi (\cite{KX3}, \cite{KX4}).

The paper will be organized as follows: In Section two we recall the definition of Brauer algebras and of Hecke algebras of type $A_{n-1}$
as well as the axiomatics of cellular algebras. Some basic results on representations of Brauer algebras are given here.
In Section three we indicate the definition of generic $q$-Brauer algebras as well as collect and extend some basic and necessary properties for later reference.
Theorem \ref{dl11} points out a particular basis on the $q$-Brauer algebras that is then shown a cell basis.
In Section four we prove the main result, Theorem \ref{dl16}, that the $q$-Brauer algebras have cellular structure.
Proposition \ref{md14} tells us that any $q$-Brauer algebra can be presented as an iterated inflation of certain Hecke algebras of type $A_{n-1}$.
The other results are obtained by applying the theory of cellular algebras.
In Section five we determine a necessary and sufficient condition for the $q$-Brauer algebra to be quasi-hereditary, Theorem \ref{qhthm}.

\section{Basic definitions and preliminaries}

In this section we recall the definition of Brauer and Hecke algebras and also that of cellular algebras.

\subsection{Brauer algebra}
Brauer algebras were introduced first by Richard Brauer \cite{Br} in order to study the ${n}$th tensor power of the defining representation of
the orthogonal groups and symplectic groups. Afterward, they were studied in more detail by various mathematicians.
We refer the reader to work of Brown (see \cite{Bro1, Bro2}), Hanlon and Wales (\cite{HW1, HW2, HW3})
or Graham and Lehrer \cite{GL}, Koenig and Xi \cite{KX3, KX4}, Wenzl \cite{W1} for more information.

\subsubsection{Definition of Brauer algebras}
The Brauer algebra is defined over the ring $\mathbb{Z}[x]$ via a basis given by diagrams with 2n vertices, arranged in two rows with n edges in each row, where each vertex belongs to exactly one edge.
The edges which connect two vertices on the same row are called \textit{horizontal edges}.
The other ones are called \textit{vertical edges}. We denote by $D_{n}(x)$  the Brauer algebra
which the vertices of diagrams are numbered 1 to n from left to right in both the top and the bottom.
Two diagrams $d_{1}$ and $d_{2}$ are multiplied by concatenation, that is, the bottom vertices of $d_1$
are identified with the top vertices of $d_2$, hence defining diagram $d$.
Then $d_1 \cdot d_2$ is defined to be $x^{\gamma(d_{1},\ d_{2})}$d, where
$\gamma(d_{1},\ d_{2})$ denote the number of those connected components of the concatenation of
$d_1$ and $d_2$ which do not appear in $d$, that is, which contain neither a top vertex of $d_1$ nor a bottom vertex of $d_2$.\\
Let us demonstrate this by an example. We multiply two elements in $D_{7}(x)$:
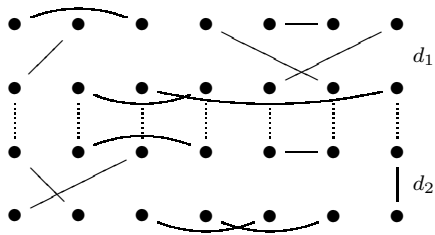
\begin{center}
$\begin{array}{c}
\begin{xy}
\xymatrix@!=0.01pc{ \bullet \ar@{-}@/^/[rr] &\bullet \ar@{-}[ld] &\bullet & \bullet \ar@{-}[rrd] &\bullet \ar@{-}[r] &\bullet &\bullet \ar@{-}[lld] \ar@{}[d]^{\text{ $d_{1}$ }} \\
    \bullet \dta[d] & \bullet \dta[d]  \ar@{-}@/_/[rr] &\bullet \dta[d] \ar@{-}@/_/[rrrr] & \bullet \dta[d] &\bullet \dta[d] &\bullet \dta[d] &\bullet \dta[d] \\
    \bullet \ar@{-}[rd] & \bullet \ar@{-}@/^/[rr] &\bullet \ar@{-}[lld] & \bullet &\bullet \ar@{-}[r] &\bullet &\bullet \ar@{-}[d] \ar@{}[d]^{\text{ $ d_{2}$ }} \\
    \bullet & \bullet & \bullet \ar@{-}@/_/[rr] & \bullet \ar@{-}@/_/[rr] &\bullet &\bullet &\bullet }
\end{xy}
\end{array}$
\end{center}
and the resulting diagram is
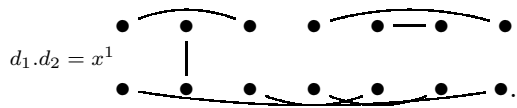
\begin{center}
$\begin{array}{c}
\begin{xy}
\xymatrix@!=0.01pc{ \bullet \ar@{-}@/^/[rr] &\bullet \ar@{-}[d] &\bullet & \bullet \ar@{-}@/^/[rrr] &\bullet \ar@{-}[r] &\bullet &\bullet \\
            \ar@{}[u]^{\text{$d_{1}. d_{2} = x^{1}$}} \bullet \ar@{-}@/_/[rrrrrr] & \bullet & \bullet \ar@{-}@/_/[rr] & \bullet \ar@{-}@/_/[rr] &\bullet &\bullet &\bullet .}
\end{xy}
\end{array}$
\end{center}

In (\cite{Br}, Section 5) Brauer points out that each basis diagram on $D_{n}(x)$ which has exactly 2k horizontal edges
can be obtained in the form $\omega_{1}e_{(k)}\omega_{2}$ with $\omega_{1}$ and $\omega_{2}$ are permutations in the symmetric group $S_{n}$,
and $e_{(k)}$ is a diagram of the following form:
$$ \begin{array}{c}
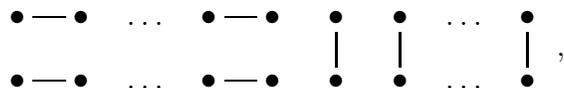
\begin{xy}
\xymatrix@!=0.01pc{ \bullet \ar@{-}[r] & \bullet & \ldots &\bullet \ar@{-}[r] & \bullet &\bullet \ar@{-}[d] &\bullet \ar@{-}[d] &\ldots &\bullet \ar@{-}[d]  \\
    \bullet \ar@{-}[r] & \bullet & \ldots & \bullet \ar@{-}[r] & \bullet &\bullet &\bullet &\ldots &\bullet }
\end{xy}
\end{array},$$
where each row has exactly k horizontal edges.

As a consequence, the Brauer algebra can be considered over a
polynomial ring over $\mathbb{Z}$ and is defined via generators and relations as follow:

Take $N$ to be an indeterminate
over $\mathbb{Z}$; Let $R=\mathbb{Z}[N]$ and define the Brauer
algebra $D_{n}(N)$ over $R$ as the associative unital $R$--algebra
generated by the transpositions $s_1,s_2,\dots,s_{n-1}$, together
with elements $e_{(1)},e_{(2)},\dots,e_{([n/2])}$, which satisfy the defining relations:
\begin{align*}
&(S_{0}) \hspace{1cm} s_i2=1&&\text{for $1\le i<n$;}\\
&(S_{1}) \hspace{1cm} s_is_{i+1}s_i=s_{i+1}s_is_{i+1}&&\text{for $1\le i<n-1$;}\\
&(S_{2}) \hspace{1cm} s_is_j=s_js_i&&\text{for $2\le|i-j|$;}\\
&(1) \hspace{1cm} e_{(k)}e_{(i)}=e_{(i)}e_{(k)}=N^{i}e_{(k)}&&\text{for $1 \le i \le k \le [n/2]$;}\\
&(2) \hspace{1cm}  e_{(i)}s_{2j}e_{(k)}=e_{(k)}s_{2j}e_{(i)}=N^{i - 1}e_{(k)}&&\text{for $1\le j \le i \le k \le [n/2]$;}\\
&(3) \hspace{1cm}  s_{2i+1}e_{(k)}=e_{(k)}s_{2i+1} =e_{(k)} &&\text{for $0 \le i < k \le [n/2]$;}\\
&(4) \hspace{1cm}  e_{(k)}s_i=s_ie_{(k)}&&\text{for $2k < i<n$ ;}\\
&(5) \hspace{1cm}  s_{(2i-1)}s_{2i}e_{(k)}=s_{(2i+1)}s_{2i}e_{(k)}&&\text{for $1\le i < k \le k \le [n/2]$;}\\
&(6) \hspace{1cm}  e_{(k)}s_{2i}s_{(2i-1)}=e_{(k)}s_{2i}s_{(2i+1)}&&\text{for $1\le i < k \le k \le [n/2]$;}\\
&(7) \hspace{1cm}  e_{(k+1)} = e_{(1)}s_{2, 2k+1}s_{1, 2k}e_{(k)}&&\text{for $1\le k \le [n/2]-1$.}
\end{align*}
Regard the group ring $RS_n$ as the subring of $D_n(N)$
generated by the transpositions
$$\{s_i=(i,i+1) \text{for $1\le i<n$}\}.$$
By Brown (\cite{Bro2}, Section 3) the Brauer algebra has a decomposition as direct sum of vector spaces
$$D_{n}(x)\cong\bigoplus_{k=0}^{[n/2]}\mathbb{Z}[x]S_{n}e_{(k)}S_{n}.$$
Set $$\textit{I(m)}=\bigoplus_{k\geq{m}}\mathbb{Z}[x]S_{n}e_{(k)}S_{n},$$
then \textit{I(m)} is a two-sided ideal in $D_{n}(x)$ for each $m\leq{[n/2]}.$

\subsubsection{The modules $V^*_k$ and $V_k$ for Brauer algebras}
In this subsection we recall particular modules of Brauer algebras.
$D_{n}(N)$ has a decomposition into $D_{n}(N) - D_{n}(N)$ bimodules
$$D_{n}(N)\cong\bigoplus_{k=0}^{[n/2]}\mathbb{Z}[N]S_{n}e_{(k)}S_{n}+\textit{I(k+1)/I(k+1)}.$$
Using the same arguments as in Section 1 (\cite{W3}), each factor module
$$\mathbb{Z}[N]S_{n}{e_{(k)}}\omega_{j} + \textit{I(k+1)/I(k+1)}$$
is a left $D_{n}(N)$-module with a basis given by the basis diagrams of
$\mathbb{Z}[N]S_{n}{e_{(k)}}\omega_{j}$, where $\omega_{j} \in S_{n}$ is a diagram
such that ${e_{(k)}}\omega_{j}$ is a diagram in $D_{n}(N)$ with no intersection between any two vertical edges.
In particular
\begin{align} \label{eq1}
\hspace{1cm}  \textit{I(k)/I(k+1)} \cong\bigoplus_{j \in P(n,k)}(\mathbb{Z}[N]S_{n}e_{(k)}\omega_{j}+\textit{I(k+1))/I(k+1)},
\end{align}
where $P(n,k)$ is the set of all possibilities of $\omega_{j}.$
As multiplication from the right by $\omega_{j}$ commutes with the $D_{n}(N)$-action,
it implies that each summand on the right hand side is isomorphic to the module
\begin{align} \label{eq2}
\hspace{1cm}  V^*_k=(\mathbb{Z}[N]S_{n}{e_{(k)}} + \textit{I(k+1))/I(k+1)}.
\end{align}
Combinatorially, $V^*_k$ is spanned by basis diagrams with exactly k edges in the bottom row, where the $i-th$ edge is connected the vertices $2i-1$ and $2i$.
Observe that $V^*_k$ is a free, finitely generated $\mathbb{Z}[N]$ module with $\mathbb{Z}[N]$-rank $n!/2^{k}k!$.\\
Similarly, the right $D_{n}(N)$-module is defined
\begin{align} \label{eq3}
\hspace{1cm} V_k=(\mathbb{Z}[N]{e_{(k)}S_{n}} + \textit{I(k+1))/I(k+1)},
\end{align}
where the basis diagrams are obtained from those in $V^*_k$ by an involution, say *, of $D_n(N)$ which rotates a diagram $d \in V^*_k$ around its horizontal axis downward.
For convenience later in Section 4 we use the term $V^*_k$ replacing $V_{n}^{(k)}$ in \cite{W3}. More details for setting up $V^*_k$ can be found in Section 1 of \cite{W3}.
\begin{lem} \label{bd1}(\cite{W3}, Lemma 1.1(d)). The algebra $D_{n}(N)$ is faithfully represented on $\bigoplus_{k=0}^{[n/2]}V^*_k$ (and also on $\bigoplus_{k=0}^{[n/2]}V_k$).
\end{lem}

\subsubsection{Length function for Brauer algebras $D_{n}(N)$}
Generalizing the length of elements in reflection groups, Wenzl \cite{W3} defined a length function for a diagram of $D_{n}(N)$ as follows:\\
For a diagram d $\in D_{n}(N)$ with exactly $2k$ horizontal edges, the definition of the length $\ell{(d)}$ is given by
$$\ell{(d)}=min\{\ell{(\omega_{1})} +\ell{(\omega_{2})} | \quad  \omega_{1}e_{(k)}\omega_{2}=d, \ \omega_{1}, \omega_{2} \in {S_{n}} \}.$$
We will call the diagrams $d$ of the form $\omega{e_{(k)}}$ where $l(\omega) = l(d)$ and $\omega \in S_{n}$ {\it{basis diagrams}} of the module $V^*_k$.
\begin{rem}
1. Recall that the length of a permutation $\omega \in S_{n}$ is defined by $\ell(\omega) =$ the cardinality of set $\{ (i, j) |  (j)\omega < (i)\omega$ and $1 \leq i < j \leq  n\}$,
where the symmetric group acts on $\{ 1, 2, ..., n \}$ on the right.

2. Given a diagram d, there can be more than one $\omega$ satisfying $\omega{e_{(k)}}= d$ and $\ell{(\omega)} = \ell{(d)}$.
e.g. $s_{2j-1}s_{2j}e_{(k)} = s_{2j+1}s_{2j}e_{(k)}$ for $2j+1 < k$.
This means that such an expression of d is not unique with respect to $\omega \in S_{n}$.
\item For later use, if $s_{i} = (i, i+1)$ is a transposition in symmetric goup $S_{n}$ with $i, j= 1, ... , k$
let
\[s_{i,j} =
\begin{cases}
s_{i}s_{i+1}...s_{j} &\text{ if \ } i \leq j,\\
s_{i}s_{i-1}...s_{j} &\text{ if \ } i > j.
\end{cases} \]
A permutation $\omega \in S_{n}$ can be written uniquely in the form $\omega = t_{n-1}t_{n-2}...t_{1}$,
where $t_{j} = 1$ or $t_{j} = s_{i_{j},j}$ with $1 \leq{i_{j}} \leq{j}$ and $1 \leq{j}<{n}$.
This can be seen as follows: For given $\omega \in S_{n}$,
there exists a unique $t_{n-1}$ such that $(n)t_{n-1}=(n)\omega$. Hence $\omega' = (n)t^{-1}_{n-1}\omega = n$,
and we can consider $\omega'$ as an element of $S_{n-1}.$ Repeating this process on n implies the general claim.
Set
\begin{align} \label{eq4}
\qquad B^*_k = \{ t_{n-1} t_{n-2}... t_{2k} t_{2k-2}t_{2k-4}...t_{2} \}.
\end{align}
By the definition of $t_{j}$ given above, the number of possibilities of $t_{j}$ is $j+1$.
A direct computation shows that $B^*_k$ has $n!/2^{k}k!$ elements.
In fact, the number of elements in $B^*_k$ is equal to the number of diagrams $d^*$ in $D_{n}(N)$
in which $d^*$ has k horizontal edges in each row and one of its rows is fixed like that of $e_{(k)}$.

3. From now on, a permutation of symmetric group is seen as a diagram with no horizontal edges,
and the product $\omega_{1}\omega_{2}$ in $S_{n}$ is seen as a concatenation of two diagrams in $D_{n}(N)$.

4. Given a basis diagram $d^*=\omega e_{(k)}$ with $\ell(\omega) = \ell{(d^*)}$
does not imply that $\omega \in B^*_k$, but there does exist $\omega' \in B^*_k$ such that  $d^* = \omega'e_{(k)}$.
The latter was already shown by Wenzl (Lemma 1.2(a), \cite{W3}) to exist and to be unique for each basis element of the module $V^*_k$.
He even got $\ell{(d^*)} = \ell(\omega) = \ell(\omega')$,
where $\ell(\omega')$ is the number of factors for $\omega'$ in $B^*_k$.
\end{rem}

\begin{ex} As indicated in remark (2), we choose $j=1, k = 2$. Given a basis diagram $d^*$ in $V^*_2$ is in the form:
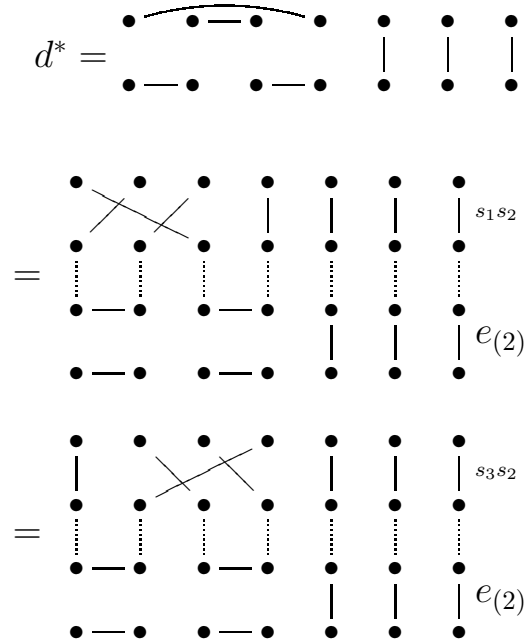
\begin{center}
$ \begin{array}{c}
\begin{xy}
\xymatrix@!=0.01pc{ \ & \ &\ & \ & \ &\  \\ \bullet \ar@{-}@/^/[rrr] &\bullet \ar@{-}[r] &\bullet &\bullet &\bullet \ar@{-}[d] &\bullet \ar@{-}[d] &\bullet \ar@{-}[d]  \\
    \ar@{}[u]^{\text{ \large $d^* = $ }} \bullet \ar@{-}[r] &\bullet &\bullet \ar@{-}[r] &\bullet &\bullet &\bullet &\bullet \\
    \ & \ & \ & \ & \ & \ }
\end{xy}
\end{array}$

$\begin{array}{c}
\begin{xy}
\xymatrix@!=0.01pc{ \bullet \ar@{-}[rrd] &\bullet \ar@{-}[ld] &\bullet \ar@{-}[ld] & \bullet \ar@{-}[d] &\bullet \ar@{-}[d] &\bullet \ar@{-}[d] &\bullet \ar@{-}[d] \ar@{}[d]^{\text{ $s_{1}s_{2}$  }} \\
    \bullet \dta[d] & \bullet \dta[d] &\bullet \dta[d] & \bullet \dta[d] &\bullet \dta[d] &\bullet \dta[d] &\bullet \dta[d] \\
    \ar@{}[u]^{\text{ \large $ =$ \ }} \bullet \ar@{-}[r] & \bullet &\bullet \ar@{-}[r] & \bullet &\bullet \ar@{-}[d] &\bullet \ar@{-}[d] &\bullet \ar@{-}[d] \ar@{}[d]^{\text{ \large $ e_{(2)}$ }} \\
    \bullet \ar@{-}[r] & \bullet & \bullet \ar@{-}[r] & \bullet &\bullet &\bullet &\bullet }
\end{xy}
\end{array}$

\bigskip

$\begin{array}{c}
\begin{xy}
\xymatrix@!=0.01pc{ \bullet \ar@{-}[d] &\bullet \ar@{-}[rd] &\bullet \ar@{-}[rd] & \bullet \ar@{-}[lld] &\bullet \ar@{-}[d] &\bullet \ar@{-}[d] &\bullet \ar@{-}[d] \ar@{}[d]^{\text{ $ s_{3}s_{2}$  }} \\
    \bullet \dta[d] & \bullet \dta[d] &\bullet \dta[d] & \bullet \dta[d] &\bullet \dta[d] &\bullet \dta[d] &\bullet \dta[d] \\
    \ar@{}[u]^{\text{ \large $ =$ \ }} \bullet \ar@{-}[r] & \bullet &\bullet \ar@{-}[r] & \bullet &\bullet \ar@{-}[d] &\bullet \ar@{-}[d] &\bullet \ar@{-}[d] \ar@{}[d]^{\text{ \large $ e_{(2)}$ }} \\
    \bullet \ar@{-}[r] & \bullet & \bullet \ar@{-}[r] & \bullet &\bullet &\bullet &\bullet }
\end{xy}
\end{array}$

\end{center}
In the picture $d^*$ has two representations $d^*=s_{1}s_{2}e_{(2)} = s_{3}s_{2}e_{(2)}$ satisfying
$$\ell{(d^*)} = \ell(s_{1}s_{2})=\ell(s_{3}s_{2}) = 2.$$
However, $s_{1}s_{2}$ is in $B^*_2$ but $s_{3}s_{2}$ is not.
In general, given a basis diagram $d^*$ in $V^*_k$ there always exists a unique permutation $\omega \in B^*_k$
such that $d^* = \omega e_{k}$ and $l(d^*)= l(\omega)$.\\
The statement in the remark (5) is shown in the following lemma.
\end{ex}

\begin{lem}\label{bd2}(\cite{W3}, Lemma 1.2).
(a) The module $V^*_k$ has a basis $\{ \omega v_{1} = v_{\omega e_{(k)}},\ \omega \in  B^*_k \}$ with $\ell{(\omega e_{(k)})} = \ell(\omega).$
Here $\ell(\omega)$ is the number of factors for $\omega$ in \eqref{eq4},
and\\$v_{1} = (e_{(k)} + I(k+1))/I(k+1) \in V^*_k$.

(b) For any basis element $d^*$ of $V^*_k$, we have $|\ell{(s_{i}d^*)} - \ell{(d^*)}| \leq 1$. Equality of lengths holds only if $s_{i}d^* = d^*$.
\end{lem}

For $k\leq [n/2]$, let
\begin{align} \label{eq5}
B_k = \{\omega^{-1} | \ \omega \in B^*_k \}.
\end{align}
The following statement is similar to Lemma \ref{bd2} above.

\begin{lem} \label{bd3}
(a) The module $V_k$  has a basis $\{ v_{1}\omega = v_{e_{(k)}\omega },\ \omega \in B_k \}$
with $\ell{(e_{(k)}\omega)} = \ell(\omega).$ Here $\ell(\omega)$ is the number of factors for $\omega$ in \eqref{eq5},
and  $v_{1} = (e_{(k)} + I(k+1))/I(k+1) \in V_k$.

(b) For any basis element $d$ of $V^{(k)}_n$, we have $|\ell{(ds_{i})} - \ell{(d)}| \leq 1$.
Equality of lengths holds only if $ds_{i} = d$.
\end{lem}

\subsection{Hecke algebra}
In this section, we recall the definition of the Hecke algebra of type $A_{n-1}$, following Dipper and James \cite{DJ1},
and we then review some basic facts which are necessary for our subsequent work.

\begin{defn} \label{dn5} Let $R$ be a commutative ring with identity 1, and let q be an invertible element of $R$.
The Hecke algebra $ H = H_{R, q}= H_{R, q}(S_{n})$ of the symmetric group $S_{n}$
over $R$ is defined as follows. As an $R$-module, $H$ is free with basis $\{ g_{\omega} |\ \omega \in S_{n}\}$.
The multiplication in $H$ satisfies the following relations:
\begin{itemize}
\item[(i)] $\boldmath{1} \in H$;
\item[(ii)] If $\omega = s_{1}s_{2}...s_{j}$ is a reduced expression for $\omega \in  S_{n}$, then $g_{\omega} = g_{s_{1}}g_{s_{2}}...g_{s_{j}}$;
\item[(iii)] $g^{2}_{s_{j}} = (q-1)g_{s_{j}} + q$ for all transpositions $s_{j}$, where $q = q.\boldmath{1} \in H$.
\end{itemize}
\end{defn}

\smallskip

Recall that if $\omega = s_{1}s_{2}...s_{k}$ with $s_{j} = (j, j+1)$ and k is minimal with this property,
then $\ell(\omega)= k$, and we call $s_{1}s_{2}...s_{k}$ \textit{a reduced expression} for $\omega$.
It is useful to abbreviate $g_{s_{j}}$ by $g_{j}$.
Let $R = \Z[q, q^{-1}]$, and let n be a natural number. We use the term $H_{n}$ to indicate $H_{R, q}(S_{n})$.
Denote by $H_{2k+1,n}$ the subalgebra of the Hecke algebra which is generated by elements $g_{2k+1}, g_{2k+2},..., g_{n-1} \ in \ H_{n}$.
As a free R-module, $H_{2k+1,n}$ has an $R$-basis \hspace{2cm} $\{ g_{\omega} |\ \omega \in S_{2k+1,n}\}$,
where $S_{2k+1,n}$ generated by elements $s_{2k+1}, s_{2k+2},..., s_{n-1}$ is a subgroup of the symmetric group $S_{n}.$
In the next lemma we collect some basic facts on $H_{n}$.

\begin{lem} \label{bd4}
(a) If $\omega,\ \omega' \in S_{n}$ and $l(\omega \omega') = l(\omega) + l(\omega')$, then $g_{\omega}g_{\omega'}= g_{\omega \omega'}.$

(b) Let $s_{j}$ be a transposition and $\omega \in S_{n}$, then
\[g_{j}g_{\omega} =
\begin{cases}
g_{s_{j}\omega } &\text{ if \ } l(s_{j}\omega ) = l(\omega) + 1 \\
(q-1)g_{\omega } + qg_{s_{j}\omega } &\text{otherwise, }
\end{cases} \]
and
\[g_{\omega}g_{j} =
\begin{cases}
g_{\omega s_{j} } &\text{ if \ } l(\omega s_{j} ) = l(\omega) + 1 \\
(q-1)g_{\omega } + qg_{\omega s_{j}} &\text{otherwise. }
\end{cases} \]

(c) Let $\omega \in S_{n}$. Then $g_{\omega}$ is invertible in $H_{n}$
with inverse $g^{-1}_{\omega} = g^{-1}_{j}g^{-1}_{j-1}...g^{-1}_{2}g^{-1}_{1}$,
where $\omega = s_{1}s_{2}...s_{j}$ is a reduced expression for $\omega$, and
$$g_{j}^{-1} = q^{-1}g_{j} + (q^{-1} -1), \ \ so \ \ g_{j} = qg_{j}^{-1} + (q - 1) \text{ for all $s_{j}$}.$$
\end{lem}

In the literature, the Hecke algebras $H_{n}$ of type $A_{n -1}$ are commonly defined by generators
$g_{i}$, $1 \leq i < n$ and relations
\begin{align*}
&(H_{1}) \hspace{3cm} g_{i}g_{i + 1}g_{i} = g_{i + 1}g_{i}g_{i + 1} && \text{ for $1 \leq i \le n-1$};\\
&(H_{2})  \hspace{3cm} g_{i}g_{j} = g_{j}g_{i} && \text{ for  $|i - j| > 1$}.
\end{align*}
The next statement is implicit in \cite{DJ1} (or see Lemma 2.3 of \cite{M}).

\begin{lem} \label{bd5}
The $R$-linear map $i$: $H_{n}\ \lra \ H_{n}$ determined by $i(g_{\omega}) = g_{\omega^{-1}}$
for each $\omega \in S_{n}$ is an involution on $H_{n}$
\end{lem}

\subsection{Cellular algebras}
In this section we recall the original definition of cellular algebras in the sense of Graham and Lehrer in \cite{GL} and an equivalent definition given in \cite{KX1} by Koenig and Xi.
Then we are going to apply the idea and approach in (\cite{KX4, X2}) to cellular structure
to prove that the q-Brauer algebra $Br_{n}(r, q)$ is cellular.
In fact, to obtain this result we firstly construct a basis for $Br_{n}(r, q)$ that is naturally indexed by the basis of Brauer algebra $D_{n}(N)$.
Subsequently, the q-Brauer algebra is shown to be an iterated inflation of Hecke algebras of type $A_{n-1}$. More details about the concept ''iterated inflation'' were given in \cite{KX2, KX4}. Using the result in this section, we can determine the quasi-hereditary q-Brauer algebras in Section 5.

\begin{defn}(Graham and Lehrer, \cite{GL}). Let $R$ be a commutative Noetherian integral domain with identity.
A \textit{cellular algebra} over $R$ is an associative (unital) algebra $A$
together with \textit{cell datum} $(\varLambda, M, C, i)$, where
\begin{itemize}
\item[(C1)] $\varLambda$ is a partially ordered set (poset) and for each $\lambda \in \varLambda$,
$M(\lambda)$ is a finite set such that the algebra $A$ has an $R$-basis $C^{\lambda}_{S, T}$,
where $(S, T)$ runs through all elements of $M(\lambda) \times M(\lambda)$ for all $\lambda \in \varLambda.$
\item[(C2)] If $\lambda \in \varLambda$ and $S, T \in M(\lambda).$
  Then $i$ is involution of $A$ such that $i(C^{\lambda}_{S, T}) = C^{\lambda}_{T, S}.$
\item[(C3)] For each $\lambda \in \varLambda$ and $S, T \in M(\lambda)$ then for any element $a \in A$ we have
  $$aC^{\lambda}_{S, T} \equiv \sum_{U \in M(\lambda)} r_{a}(U, S)C^{\lambda}_{U, T} \ (mod \ A(< \lambda)),$$
  where $r_{a}(U, S) \in R$ is independent of T, and $A(< \lambda)$ is
  the $R$-submodule of $A$ generated by $\{C^{\mu}_{S^{'}, T^{'}} | \mu < \lambda;\ S^{'}, T^{'} \in M(\mu)\}.$
\end{itemize}
\end{defn}
The basis $\{C^{\lambda}_{S, T}\}$ of a cellular algebra $A$ is called as a \textit{cell basis}.
In \cite{GL}, Graham and Lehrer defined a bilinear form $\phi_{\lambda}$ for each $\lambda \in \varLambda$ with respect to this basis as follows.
$$C^{\lambda}_{S, T}C^{\lambda}_{U, V} \equiv \phi_{\lambda}(T, U)C^{\lambda}_{S, V} \ (mod \ A < \lambda).$$
They also proved that the isomorphism classes of simple modules are parametrized by
the set $$\varLambda_{0} = \{ \lambda \in \varLambda |\ \phi_{\lambda} \neq 0\}.$$
The following is an equivalent definition of cellular algebra.

\begin{defn}(Koenig and Xi, \cite{KX1}).
Let $A$ be an $R$-algebra where $R$ is a commutative noetherian integral domain.
Assume there is an involution $i$ on $A$, a two-sided ideal $J$ in $A$ is called \textit{ cell ideal} \iff
$i(J) = J$ and there exists a left ideal $\Delta \subset J$ such that $\Delta$ is finitely generated and free over $R$ and
such that there is an isomorphism of $A$-bimodules $\alpha \ : \ J \simeq \Delta \otimes_{R}i(\Delta)$
(where $i(\Delta) \subset J$ is the $i$-image of $\Delta$) making the following diagram commutative:

\[
\begin{xy}
\xymatrix
{
J \ar[r]^{ \alpha} \ar[d]_{i}
& \Delta \otimes_{R}i(\Delta) \ar[d]^{x\otimes y \mapsto i(y)\otimes i(x)}
\\
J \ar[r]^{\alpha}
& \Delta \otimes_{R}i(\Delta)
\\
}
\end{xy}
\]

The algebra $A$ with the involution $i$ is called \textit{cellular} \iff there is an \hspace{2cm} $R$-~module decomposition
$A = J^{'}_{1} \oplus J^{'}_{2} \oplus ... J^{'}_{n}$ (for some n) with $i(J^{'}_{j}) = J^{'}_{j}$ for each j
and such that setting $J_{j} = \oplus_{l = 1}^{j} J^{'}_{l}$ gives a chain of two-sided ideals of $A$:
$0 = J_{0} \subset  J_{1} \subset J_{2} \subset ... \subset J_{n}= A$ (each of them fixed by $i$) and
for each j ($j = 1, ..., n$) the quotient $J^{'}_{j} = J_{j}/J_{j-1}$ is a cell ideal (with respect to the involution induced by $i$ on the quotient) of $A/J_{j-1}.$
\end{defn}

Recall that an involution $i$ is defined as an $R$-linear anti-automorphism of $A$ with $i^{2} = id.$
The $\Delta^{'}s$ obtained from each section $J_{j}/J_{j-1}$ are called $cell \ modules$ of the cellular algebra $A$.
Note that all simple modules are obtained from cell modules \cite{GL}.

In \cite{KX1}, Koenig and Xi proved that the two definitions of cellular algebra are equivalent.
The first definition can be used to check concrete examples, the latter, however, is convenient to look at the structure of cellular algebras as well as to check cellularity of an algebra.

Typical examples of cellular algebras are the following:
Group algebras of symmetric groups, Hecke algebras of type $A_{n-1}$
or even of Ariki-Koike finite type \cite{G} (i.e., cyclotomic Hecke algebras), Schur algebras of type A,
Brauer algebras, Temperley-Lieb and Jones algebras which are subalgebras of the Brauer algebra \cite{GL},
partition algebras \cite{X1}, BMW-algebras \cite{X2}, and recently Hecke algebras of finite type \cite{G}.

The following results are shown in \cite{GL}; see also \cite{KX4}.
\begin{thm}  The Brauer algebra $D_{n}(x)$ is cellular for any commutative noetherian integral domain with identity $R$ and parameter $x \in R$.
\end{thm}

\begin{thm}\label{dl1}
Let $R = \Z[q, q^{-1}]$. Then $R$-algebra $H_{n}$ is a cellular algebra.
\end{thm}

Next we will prove that q-Brauer algebras are also cellular.
Before doing this, let us introduce what the q-Brauer algebra is and construct a particular basis in the following section.

\section{ $q$-Brauer algebras}
Firstly, we recall the definition of the generic $q$-Brauer algebras due to Wenzl \cite{W3}.
Then we collect and extend basic properties of $q$-Brauer algebras that will be needed in Section 4.
The main result of this section is in Subsection 3.2 where a particular basis for $q$-Brauer algebras is given in Theorem \ref{dl11}.

\begin{defn} Fix $N \in \Z \setminus \{0\}$, let $q$ and $r$ be invertible elements. Moreover, assume that if $q = 1$ then $r= q^N$.
The $q$-Brauer algebra $Br_{n}(r, q)$ is defined over the ring
$\mathbb{Z}[q^{\pm1}, r^{\pm1}, ((r-1)/(q-1))^{\pm 1}]$ by generators $g_{1}$, $g_{2}$, $g_{3}$, ..., $g_{n-1}$ and $e$ and relations
\begin{enumerate}
\item[(H)] The elements $g_{1}$, $g_{2}$, $g_{3}$, ..., $g_{n-1}$  satisfy the relations of the Hecke algebra $H_{n}$;

\item[$(E_{1})$] $e^{2} = \dfrac{r-1} {q-1}e$;

\item[$(E_{2})$]  $eg_{i} = g_{i}e$ for $i > 2,$ $eg_{1} = g_{1}e = qe$, $eg_{2}e = re$ and $eg^{-1}_{2}e = q^{-1}e$;

\item[$(E_{3})$] $e_{(2)} = g_{2}g_{3}g^{-1}_{1}g^{-1}_{2}e_{(2)} = e_{(2)}g_{2}g_{3}g^{-1}_{1}g^{-1}_{2}$,
where $ e_{(2)} = e(g_{2}g_{3}g^{-1}_{1}g^{-1}_{2})e.$
\end{enumerate}
\end{defn}

For $1 \le l,  k \le n$, let
\[ g^{+}_{l, k} =
\begin{cases}
&g_{l}g_{l+1}...g_{k} \hspace{1.5cm} \text{ if $ l \le k $}, \\
&g_{l}g_{l-1}...g_{k} \hspace{1.5cm} \text{ if $ l > k $},
\end{cases}
\]
and
\[ g^{-}_{l, k} =
\begin{cases}
&g^{-1}_{l}g^{-1}_{l+1}...g^{-1}_{k} \hspace{1.5cm} \text{ if $l \le k $}, \\
&g^{-1}_{l}g^{-1}_{l-1}...g^{-1}_{k} \hspace{1.5cm} \text{ if $ l > k $}.
\end{cases}
\]
Now the elements $e_{(k)}$ in $Br_{n}(r, q)$ are defined inductively by $e_{(1)} =e$
and by
\begin{align} \label{eq6}
e_{(k+1)} = eg^{+}_{2, 2k+1}g^{-}_{1, 2k}e_{(k)}.
\end{align}
Note that from the above definition we also obtain the equalities
\begin{align} \label{eq7}
eg^{-1}_{1} = g^{-1}_{1}e = q^{-1}e.
\end{align}

\begin{rem} \label{remofqBr}
1. Another version of the $q$-Brauer algebra, due to Wenzl (see \cite{W3}, Definition 3.1), is defined as follows:
The $q$-Brauer algebra $Br_n(N)$ is defined over the ring $\Z[q,q^{-1}]$ with the same generators
as before, with relations (H) and $(E_3)$ unchanged, and with
\begin{enumerate}
\item[$(E1)'$] $e2=[N]e$ where $[N] = 1 + q^{1} + \cdots + q^{N-1}$;
\item[$(E2)'$] $eg_i=g_ie$ for $i>2$, $eg_1=g_1e=qe$, $eg_2e=q^Ne$ and
$eg_2^{-1}e=q^{-1}e$.
\end{enumerate}
\smallskip
Notice that the parameter $[N]$ above is different from that in Wenzl's definition.
In \cite{W3}, Wenzl sets $[N]= (q^{N}-1)/(q-1)$ for $N \in \Z \setminus \{0\}$.
This setting implies that his $q$-Brauer algebra $Br_n(N)$ gets back the Brauer algebra $D_n(N)$ only with respect to the ground rings which admit the limit $q \ra 1$, such as the field of real or complex numbers.
The new parameter $[N]$ enables us to study structure and representation theory of the $q$-Brauer algebra $Br_n(N)$ over an arbitrary field of any characteristic, as well as to compare it with the Brauer algebra even in the case $q=1$.
More precisely, our choice of the new parameter does not affect the properties of the $q$-Brauer algebra $Br_n(N)$, which were studied in detail by Wenzl.
Moreover, all results on $Br_{n}(r, q)$ in the subsequent sections hold true for $Br_n(N)$ as well.

2. It is obvious that if $q=1$ the $q$-Brauer algebra $Br_n(N)$ recovers the classical Brauer algebra $D_n(N)$. In this case $g_i$ becomes the simple reflection $s_i$
and the element $e_{(k)}$ can be identified with the diagram $e_{(k)}$.
Similarly, over the ground field of real or complex numbers the generic $q$-Brauer algebra $Br_{n}(r, q)$ coincides with the Brauer algebra $D_n(N)$ when fixing $r=q^{N}$ in the limit $q \rightarrow 1$ (see \cite{W3}, Remark 3.1).

3. A closely related algebra has appeared in Molev's work.
In 2003, Molev \cite{Mo1} introduced a new $q$-analogue of the Brauer algebra by considering the centralizer of the natural action in tensors of the nonstandard deformation of the universal enveloping algebra $U(o_N)$.
He defined relations for these algebras and constructed representations of them on tensor spaces.
However, in general these representations are not faithful,
and little is known about these abstract algebras besides these representations.
The $q$-Brauer algebra, a closely related algebra with that of Molev, was introduced later by Wenzl \cite{W3} via generators and relations who gave a detailed description of its structure. In particular, he proved that generically it is semisimple and isomorphic to the Brauer algebra.
It can be checked that the representations of Molev's algebras in \cite{Mo1} are also representations of the $q$-Brauer algebras (see e.g. \cite{W5}, Section 2.2) and that the relations written down by Molev are also satisfied by the generators of the $q$-Brauer algebras;
but potentially, Molev's abstractly defined algebras could be larger (see \cite{Mo2}).
\end{rem}

The next lemmas will indicate how the Brauer algebra relations extend to the $q$-Brauer algebra.

\begin{lem}\label{bd6}(\cite{W3}, Lemma 3.3)
(a) The elements $e_{(k)}$ are well-defined.

(b) $g^{+}_{1, 2l}e_{(k)} = g^{+}_{2l+1, 2}e_{(k)}$ and
  $g^{-}_{1, 2l}e_{(k)} = g^{-}_{2l+1, 2}e_{(k)}$ for $l < k.$

(c) $g_{2j-1}g_{2j}e_{(k)} = g_{2j+1}g_{2j}e_{(k)}$
  and $g^{-1}_{2j-1}g^{-1}_{2j}e_{(k)} = g^{-1}_{2j+1}g^{-1}_{2j}e_{(k)}$ for $1 \leq j < k.$

(d) For any $j \leq k$ we have $e_{(j)}e_{(k)} = e_{(k)}e_{(j)} = (\dfrac{r-1} {q-1})^{j}e_{(k)}$.
\item[(e)] $(\dfrac{r-1} {q-1})^{j-1}e_{(k+1)} = e_{(j)}g^{+}_{2j, 2k+1}g^{-}_{2j-1, 2k}e_{(k)}$ for $1 \leq j < k$.

(f) $e_{(j)}g_{2j}e_{(k)} = r(\dfrac{r-1} {q-1})^{j-1}e_{(k)}$ for $1 \leq j \leq k.$
\end{lem}

\begin{lem} \label{bd7}
(a) $g_{2j+1}g^{+}_{2, 2k+1}g^{-}_{1, 2k} = g^{+}_{2, 2k+1}g^{-}_{1, 2k}g_{2j-1}$,
and $g^{-}_{2k, 1}g^{+}_{2k+1, 2}g_{2j+1} = g_{2j-1}g^{-}_{2k, 1}g^{+}_{2k+1, 2}$  for $1 \leq j \leq k$.

(b)  $g_{2j+1} e_{(k)} = e_{(k)}g_{2j+1} = qe_{(k)}$, and
  $g^{-1}_{2j+1} e_{(k)} = e_{(k)}g^{-1}_{2j+1} = q^{-1}e_{(k)}$
for $0 \leq j < k$.

(c) $e_{(k+1)} = e_{(k)}g^{-}_{2k, 1}g^{+}_{2k+1, 2}e$.
\end{lem}

\begin{proof}
(a) Let us to prove the first equality, the other one is similar.
\begin{align*}
g_{2j+1}g^{+}_{2, 2k+1}g^{-}_{1, 2k}& \overset{(H_{2})}{=} g^{+}_{2, 2j-1}(g_{2j+1}g_{2j}g_{2j+1})g^{+}_{2j+2, 2k+1}g^{-}_{1, 2k}\\
& \overset{(H_{1})}{=} g^{+}_{2, 2j-1}(g_{2j}g_{2j+1}g_{2j})g^{+}_{2j+2, 2k+1}g^{-}_{1, 2k}\\
& \overset{}{=} g^{+}_{2, 2j+1}g_{2j}g^{+}_{2j+2, 2k+1}g^{-}_{1, 2k}\\
& \overset{(H_{2})}{=} g^{+}_{2, 2k+1}g_{2j}g^{-}_{1, 2k}
\overset{(H_{2})}{=} g^{+}_{2, 2k+1}g^{-}_{1, 2j-2}g_{2j}g^{-}_{2j-1, 2k}\\
& \overset{L\ref{bd4}(c)}{=} g^{+}_{2, 2k+1}g^{-}_{1, 2j-2}[(q-1) + qg^{-1}_{2j}]g^{-}_{2j-1, 2k}\\
& \overset{}{=} (q-1)g^{+}_{2, 2k+1}g^{-}_{1, 2k} + qg^{+}_{2, 2k+1}g^{-}_{1, 2j-2} (g^{-1}_{2j}g^{-1}_{2j-1}g^{-1}_{2j}) g^{-}_{2j+1, 2k}\\
& \overset{(H_{1})}{=} (q-1)g^{+}_{2, 2k+1}g^{-}_{1, 2k} + qg^{+}_{2, 2k+1}g^{-}_{1, 2j-2} (g^{-1}_{2j-1}g^{-1}_{2j}g^{-1}_{2j-1}) g^{-}_{2j+1, 2k}\\
& \overset{}{=} (q-1)g^{+}_{2, 2k+1}g^{-}_{1, 2k} + qg^{+}_{2, 2k+1}g^{-}_{1, 2j} g^{-1}_{2j-1} g^{-}_{2j+1, 2k}\\
& \overset{(H_{2})}{=} (q-1)g^{+}_{2, 2k+1}g^{-}_{1, 2k} + qg^{+}_{2, 2k+1}g^{-}_{1, 2k} g^{-1}_{2j-1}\\
& \overset{L\ref{bd4}(c)}{=} (q-1)g^{+}_{2, 2k+1}g^{-}_{1, 2k} + qg^{+}_{2, 2k+1}g^{-}_{1, 2k} [q^{-1}g_{2j-1} + (q^{-1}-1)]\\
& \overset{}{=} g^{+}_{2, 2k+1}g^{-}_{1, 2k}g_{2j-1}.
\end{align*}
Notice that when $j=k$ then $g^{+}_{2j+2, 2k+1} = g^{-}_{2j+1, 2k}= 1$, where 1 is the identity element in $Br_{n}(r, q).$

(b)  To prove (b), we begin by showing the equality
\begin{align} \label{eq8}
g_{2j+1} e_{(k)} = qe_{(k)}\ with\ 0 \leq j < k.
\end{align}
Then the equality
\begin{align} \label{eq9}
g^{-1}_{2j+1} e_{(k)} = q^{-1}e_{(k)}\ with\ 0 \leq j < k
\end{align}
comes as a consequence.

The other equalities will be shown simultaneously with proving (c).
The equality (\ref{eq8}) is shown by induction on $k$ as follows:

For $ k=1$, the claim follows from $(E_{2}).$
Now suppose that the equality (\ref{eq8}) holds for $k-1$, that is,
$$g_{2j+1}e_{(k-1)} = qe_{(k-1)} \text{ for $j < k -1$}.$$
Then with $j < k$
\begin{align*}
g_{2j+1}e_{(k)} \overset{(\ref{eq6})}{=} g_{2j+1}eg^{+}_{2, 2k-1}g^{-}_{1, 2k-2}e_{(k-1)}
\overset{(a)\ for \ j< k}{=} eg^{+}_{2, 2k-1}g^{-}_{1, 2k-2}(g_{2j-1}e_{(k-1)}) \\
\overset{(a)}{=} qeg^{+}_{2, 2k-1}g^{-}_{1, 2k-2}e_{(k-1)}\overset{(\ref{eq6})}{=} qe_{(k)}
\end{align*}
by the induction assumption.
The equality (\ref{eq9}) is obtained immediately by multiplying the equality (\ref{eq8}) by $g_{2j+1}^{-1}$ on the left.

The following equalities
\begin{align}
\label{eq10} e_{(k)}g_{2j+1} = qe_{(k)} \hspace{0.2cm}
\end{align}
and
\begin{align} \label{eq11}
e_{(k)}g^{-1}_{2j+1} = q^{-1}e_{(k)} \ \ \  for  \ \ 0 \leq j < k,
\end{align}
are proven by induction on $k$ in a combination with (c) in the following way:
If (c) holds for $k-1$ then the equalities (\ref{eq10}) and (\ref{eq11}) are proven to hold for $k$.
This result implies that (c) holds for $k$, and hence, the equalities (\ref{eq10}) and (\ref{eq11}), again, are true for $k+1$.
Proceeding in this way, all relations (c), (\ref{eq10}) and (\ref{eq11}) are obtained.
Indeed, when $k=1$ (c) follows from direct calculation:
\begin{align} \label{eq12}
e_{(2)} &\overset{(E_{3})}{=} e(g_{2}g_{3}g^{-1}_{1}g^{-1}_{2})e  \overset{L\ref{bd4}(c)}{=} e((q-1) + qg^{-1}_{2})g_{3}g^{-1}_{1}(q^{-1}g_{2} - q^{-1}(q-1))e\\
&\overset{}{=} q^{-1}(q - 1)eg_{3}g^{-1}_{1}g_{2}e + eg^{-1}_{2}g_{3}g^{-1}_{1}g_{2}e - q^{-1}(q -1)^{2}eg_{3}(g^{-1}_{1}e) - (q - 1)eg^{-1}_{2}g_{3}(g^{-1}_{1}e) \notag \\
& \overset{(E_{2})}{=} eg^{-1}_{2}g_{3}g^{-1}_{1}g_{2}e + q^{-1}(q - 1)g_{3}(eg^{-1}_{1})g_{2}e - q^{-1}(q -1)^{2}eg_{3}(g^{-1}_{1}e) - (q - 1)eg^{-1}_{2}g_{3}(g^{-1}_{1}e)\notag \\
& \overset{\eqref{eq7}}{=} eg^{-1}_{2}g_{3}g^{-1}_{1}g_{2}e + q^{-2}(q - 1)g_{3}(eg_{2}e) - q^{-2}(q -1)^{2}eg_{3}e - q^{-1}(q - 1)eg^{-1}_{2}g_{3}e\notag \\
& \overset{(E_{2})}{=} eg^{-1}_{2}g_{3}g^{-1}_{1}g_{2}e + q^{-2}(q - 1)rg_{3}e - q^{-2}(q -1)^{2}e^{2}g_{3} - q^{-1}(q - 1)eg^{-1}_{2}eg_{3}\notag \\
& \overset{(E_{1}),\ (E_{2})}{=} eg^{-1}_{2}g_{3}g^{-1}_{1}g_{2}e + q^{-2}(q - 1)rg_{3}e - q^{-2}(q -1)(r-1)eg_{3} - q^{-2}(q - 1)eg_{3}\notag \\
& = eg^{-1}_{2}g_{3}g^{-1}_{1}g_{2}e \overset{(H_{2})}{=} eg^{-1}_{2}g^{-1}_{1}g_{3}g_{2}e.\notag
\end{align}
The above equality implies (\ref{eq10}) for $k=2$ and $j < 2$ in the following way:
\begin{align*}
&e_{(2)}g_{1} \overset{(E_{3})}{=} eg^{+}_{2, 3}g^{-}_{1, 2}(eg_{1}) \overset{(E_{2}), (E_{3})}{=} qe_{(2)};\\
&e_{(2)}g_{3} \overset{(\ref{eq12})}{=} (eg^{-1}_{2}g^{-1}_{1}g_{3}g_{ 2}e)g_{3} \overset{(E_{2})}{=} e(g^{-1}_{2}g^{-1}_{1}g_{3}g_{ 2})g_{3}e\\
&\overset{(a)\ for\ k=1}{=} (eg_{1})g^{-1}_{2}g^{-1}_{1}g_{3}g_{ 2}e \overset{(E_{2})}{=} qeg^{-1}_{2}g^{-1}_{1}g_{3}g_{ 2}e
\overset{(\ref{eq12})}{=} qe_{(2)}.
\end{align*}
Therefore, in this case the equality (\ref{eq11}) follows from multiplying the equality (\ref{eq10}) with $g_{2j+1}^{-1}$ on the right.
As a consequence, (c) is shown to be true for $k=2$ by following calculation. \\
$ e_{(2)}g^{-}_{4, 1}g^{+}_{5, 2}e \overset{(E_{3})}{=} \ (eg^{+}_{2, 3}g^{-}_{1, 2}e)g^{-}_{4, 1}g^{+}_{5, 2}e
\overset{(H_{2})}{=} (eg^{+}_{2, 3}g^{-}_{1, 2}e)(g^{-}_{4, 3}g^{+}_{5, 4})g^{-}_{2, 1}g^{+}_{3, 2}e \\
\overset{(E_{2})}{=} eg^{+}_{2, 3}g^{-}_{1, 2}(g^{-}_{4, 3}g^{+}_{5, 4})(eg^{-}_{2, 1}g^{+}_{3, 2}e)
\overset{(H_{2}), (E_{3})}{=} eg^{+}_{2, 3}g^{-1}_{4}g_{5}g^{-1}_{1, 3}g_{ 4}e_{(2)}\\
\overset{L\ref{bd4}(c)}{=}  eg^{+}_{2, 3}[q^{-1}g_{4} + (q^{-1}-1)] g_{5}g^{-}_{1, 3}[(q-1) + qg^{-1}_{ 4}]e_{(2)} \\
\overset{}{=} q^{-1}(q-1) eg^{+}_{2, 5}g^{-}_{1, 3} e_{(2)} + eg^{+}_{2, 5}g^{-}_{1, 4}e_{(2)} - q^{-1}(q-1)^{2}e g^{+}_{2, 3}g_{5}g^{-}_{1, 3}e_{(2)} - (q-1)e g^{+}_{2, 3}g_{5}g^{-}_{1, 4}e_{(2)} \\
\overset{(\ref{eq6})}{=} q^{-1}(q-1) eg^{+}_{2, 5}g^{-}_{1, 3} e_{(2)} + e_{(3)} - q^{-1}(q-1)^{2}e g^{+}_{2, 3}g_{5}g^{-}_{1, 3}e_{(2)} - (q-1)e g^{+}_{2, 3}g_{5}g^{-}_{1, 4}e_{(2)}.$\\
Subsequently, it remains to prove that
$$ q^{-1}(q-1) eg^{+}_{2, 5}g^{-}_{1, 3} e_{(2)} - q^{-1}(q-1)^{2}e g^{+}_{2, 3}g_{5}g^{-}_{1, 3}e_{(2)} - (q-1)e g^{+}_{2, 3}g_{5}g^{-}_{1, 4}e_{(2)} \ = \ 0.$$
To this end, considering separately each summand in the left hand side of the last equality, it yields
\begin{align} \label{eq13}
\dfrac{q-1}{q} eg^{+}_{2, 5}g^{-}_{1, 3} e_{(2)} &\overset{(H_{2})}{=} \dfrac{q-1} {q}  eg^{+}_{2, 3}g^{-}_{1, 2}g^{+}_{4, 5} (g^{-1}_{3}e_{(2)})\\
&\overset{(\ref{eq9})\ for \ k=2}{=}  \dfrac{q-1} {q^{2}} eg^{+}_{2, 3}g^{-}_{1, 2}g^{+}_{4, 5}e_{(2)}
\overset{L\ref{bd6}(d)}{=} \dfrac {(q-1)^{2}} {q^{2}(r-1)}(eg^{+}_{2, 3}g^{-}_{1, 2})g^{+}_{4, 5}ee_{(2)}\notag\\
&\overset{(E_{2})}{=} \dfrac {(q-1)^{2}} {q^{2}(r-1)}(eg^{+}_{2, 3}g^{-}_{1, 2}e)g^{+}_{4, 5}e_{(2)}
\overset{(E_{3})}{=}\dfrac {(q-1)^{2}} {q^{2}(r-1)}e_{(2)}g^{+}_{4, 5}e_{(2)}\notag \\
&\overset{(E_{2}),\ (H_{2})}{=}\dfrac {(q-1)^{2}} {q^{2}(r-1)}e_{(2)}g_{4}e_{(2)}g_{5}
\overset{L\ref{bd6}(f)}{=} \dfrac {r(q-1)^{3}} {q^{2}(r-1)^{2}} e_{(2)}g_{ 5}.\notag
\end{align}

\begin{align} \label{eq14}
\dfrac{(q-1)^{2}}{q}e g^{+}_{2, 3}g_{5}&g^{-}_{1, 3}e_{(2)} \overset{(E_{2}),\ (H_{2})}{=} \dfrac {(q-1)^{2}} {q} e g^{+}_{2, 3}g^{-}_{1, 2}(g^{-1}_{3}e_{(2)})g_{5}\\
&\overset{(\ref{eq8})\ for \ k=2}{=} \dfrac {(q-1)^{2}} {q^{2}} e g^{+}_{2, 3}g^{-}_{1, 2}e_{(2)}g_{5}
\overset{L\ref{bd6}(d)}{=} \dfrac {(q-1)^{3}} {q^{2}(r-1)} (e g^{+}_{2, 3}g^{-}_{1, 2}e)e_{(2)}g_{5} \notag \\
&= \dfrac{(q-1)^{3}} {q^{2}(r-1)}e_{(2)} e_{(2)}g_{5}
\overset{L\ref{bd6}(d)}{=}  \dfrac {(q-1)(r-1)} {q^{2}}e_{(2)}g_{5}.\notag
\end{align}

\begin{align} \label{eq15}
(q-1)e g^{+}_{2, 3}g_{5}g^{-}_{1, 4}e_{(2)}& \overset{(E_{2}),\ (H_{2})}{=} (q-1)g_{5}e g^{+}_{2, 3}g^{-}_{1, 2}g^{-}_{3, 4}e_{(2)}\\
&\overset{L\ref{bd6}(d)}{=} \dfrac {(q-1)^{2}} {r-1} g_{5}(e g^{+}_{2, 3}g^{-}_{1, 2})g^{-}_{3, 4}(ee_{(2)})\notag\\
&\overset{(E_{2})}{=} \dfrac {(q-1)^{2}} {r-1} g_{5}(e g^{+}_{2, 3}g^{-}_{1, 2}e)g^{-}_{3, 4}e_{(2)}
\overset{(E_{3})}{=} \dfrac {(q-1)^{2}} {r-1} g_{5}(e_{ (2)}g^{-1}_{3})g^{-1}_{ 4}e_{(2)}\notag \\
&\overset{(\ref{eq11}) for \ k=2}{=}\dfrac {(q-1)^{2}} {q(r-1)} g_{5}e_{ (2)}g^{-1}_{ 4}e_{(2)}
\overset{L\ref{bd4}(c)}{=} \dfrac {(q-1)^{2}} {q(r-1)} g_{5}e_{ (2)}[q^{-1}g_{ 4} + (q^{-1}-1)] e_{(2)}\notag \\
&\overset{}{=} \dfrac {(q-1)^{2}} {q^{2}(r-1)} g_{5}e_{ (2)}g_{ 4}e_{(2)}  -  \dfrac {(q-1)^{3}} {q^{2}(r-1)} g_{5}e_{ (2)}e_{(2)}  \notag \\
&\overset{(E_{2}),\ (H_{2}),\ L\ref{bd6}(d), (f)}{=}  \dfrac {r(q-1)^{3}} {q^{2}(r-1)^{2}} e_{(2)}g_{ 5}  -  \dfrac {(q-1)(r-1)} {q^{2}} e_{(2)}g_{5}. \notag
\end{align}
By (\ref{eq13}), (\ref{eq14}) and (\ref{eq15}), it implies the equation (c) for $k =2$.
Now suppose that the relations (\ref{eq10}) and (\ref{eq11}) hold for $k$. We will show that (c) holds for $k$, and as a consequence both (\ref{eq10}) and (\ref{eq11}) hold with $k+1$.
Indeed, we have:

\begin{align*}
e_{(k)}g^{-}_{2k, 1}g^{+}_{2k+1, 2}e &\overset{(\ref{eq6})}{=} (eg^{+}_{2, 2k - 1}g^{-}_{1, 2k - 2} e_{(k - 1)})(g^{-1}_{2k}g^{-1}_{2k - 1}g^{-}_{2k-2, 1})(g_{2k + 1}g_{2k}g^{+}_{2k - 1, 2}e) \\
&\overset{(H_{2})}{=} (eg^{+}_{2, 2k - 1}g^{-}_{1, 2k - 2} e_{(k - 1)})(g^{-1}_{2k}g_{2k + 1})(g^{-1}_{2k - 1}g_{2k})g^{-}_{2k-2, 1}g^{+}_{2k - 1, 2}e \\
&\overset{(E_{2}),\ (H_{2})}{=} eg^{+}_{2, 2k - 1}g^{-}_{1, 2k - 2}(g^{-1}_{2k}g_{2k + 1})(g^{-1}_{2k - 1}g_{2k})e_{(k - 1)}g^{-}_{2k-2, 1}g^{+}_{2k - 1, 2}e \\
&\overset{(H_{2})}{=} eg^{+}_{2, 2k - 1}(g^{-1}_{2k}g_{2k + 1})g^{-}_{1, 2k - 2}(g^{-1}_{2k - 1}g_{2k})(e_{(k - 1)}g^{-}_{2k-2, 1}g^{+}_{2k - 1, 2}e).
\end{align*}
By induction assumption, the last formula is equal to
$$eg^{+}_{2, 2k - 1}[q^{-1}g_{2k}\ +\ (q^{-1} - 1)]g_{2k + 1}g^{-}_{1, 2k - 1}[(q - 1)\ +\ qg^{-1}_{2k}]e_{(k)},$$
and direct calculation implies
\begin{align*}
eg^{+}_{2, 2k + 1}g^{-}_{1, 2k}e_{(k)} &+ q^{-1}(q - 1) eg^{+}_{2, 2k + 1}g^{-}_{1, 2k - 1}e_{(k)}\\
&- q^{-1}(q - 1)^{2}eg^{+}_{2, 2k - 1}g_{2k+ 1}g^{-}_{1, 2k - 1}e_{(k)} - (q - 1)eg^{+}_{2, 2k - 1}g_{2k+ 1}g^{-}_{1, 2k}e_{(k)} \\
\overset{(\ref{eq6})}{=} e_{(k + 1)} &+ q^{-1}(q - 1) eg^{+}_{2, 2k + 1}g^{-}_{1, 2k - 1}e_{(k)}\\
&- q^{-1}(q - 1)^{2}eg^{+}_{2, 2k - 1}g_{2k+ 1}g^{-}_{1, 2k - 1}e_{(k)} - (q - 1)eg^{+}_{2, 2k - 1}g_{2k+ 1}g^{-}_{1, 2k}e_{(k)}.
\end{align*}

Applying the same arguments as in the case k = 3, each separate summand in the above formula can be computed as follows:
\begin{align} \label{eq16}
\dfrac{q - 1}{q} eg^{+}_{2, 2k + 1}&g^{-}_{1, 2k - 1}e_{(k)} \overset{}{=} \dfrac {q-1} {q} eg^{+}_{2, 2k + 1}g^{-}_{1, 2k - 1}(g^{-1}_{2k - 1}e_{(k)})\\
&\overset{(\ref{eq9}) \ for\ k}{=} \dfrac {q-1} {q^{2}} eg^{+}_{2, 2k + 1}g^{-}_{1, 2k - 2}e_{(k)}
\overset{(H_{2})}{=} \dfrac {q-1} {q^{2}} eg^{+}_{2, 2k - 1}g^{-}_{1, 2k - 2}g^{+}_{2k, 2k +1}e_{(k)} \notag \\
&\overset{L\ref{bd6}(d)}{=} \dfrac {(q-1)^{k}} {q^{2}(r-1)^{k-1}} (eg^{+}_{2, 2k - 1}g^{-}_{1, 2k - 2})g^{+}_{2k, 2k+1}(e_{(k - 1)}e_{(k)})\notag \\
&\overset{(H_{2}),\ (E_{2})}{=} \dfrac {(q-1)^{k}} {q^{2}(r-1)^{k-1}} (eg^{+}_{2, 2k - 1}g^{-}_{1, 2k - 2}e_{(k - 1)})g_{2k}e_{(k)}g_{2k + 1}\notag \\
&\overset{(\ref{eq6})}{=} \dfrac {(q-1)^{k}} {q^{2}(r-1)^{k-1}} e_{(k)}g_{2k}e_{(k)}g_{2k + 1} \notag
\overset{L\ref{bd6}(f)}{=} \dfrac {r(q-1)^{2k-1}} {q^{2}(r-1)^{2k-2}} e_{(k)}g_{2k + 1}.\notag
\end{align}
\begin{align} \label{eq17}
\dfrac{(q - 1)^{2}}{q}  eg^{+}_{2, 2k - 1}&g_{2k+ 1}g^{-}_{1, 2k - 1}e_{(k)} \overset{(H_{2}),\ (E_{2})}{=} \dfrac {(q-1)^{2}} {q} g_{2k+ 1}eg^{+}_{2, 2k - 1}g^{-}_{1, 2k - 2}(g^{-1}_{2k - 1}e_{(k)})\\
&\overset{(3.3) for\ k}{=} \dfrac {(q-1)^{2}} {q^{2}} g_{2k+ 1}eg^{+}_{2, 2k - 1}g^{-}_{1, 2k - 2}e_{(k)}\notag \\
&\overset{L\ref{bd6}(d)}{=} \dfrac {(q-1)^{k+1}} {q^{2}(r-1)^{k-1}} g_{2k+ 1}(eg^{+}_{2, 2k - 1}g^{-}_{1, 2k - 2}e_{(k - 1)})e_{(k)}\notag \\
&\overset{(\ref{eq6}) }{=} \dfrac {(q-1)^{k+1}} {q^{2}(r-1)^{k-1}} g_{2k+ 1}e_{(k)}e_{(k)}
\overset{L3.2(d),\ (E_{2})}{=} \dfrac {(q-1)(r-1)} {q^{2}} e_{(k)}g_{2k+ 1} \notag
\end{align}
\begin{align} \label{eq18}
(q - 1) & eg^{+}_{2, 2k - 1} g_{2k+ 1}g^{-}_{1, 2k}e_{(k)} \overset{(H_{2}), (E_{2})}{=} (q - 1)g_{2k+ 1}(eg^{+}_{2, 2k - 1}g^{-}_{1, 2k - 2})g^{-}_{2k - 1, 2k}e_{(k)} \\
&\overset{L\ref{bd6}(d)}{=} \dfrac {(q-1)^{k}} {(r-1)^{k-1}} g_{2k+ 1}(eg^{+}_{2, 2k - 1}g^{-}_{1, 2k - 2})g^{-}_{2k - 1, 2k}(e_{(k - 1)}e_{(k)})\notag \\
&\overset{(E_{2}),\ (H_{2})}{=} \dfrac {(q-1)^{k}} {(r-1)^{k-1}} g_{2k+ 1}(eg^{+}_{2, 2k - 1}g^{-}_{1, 2k - 2}e_{(k - 1)})g^{-}_{2k - 1, 2k}e_{(k)}\notag
\end{align}
\begin{align*}
&\overset{(\ref{eq6})}{=} \dfrac {(q-1)^{k}} {(r-1)^{k-1}} g_{2k+ 1}(e_{(k)}g^{-1}_{2k - 1})g^{-1}_{2k}e_{(k)}
\overset{(\ref{eq11})\ for \ k}{=} \dfrac {(q-1)^{k}} {q(r-1)^{k-1}} g_{2k+ 1}e_{(k)}g^{-1}_{2k}e_{(k)}\notag \\
&\overset{L\ref{bd4}(c)}{=}\dfrac {(q-1)^{k}} {q(r-1)^{k-1}} g_{2k+ 1}e_{(k)}[q^{-1}g_{2k} + (q^{-1} - 1)]e_{(k)}\notag \\
& = \dfrac {(q-1)^{k}} {q^{2}(r-1)^{k-1}} g_{2k+ 1}e_{(k)}g_{2k}e_{(k)}\ - \  \dfrac {(q-1)^{k+1}} {q^{2}(r-1)^{k-1}} g_{2k+ 1}(e_{(k)})^{2}\notag \\
&\overset{L\ref{bd6}(d), (f)}{=} \ \dfrac {r(q-1)^{2k-1}} {q^{2}(r-1)^{2k-2}} e_{(k)}g_{2k+ 1}\ - \ \dfrac {(q-1)(r-1)} {q^{2}} e_{(k)}g_{2k+ 1}.\notag
\end{align*}
The above calculations imply that $e_{(k)}g^{-}_{2k, 1}g^{+}_{2k+1, 2}e = e_{(k+1)}.$ Thus (c) holds with value $k$.
Using the relation (a), and the equalities (\ref{eq10}) and (c) for $k$, it yields the equalities (\ref{eq10}) and (\ref{eq11}) for $k + 1$ as follows:

For $j < (k +1)$ then
\begin{align*}
e_{(k +1)}g_{2j + 1} &\overset{(c)\ for \ k}{=} (e_{(k)}g^{-}_{2k, 1}g^{+}_{2k +1, 2} e)g_{2j + 1}\\
&\overset{(E_{2})}{=}  e_{(k)} (g^{-}_{2k, 1}g^{+}_{2k +1, 2}g_{2j +1}) e
\overset{(a) for\ k}{=}  (e_{(k)}g_{2j -1}) g^{-}_{2k, 1}g^{+}_{2k +1, 2} e\\
&\overset{(\ref{eq10})\ for \ k}{=}  q(e_{(k)}g^{-}_{2k, 1}g^{+}_{2k +1, 2} e)
\overset{(c) \ for \ k}{=}  qe_{(k+1)}.
\end{align*}
The relation (\ref{eq11}) for $k+1$ is obtained by multiplying the relation (\ref{eq10}) for $k + 1$ by $g^{-1}_{2j+1}$ on the right.
\end{proof}

\begin{lem} \label{bd8} (\cite{W3}, Lemma 3.4)
We have $e_{(j)}H_{n}e_{(k)} \subset H_{2j+1, n}e_{(k)} \ + \sum_{m \geq k+1} H_{n} e_{(m)}H_{n},$
where $j \leq k$. Moreover, if $j_{1} \geq 2k$ and $j_{2} \geq 2k+1,$ we also have:

(a) $eg^{+}_{2, j_{2}}g^{+}_{1, j_{1}}e_{(k)} = e_{(k+1)}g^{+}_{2k+1, j_{2}}g^{-}_{2k+1, j_{1}},$
if $j_{1} \geq 2k$ and $j_{2} \geq 2k+1$.

(b) $eg^{+}_{2, j_{2}}g^{+}_{1, j_{1}}$ is equal to
\begin{center} $e_{(k+1)}g^{+}_{2k+2, j_{1}}g^{+}_{2k+1, j_{2}} + q^{N+1}(q-1)\sum^{k}_{l=1}q^{2l-2}(g_{2l+1}+1)g^{+}_{2l+2, j_{2}}g^{+}_{2l+1, j_{1}}e_{(k)}.$ \end{center}
\end{lem}
The next lemma is necessary for showing a basis of $Br_{n}(r, q).$

\begin{lem}\label{bd9}(\cite{W3}, Proposition 3.5). The algebra $Br_{n}(r, q)$ is spanned by $\sum^{[n/2]}_{k=0}H_{n}e_{(k)}H_{n}.$
In particular, its dimension is at most the one of the Brauer algebra.
\end{lem}

\subsection{The $Br_{n}(r, q)$-modules $V^*_k$}
Wenzl \cite{W3} defined an action of generators of $q$-Brauer algebra on module $V^*_k$ as follows:
\[g_{j}v_{d} =
\begin{cases}
qv_{d} & \text{if} \ s_{j}d = d,\\
v_{s_{j}d } &\text{ if }\ l( s_{j}d) >  l(d), \\
(q-1)v_{d} + qv_{s_{j}d} &\text{if } \ l(s_{j}d) <  l(d),
\end{cases} \]
and
\[ehg^{+}_{2, j_{2}}g^{-}_{1, j_{1}}v_{1} =
\begin{cases}
q^{N}hg^{+}_{3, j_{2}}v_{1} & \text{if}\  g^{-}_{1, j_{1}} = 1,\\
q^{-1}hg^{-}_{j_{2}+1, j_{1}}v_{1} &\text{ if }\ g^{-}_{1, j_{1}} = 1, \\
0 &\text{if } j_{1} \neq 2k \ and \ j_{2} \neq 2k + 1,
\end{cases} \]
where $v_{1}$ is defined as in Lemma \ref{bd2} and h $\in H_{3, n}$.

\begin{lem}\label{bd10}(\cite{W3}, Proposition 3.6).
The action of the elements $g_{j}$ with $1 \leq j < n$ and $e$ on $V^*_k$ as given above
defines a representation of $Br_{n}(r, q).$
\end{lem}

\subsection{ A basis for $q$-Brauer algebra}
We can now find a basis for the $q$-Brauer algebra that bases on
considering diagrams of Brauer algebra.
This basis gives a cellular structure on the $q$-Brauer algebra that will be shown in the next section.
In this section we choose the parameter on the Brauer algebra to be an integer N $\in \Z \setminus \{0\}$.

Given a diagram $d \in D_{n}(N)$ with exactly 2k horizontal edges,
it is considered as concatenation of three diagrams $(d_{1}, \omega_{(d)}, d_{2})$ as follows:

1. $d_{1}$ is a diagram in which the top row has the positions of horizontal edges as these in the first row of d,
its bottom row is like a row of diagram $e_{(k)}$, and there is no crossing between any two vertical edges.

2. Similarly, $d_{2}$ is a diagram where its bottom row is the same row in d,
the other one is similar to that of $e_{(k)}$,
and there is no crossing between any two vertical edges.

3. The diagram $\omega_{(d)}$ is described as follows : we enumerate the free vertices,
which just make vertical edges, in both rows of $d$ from left to right by $2k+1, 2k+2,..., n$.
We also enumerate the vertices in each of two rows of $\omega_{(d)}$ from left to right by $1, 2, ..., 2k+1, 2k+2,..., n$.
Assume that each vertical edge in d is connected by $i-th$ vertex of the top row and $j-th$ vertex
of the other one with $2k+1 \leq i, j \leq n$.
Define $\omega_{(d)}$ a diagram which has first vertical edges 2k joining $m-$th points in each of two rows together with $1 \leq m \leq 2k$,
and its other vertical edges are obtained by maintaining the vertical edges $(i, j)$ of $d$.

\begin{ex}\label{ex7} For $n=7, k=2.$ Given a diagram
$$ \begin{array}{c} 
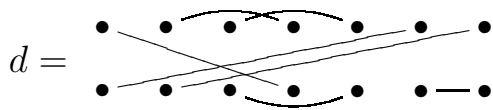
\begin{xy}
\xymatrix@!=0.01pc{ \bullet \ar@{-}[rrrd] & \bullet \ar@{-}@/^/[rr] &\bullet \ar@{-}@/^/[rr] &\bullet & \bullet &\bullet \ar@{-}[llllld] & \bullet \ar@{-}[llllld] \\
  \ar@{}[u]^{\text{ \large $d = $ \  }} \bullet & \bullet &\bullet \ar@{-}@/_/[rr] &\bullet &\bullet &\bullet \ar@{-}[r] & \bullet }
\end{xy}\end{array}$$
the diagram d can be expressed as product of diagrams in the following:
$$ \begin{array}{c} 
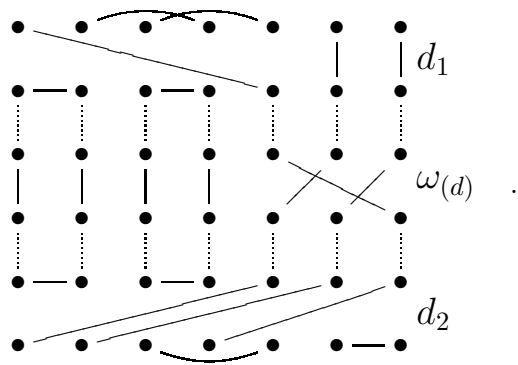
\begin{xy}
\xymatrix@!=0.01pc{ \bullet \ar@{-}[rrrrd] & \bullet \ar@{-}@/^/[rr] &\bullet \ar@{-}@/^/[rr] &\bullet & \bullet &\bullet \ar@{-}[d] & \bullet \ar@{-}[d] \ar@{}[d]^{\text{ \large $d_{1}$ \  }} \\
    \bullet \dta[d] \ar@{-}[r] & \bullet \dta[d] & \bullet \dta[d] \ar@{-}[r] & \bullet \dta[d] &\bullet \dta[d] &\bullet \dta[d] &\bullet \dta[d] \\
    \bullet \ar@{-}[d] & \bullet \ar@{-}[d] &\bullet \ar@{-}[d] & \bullet \ar@{-}[d] &\bullet \ar@{-}[rrd] &\bullet \ar@{-}[ld] &\bullet \ar@{-}[ld] \ar@{}[d]^{\text{ \large $\omega_{(d)}$  }}  \\
    \bullet \dta[d] & \bullet \dta[d] &\bullet \dta[d] &\bullet \dta[d] &\bullet \dta[d] &\bullet \dta[d] &\bullet \dta[d] \\
    \bullet \ar@{-}[r] & \bullet &\bullet \ar@{-}[r] & \bullet &\bullet \ar@{-}[lllld] &\bullet \ar@{-}[lllld] &\bullet \ar@{-}[llld] \ar@{}[d]^{\text{ \large $d_{2}$ \  }} \\
    \bullet & \bullet &\bullet \ar@{-}@/_/[rr] &\bullet &\bullet &\bullet \ar@{-}[r] & \bullet }
\end{xy}\end{array}.$$
Thus we have $ d= (N)^{-2} \cdot d_{1} \omega_{(d)} d_{2}.$
\end{ex}

Notice that diagram $\omega_{(d)}$ can be seen as a permutation of symmetric group $S_{2k+1, n}.$
Since the above expression is unique with respect to each diagram $d$, the form $(d_{1},\ \omega_{(d)},\ d_{2})$ is determined uniquely.
By Lemmas \ref{bd2} and \ref{bd3}, there exist unique permutations
$\omega_{1} \in B^*_k$ and $\omega_{2} \in B_k$
such that  $d_{1}=\omega_{1}e_{(k)}$ and $d_{2}=e_{(k)}\omega_{2}$
with $\ell(d_{1}) = \ell(\omega_{1})$, $\ell(d_{2}) = \ell(\omega_{2}).$
Thus a diagram $d$ is uniquely represented by the 3-tuple $(\omega_{1},\ \omega_{(d)},\ \omega_{2})$ with $\omega_{1} \in B^*_k$, $\omega_{2} \in B_k$ and $\omega_{(d)} \in S_{2k+1, n}$  such that
$d= N^{-k}\omega_{1}e_{(k)} \omega_{(d)}e_{(k)}\omega_{2}$ and $\ell(d) = \ell(\omega_{1}) + \ell(\omega_{(d)}) + \ell(\omega_{2}).$
We call such a unique representation {\it a reduced expression} of $d$ and briefly write $(\omega_{1}, \omega_{(d)},\omega_{2})$.

\begin{ex}\label{ex8} The above example implies that $d_{1}=\omega_{1}e_{(2)}$ with $\omega_{1} = s_{1, 4}s_{2} \in B^*_2$,
$d_{2}=e_{(2)}\omega_{2}$ with $\omega_{2} = s_{4,1}s_{5,2}s_{6,4} \in B_2$
and $\omega_{(d)} = s_{5}s_{6}$.
Thus we obtain
$$d= N^{-2}(\omega_{1}e_{(2)}) (s_{5}s_{6})(e_{(2)}\omega_{2})
=\omega_{1}e_{(2)} s_{5}s_{6}\omega_{2} =\omega_{1}s_{5}s_{6}e_{(2)} \omega_{2}.$$
Hence, $d$ has a unique reduced expression $(\omega_{1},\ s_{5}s_{6},\ \omega_{2}) = (s_{1,4}s_{2},\ s_{5}s_{6},\ s_{4,1}s_{5,2}s_{6,4})$
with
$$\ell(d) = \ell(\omega_{1}) + \ell(s_{5}s_{6}) + \ell(\omega_{2}) = 5 + 2 + 11 = 18.$$
\end{ex}

\begin{rem} \label{rem1}
1. We abuse notation by denoting $e_{(k)}$ both a certain diagram in the Brauer algebra $D_{n}(N)$ and an element in the $q$-Brauer algebra $Br_{n}(r, q)$.
Given a diagram $d$ the above arguments imply that diagrams $e_{(k)}$ and $\omega_{(d)}$ commute on the Brauer algebra.
Similarly, this also remains on the level of $q$-Brauer algebra, that is, $e_{(k)}g_{\omega_{(d)}} = g_{\omega_{(d)}}e_{(k)}$.
In particular, since $\omega_{(d)} \in S_{2k+1, n}$, it is sufficient to show that
$$e_{(k)}g_{i} =  g_{i}e_{(k)}\ with\ 2k+1 \leq i \leq n-1.$$
Obviously, $eg_{i}= g_{i}e$ by the equality in $(E_{2}).$ Suppose that  $e_{(k-1)}g_{i}= g_{i}e_{(k-1)}$, then
\begin{align*}
e_{(k)}g_{i} & \overset{(\ref{eq6})}{=} (eg^{+}_{2, 2k-1}g^{-}_{1, 2k-2}e_{(k-1)})g_{i} =
e(g^{+}_{2, 2k-1}g^{-}_{1, 2k-2})g_{i}e_{(k-1)}\\
&\overset{(H_{2})}{=} g_{i}e(g^{+}_{2, 2k-1}g^{-}_{1, 2k-2})e_{(k-1)} \overset{(\ref{eq6})}{=}g_{i}e_{(k)}.
\end{align*}

2.  Wenzl \cite{W3} introduced a reduced expression of Brauer diagram $d$ in the general way
$$\ell(d)= min \{ \ell(\delta_{1})+\ell(\delta_{2}): \ d= \delta_{1}e_{(k)}\delta_{2}, \ \delta_{1}, \delta_{2} \in S_{n} \}.$$
This definition implies several different reduced expressions with respect to a diagram $d$.
In the above construction we give the notation of reduced expression of a diagram~$d$ with a difference.
Our definition also bases on the length of diagram $d$.
However, here $d$ is represented by partial diagrams $d_{1},$ $d_{2}$ and $\omega_{(d)}$ such that $d =N^{-k} d_{1}\omega_{(d)}d_{2}$,
and this allows us to produce a unique reduced expression of $d$.
No surprise that the length of a diagram $d$ is the same in both our description and that of Wenzl, since
\begin{align*}
d &= N^{-k} d_{1}\omega_{(d)}d_{2} = N^{-k}(\omega_{1}e_{(k)}) \omega_{(d)}(e_{(k)}\omega_{2})\\
&= N^{-k}(\omega_{1}\omega_{(d)})e^{2}_{(k)}\omega_{2} = (\omega_{1}\omega_{(d)})e_{(k)}\omega_{2} \\
& = \omega_{1}e_{(k)}(\omega_{(d)}\omega_{2})
= \delta_{1}e_{(k)}\delta_{2}
\end{align*}
where $(\delta_{1}, \delta_{2})= (\omega_{1}\omega_{(d)},\ \omega_2)$ or $(\delta_1, \delta_{2}) = (\omega_{1},\ \omega_{(d)}\omega_{2})$.
Therefore
$$\ell(d) = \ell(d_{1}) + \ell(\omega_{(d)}) + \ell(d_{2}) = min \{\ell(\delta_{1}) + \ell(\delta_{2}) \}.$$
\end{rem}

\begin{ex}\label{ex9} We consider the diagram $d$ as in the example \ref{ex7}. Using $reduced \ expression$ definition of Wenzl,
$d$ can be rewritten as a product $d=\delta_{1}e_{(k)}\delta_{2}$ in the following way:

{\it Case 1.}
$$ \begin{array}{c} 
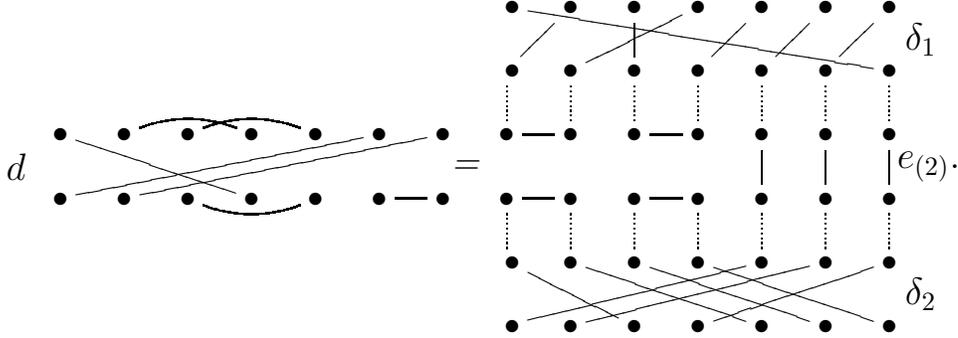
\begin{xy}
\xymatrix@!=0.01pc{ &\ &\ &\ &\ &\ &\ &\ \bullet \ar@{-}[rrrrrrd] & \bullet \ar@{-}[ld] &\bullet \ar@{-}[d] &\bullet \ar@{-}[lld] & \bullet \ar@{-}[ld] &\bullet \ar@{-}[ld] & \bullet \ar@{-}[ld] \ar@{}[d]^{\text{ \large $\delta_{1}$  }}\\
&\ &\ &\ &\ &\ &\ &\ \bullet \dta[d] & \bullet \dta[d] &\bullet \dta[d] &\bullet \dta[d] & \bullet \dta[d] &\bullet \dta[d] & \bullet \dta[d] \\
           \bullet \ar@{-}[rrrd] & \bullet \ar@{-}@/^/[rr] &\bullet \ar@{-}@/^/[rr] &\bullet & \bullet &\bullet \ar@{-}[llllld] & \bullet \ar@{-}[llllld] \ar@{}[d]^{\text{\large {=} }}
&\bullet \ar@{-}[r] & \bullet &\bullet \ar@{-}[r] &\bullet & \bullet \ar@{-}[d] &\bullet \ar@{-}[d] & \bullet \ar@{-}[d] \ar@{}[d]^{\text{\large $e_{(2)}.$ }}\\
  \ar@{}[u]^{\text{ \large $d$ \  }} \bullet & \bullet &\bullet \ar@{-}@/_/[rr] &\bullet &\bullet &\bullet \ar@{-}[r] & \bullet
& \bullet \ar@{-}[r] \dta[d] & \bullet \dta[d] & \bullet  \ar@{-}[r] \dta[d] & \bullet \dta[d] &\bullet \dta[d] &\bullet \dta[d] &\bullet \dta[d]  \\
&\ &\ &\ &\ &\ &\ &\ \bullet \ar@{-}[rrd] & \bullet \ar@{-}[rrrd] & \bullet \ar@{-}[rrrd] & \bullet \ar@{-}[rrrd] &\bullet \ar@{-}[lllld] &\bullet \ar@{-}[lllld] &\bullet \ar@{-}[llld] \ar@{}[d]^{\text{ \large $\delta_{2}$  }} \\
&\ &\ &\ &\ &\ &\ &\ \bullet & \bullet & \bullet & \bullet &\bullet &\bullet &\bullet}
\end{xy}\end{array}$$
Where $\delta_{1} = \omega_{1} s_{5}s_{6} = s_{1,4}s_{2}s_{5}s_{6} = s_{1,6}s_{2} \ \in \ B^*_2$ with $\ell(\delta_{1}e_{(2)}) = \ell(\delta_{1})= 7$,
and\\
$\delta_{2} = \omega_{2} = s_{4,1}s_{5,2}s_{6,4} \ \in \ B_2$, with $\ell(e_{(2)}\delta_{2}) = \ell(\delta_{2}) = 11.$
In this case $d$ has a reduced expression
$$(\delta_{1}, \delta_{2}) = (s_{1,4}s_{2},\ s_{4,1}s_{5,2}s_{6,4})$$
with
$$\ell(d) = \ell(\delta_{1}) + \ell(\delta_{2}) = 7 + 11 = 18.$$

{\it Case 2}. The diagram $d$ can also be represented by another way as follows:

$$ \begin{array}{c} 
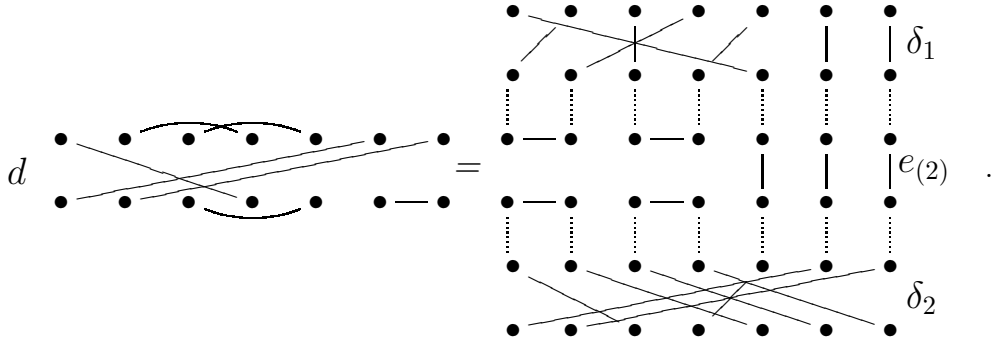
\begin{xy}
\xymatrix@!=0.01pc{ &\ &\ &\ &\ &\ &\ &\ \bullet \ar@{-}[rrrrd] & \bullet \ar@{-}[ld] &\bullet \ar@{-}[d] &\bullet \ar@{-}[lld] & \bullet \ar@{-}[ld] &\bullet \ar@{-}[d] & \bullet \ar@{-}[d] \ar@{}[d]^{\text{ \large $\delta_{1}$  }}\\
&\ &\ &\ &\ &\ &\ &\ \bullet \dta[d] & \bullet \dta[d] &\bullet \dta[d] &\bullet \dta[d] & \bullet \dta[d] &\bullet \dta[d] & \bullet \dta[d] \\
           \bullet \ar@{-}[rrrd] & \bullet \ar@{-}@/^/[rr] &\bullet \ar@{-}@/^/[rr] &\bullet & \bullet &\bullet \ar@{-}[llllld] & \bullet \ar@{-}[llllld] \ar@{}[d]^{\text{\large {=} }}
&\bullet \ar@{-}[r] & \bullet &\bullet \ar@{-}[r] &\bullet & \bullet \ar@{-}[d] &\bullet \ar@{-}[d] & \bullet \ar@{-}[d] \ar@{}[d]^{\text{\large $e_{(2)}$ }}\\
  \ar@{}[u]^{\text{ \large $d$ \  }} \bullet & \bullet &\bullet \ar@{-}@/_/[rr] &\bullet &\bullet &\bullet \ar@{-}[r] & \bullet
& \bullet \ar@{-}[r] \dta[d] & \bullet \dta[d] & \bullet  \ar@{-}[r] \dta[d] & \bullet \dta[d] &\bullet \dta[d] &\bullet \dta[d] &\bullet \dta[d]  \\
&\ &\ &\ &\ &\ &\ &\ \bullet \ar@{-}[rrd] & \bullet \ar@{-}[rrrd] & \bullet \ar@{-}[rrrd] & \bullet \ar@{-}[rrrd] &\bullet \ar@{-}[ld] &\bullet \ar@{-}[llllld] &\bullet \ar@{-}[llllld] \ar@{}[d]^{\text{ \large $\delta_{2}$  }} \\
&\ &\ &\ &\ &\ &\ &\ \bullet & \bullet & \bullet & \bullet &\bullet &\bullet &\bullet}
\end{xy}\end{array}.$$
Where $\delta_{1} = \omega_{1} = s_{1,4}s_{2}\ \in \ B^*_2$ with $\ell(\delta_{1}e_{(2)}) = \ell(\delta_{1})= 5$, \\
and $\delta_{2} = s_{5}s_{6} \omega_{2} =  s_{5}s_{6}s_{4,1}s_{5,2}s_{6,4}= s_{4,2}s_{5,1}s_{6,2}  \ \in \ B_2$ with
$\ell(e_{(2)}\delta_{2}) = \ell(\delta_{2})= 13$.\\
The reduced expression of $d$ in this case is
$$(\delta_{1}, \delta_{2}) =(s_{1,6}s_{2},\ s_{4,2}s_{5,1}s_{6,2})$$
with
$$\ell(d) = \ell(\delta_{1}) + \ell(\delta_{2}) = 5 + 13 = 18.$$
Therefore, comparing with the example \ref{ex8} implies that $\ell(d)$ is the same.
Two cases above yield that Wenzl's definition of the reduced expression of a diagram $d$ is not unique
since
$$(s_{1,4}s_{2},\ s_{4,1}s_{5,2}s_{6,4})\ \neq \ (s_{1,6}s_{2},\ s_{4,2}s_{5,1}s_{6,2}).$$
\end{ex}

\begin{defn}\label{dn10}
For each diagram $d$ of the Brauer algebra $D_{n}(N),$ we define a corresponding element, say $g_{d}$, in $Br_{n}(r, q)$ as follows:
if $d$ has exactly $2k$ horizontal edges and $(\omega_{1}, \omega_{(d)},\omega_{2})$ is a reduced expression of $d$,
then define $g_{d} := g_{\omega_{1}}e_{(k)}g_{\omega_{(d)}}g_{\omega_{2}}$.
If the diagram $d$ has no horizontal edge,
then $d$ is seen as a permutation $\omega_{(d)}$ of the symmetric group $S_{n}$.
And in this case, define $g_{d}=g_{\omega_{(d)}}.$
\end{defn}
The main result of this section is stated below.

\begin{thm} \label{dl11}
The $q$-Brauer algebra $Br_{n}(r,q)$ over the ring $R$ has a basis $\{ g_{d}$ $| d \in D_{n}(N) \}$
labeled by diagrams of the Brauer algebra.
\end{thm}

\begin{proof}
A diagram $d$ of Brauer algebra with exactly $2k$ horizontal edges
has a unique reduced expression with data $(\omega_{1},\omega_{(d)}, \omega_{2})$.
By the uniqueness of reduced expression with respect to a diagram d in $D(N)$, the elements $g_{d}$ in $Br_{n}(r,q)$ are well-defined.
Observe that these elements $g_{d}$  belong to $H_{n}e_{(k)}H_{n}$
since $g_{\omega_{1}}, g_{\omega_{(d)}}$ and $g_{\omega_{2}}$ are in $H_{n}$.
Lemma \ref{bd1}  shows that there is a faithful representation of the Brauer algebra $D_{n}(N)$ on
$\bigoplus_{k=0}^{[n/2]}V^*_k$.
By Lemma \ref{bd10}, this is a specialization of the representation of $Br_{n}(r,q)$ on the same direct sum of modules $V^*_k$,
and hence, the dimension of $Br_{n}(r,q)$ has to be at least the one of $D_{n}(N)$.
Now, the other dimension inequality follows from using the result in Lemma \ref{bd9}. The theorem is proved.
\end{proof}
The next result provides an involution on the $q$-Brauer algebra $Br_{n}(r, q).$

\begin{prop} \label{md12} Let $i$ be the map from $Br_{n}(r, q)$ to itself
defined by
$$i(g_{\omega}) = g_{\omega^{-1}} \ and \ i(e) = e$$
for each $\omega \in S_{n}$,
extended an anti-homomorphism. Then $i$ is an involution on the $q$-Brauer algebra $Br_{n}(r, q).$
\end{prop}
\begin{proof}
It is sufficient to show that $i$ maps a basis element $g_{d^*}$ to a basis element $i(g_{d^*})$
on the $q$-Brauer algebra $Br_{n}(r, q)$. If given a diagram $d^*$ with no horizontal edge,
then $d^*$ is as a permutation in $S_{n}$, it implies that
obviously $i(g_{d^*})= g_{(d^*)^{-1}}=g_d$ is a basis element
of the $q$-Brauer algebra, where $d$ is diagram which is obtained after rotating $d^*$ downward via an horizontal axis.
If the diagram $d^*= e_{(k)}$, then by Definition \ref{dn10} the corresponding basis element in the $q$-Brauer algebra is $g_{d^*}= e_{(k)}.$
The equality $i(g_{d^*}) = i(e_{(k)}) = e_{(k)}$ is obtained by induction on $k$ as follows:
with $k = 1$ obviously $i(e) = e$ by definition. Suppose $i(e_{(k-1)}) = e_{(k-1)}$, then
\begin{align*}
i(e_{(k)}) &\overset{(\ref{eq6})}{=} i(eg^{+}_{2, 2k-1}g^{-}_{1, 2k-2}e_{(k-1)})\\
&= i(e_{(k-1)}) i(g^{+}_{2, 2k-1}g^{-}_{1, 2k-2}) i(e)\\
&= e_{(k-1)} g^{-}_{2k-2, 1} g^{+}_{2k-1, 2}e \overset{L\ref{bd7}(c)}{=} e_{(k)}.
\end{align*}
Now given a reduced expression $(\omega_{1}, \omega_{(d^*)}, \omega_{2})$ of a diagram $d^*$,
where $\omega_{1} \in B^*_k$,
$\omega_{2} \in B_k$ and $\omega_{(d^*)} \in S_{2k+1, n}$,
the corresponding basis element on the $q$-Brauer algebra $Br_{n}(r, q)$ is $g_{d^*}= g_{\omega_{1}}g_{\omega_{(d^*)}} e_{(k)}g_{\omega_{2}}.$
This yields
\begin{align} \label{eq19}
i(g_{d^*}) &= i(g_{\omega_{1}}g_{\omega_{(d^*)}} e_{(k)}g_{\omega_{2}})
= i(g_{\omega_{2}})i(e_{(k)})i(g_{\omega_{(d^*)}})i(g_{\omega_{1}})\\
&= g_{\omega_{2}^{-1}}e_{(k)}g_{\omega^{-1}_{(d^*)}}g_{\omega_{1}^{-1}}
\overset{Re\ref{rem1}(1)}{=} g_{\omega_{2}^{-1}}g_{\omega^{-1}_{(d^*)}}e_{(k)}g_{\omega_{1}^{-1}} \notag.
\end{align}
$\omega_{1} \in B^*_k$ and $\omega_{2} \in B_k$ imply that
$\omega_{1}^{-1} \in B_k$ and $\omega_{2}^{-1} \in B^*_k$.
Therefore, by Lemmas \ref{bd2} and \ref{bd3}, $\ell(e_{(k)}\omega_{1}^{-1}) = \ell(\omega_{1}^{-1})$
and $\ell(\omega_{2}^{-1}e_{(k)}) = \ell(\omega_{2}^{-1}).$
This means that the 3-triple $(\omega_{2}^{-1}, \omega_{(d^*)}^{-1}, \omega_{1}^{-1})$
is a reduced expression of the diagram
$d^* =N^{-k} \omega_{2}^{-1}e_{(k)}\omega_{(d^*)}^{-1}e_{(k)}\omega_{1}^{-1}.$
Thus $i(g_{d^*}) = g_{\omega_{2}^{-1}}g_{\omega^{-1}_{(d^*)}}e_{(k)}g_{\omega_{1}^{-1}}$ is a basis element in $Br_{n}(r, q)$
corresponding to the diagram~$d$.
\end{proof}
The next corollary is needed for Section 4.

\begin{cor}\label{hq13}
(a) $g^{+}_{2m-1, 2j}e_{(k)} = g^{+}_{2j+1, 2m }e_{(k)}$ and
  $g^{-}_{2m-1, 2j}e_{(k)} = g^{-}_{2j+1, 2m }e_{(k)}$ \\ for $ 1 \leq m \leq j < k.$

(b) $e_{(k)}g^{+}_{2l, 1} = e_{(k)}g^{+}_{2, 2l+1}$ and
  $e_{(k)}g^{-}_{2l, 1} = e_{(k)}g^{-}_{2, 2l+1}$ for $l < k.$

(c) $e_{(k)}g^{+}_{2j, 2i-1} = e_{(k)}g^{+}_{2i, 2j+1}$ and
  $e_{(k)}g^{-}_{2j, 2i-1} = e_{(k)}g^{-}_{2i, 2j+1}$ for $ 1 \leq i \leq j < k.$

(d) $(\dfrac{r-1} {q-1})^{j-1}e_{(k+1)} = e_{(k)}g^{-}_{2k, 2j-1 }g^{+}_{2k+1, 2j}e_{(j)}$ for $1 \leq j < k$.

(e) $e_{(k)}g_{2j}e_{(j)} = r(\dfrac{r-1} {q-1})^{j-1}e_{(k)}$ for $1 \leq j \leq k.$

(f) $e_{(k)}H_{n}e_{(j)} \subset e_{(k)}H_{2j+1, n} \ + \sum_{m \geq k+1} H_{n} e_{(m)}H_{n},$ where $j \leq k$.
\end{cor}
\begin{proof}
These results, without the relations $(a)$ and $(c)$, are directly deduced from Lemmas \ref{bd6} and \ref{bd8} using the property of the above involution.
The statement $(a)$ can be proven by induction on $m$ as follow:
With $m=1$ obviously $(a)$ follows from Lemma \ref{bd6}(b) with $j=l.$ Suppose that $(a)$ holds for $m-1$,
that is,
\begin{align}\label{eq20}
&g^{+}_{2m-3, 2(j-1)}e_{(k)} = g^{+}_{2(j-1)+1, 2m-2 }e_{(k)} \text{ and }
\end{align}
\begin{align}\label{eq21}
&g^{-}_{2m-3, 2(j-1)}e_{(k)} = g^{-}_{2(j-1)+1, 2m-2 }e_{(k)}
\text{\hspace{2cm} for $ 1 \leq m \leq j < k$}.
\end{align}
Then
\begin{align*}
g^{+}_{2m-1, 2j}e_{(k)} &= g^{-}_{2m-2, 2m-3 }(g^{+}_{2m-3, 2j-2})g^{+}_{2j-1, 2j}e_{(k)}\\
&\overset{L\ref{bd6}(c)}= g^{-}_{2m-2, 2m-3 }(g^{+}_{2m-3, 2j-2})g^{+}_{2j+1, 2j}e_{(k)}\\
&\overset{(H_{2})}{=} g^{-}_{2m-2, 2m-3 }g^{+}_{2j+1, 2j}(g^{+}_{2m-3, 2j-2})e_{(k)}\\
&\overset{(\ref{eq20})}{=} g^{-}_{2m-2, 2m-3 }g^{+}_{2j+1, 2j}(g^{+}_{ 2j-1, 2m-2})e_{(k)}\\
&= g^{-}_{2m-2, 2m-3 }g^{+}_{2j+1, 2m-2}e_{(k)}\\
&\overset{(H_{2})}{=} g^{+}_{2j+1, 2m}g^{-}_{2m-2, 2m-3 }g^{+}_{2m-1, 2m-2}e_{(k)}\\
&\overset{L\ref{bd6}(c)}{=} g^{+}_{2j+1, 2m}g^{-}_{2m-2, 2m-3 }g^{+}_{2m-3, 2m-2}e_{(k)}\\
&= g^{+}_{2j+1, 2m}e_{(k)}.
\end{align*}
The other equality is proven similarly.
The relation $(c)$ is directly deduced from $(a)$ by using the involution.
\end{proof}
Notice that the equality (c) in Lemma \ref{bd6} is the special case of the above equality~(a).

\subsection{An algorithm producing basis elements of the $q$-Brauer algebras}
We introduce here an algorithm producing a basis element $g_{d}$ on the $q$-Brauer algebras $Br_{n}(r, q)$
from a given diagram $d$ in the Brauer algebra $D_{n}(N).$
This algorithm's construction bases on the proof of Lemma \ref{bd2}(a) (see Lemma 1.2(a), \cite{W3} for a complete proof).
From the expression of an arbitrary diagram $d$ as concatenation of three partial diagrams $(d_{1}, \omega_{(d)}, d_{2})$ in Subsection 3.2,
it is sufficient to consider diagrams $d^*$ as the form of $d_{1}$. That is, $d^*$ has exactly k horizontal edges on each row,
its bottom row is like a row of $e_{(k)}$, and there is no crossing between any two vertical edges.
Let $\mathscr{D}^*_{k, n}$ be the set of all diagrams $d^*$ above.

Recall that a permutation $s_{i,j}$ with $1 \leq i,\ j\leq n-1$
can be considered as a diagram, say $d^*_{(i, j+1)}$ if $i \leq j$ or $d^*_{(i+1, j)}$ if $i > j$, in the Brauer algebra such that its free points,
including $1$, $2$,...,$i-1$, $j+2$, ... $n$, are fixed.

\begin{ex} In $D_{7}(N)$ the permutation $s_{6, 3}$ corresponds to the diagram
$$ \begin{array}{c} 
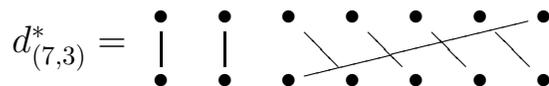
\begin{xy}
\xymatrix@!=0.01pc{ \bullet \ar@{-}[d] & \bullet \ar@{-}[d] &\bullet \ar@{-}[rd] &\bullet\ar@{-}[rd] & \bullet \ar@{-}[rd] &\bullet \ar@{-}[rd] & \bullet \ar@{-}[lllld] \\
  \ar@{}[u]^{\text{ \large $d^*_{(7, 3)}  = $ \  }} \bullet &\bullet &\bullet &\bullet &\bullet &\bullet & \bullet }
\end{xy}\end{array}$$
\end{ex}

\subsubsection{The algorithm}
Given a diagram $d^*$ of $\mathscr{D}^*_{k, n}$,
we number the vertices in both rows of $d^*$ from left to right by $1$, $2$, ..., $n.$
Note that for $2k+1 \leq i \leq n$ if the $i$-th vertex in its bottom row joins to the $f(i)$-th vertex in the top row,
then $f(i) <  f(i+1)$ since there is no intersection between any two vertical edges in the diagram $d^*$.
This implies that concatenation of diagrams $d^*_{(n, f(n))}$ and $d^*$ yields a new diagram
$$d^*_{1}=d^*_{(n, f(n))}d^*$$
whose $n$-th vertex in the bottom row joins that of the top row and whose other vertical edges retain those of $d^*$.
That is, the diagram $d^*_{1}$ has the $(n-1)$-th vertex in its bottom row joining to the point $f(n-1)$-th vertex in the top row.
Again, a concatenation of diagrams $d^*_{(n-1, f(n-1))}$ and $d^*_{1}$ produces a diagram
$$d^*_{2}=d^*_{(n-1, f(n-1))}d^*_{1} = d^*_{(n-1, f(n-1))}d^*_{(n, f(n))}d^*$$
whose $n$-th and $(n-1)$-th vertices in the bottom row join, respectively,
those of the top row and whose other vertical edges maintain these in $d^*$.

Proceeding in this way, we determine a series of diagrams $d^*_{(n, f(n))}$, $d^*_{(n-1, f(n-1))}$,\dots,
$d^*_{(2k+1, f(2k+1))}$ such that
$$d'= d^*_{(2k+1, f(2k+1))}.. .d^*_{(n-1, f(n-1))}d^*_{(n, f(n))}\ d^*$$
is a diagram in $S_{2k}e_{(k)}$.
Here, $d'$ can be seen as a diagram in $D_{2k}(N)$ with only horizontal edges
to which we add $(n-2k)$ strictly vertical edges to the right.
Subsequently, set $i_{2k-2}$ the label of the vertex in the top row of $d'$ which is connected with the $2k$-th vertex in the same row
and $t_{(2k-2)}= s_{i_{2k-2}, 2k-2}.$ Then the new diagram
$$d'_{1}= t^{-1}_{(2k-2)}d'$$
has two vertices $2k$-th and $(2k-1)$-st in the top row which are connected by a horizontal edge. Proceeding this process, finally $d^*$ transforms into $e_{(k)}.$
$$e_{(k)}=t^{-1}_{(2)}...t^{-1}_{(2k-4)} t^{-1}_{(2k-2)}d^*_{(2k+1, f(2k+1))}.. .d^*_{(n-1, f(n-1))}d^*_{(n, f(n))}\ d^*.$$
Hence, $d^*$ can be rewritten as $d^* = \omega e_{(k)}$, where
\begin{align} \label{eq22}
\omega &= (t^{-1}_{(2)}...t^{-1}_{(2k-4)} t^{-1}_{(2k-2)}d^*_{(2k+1, f(2k+1))}.. .d^*_{(n-1, f(n-1))}d^*_{(n, f(n))})^{-1}\\
&=(d^*_{(2k+1, f(2k+1))}.. .d^*_{(n-1, f(n-1))}d^*_{(n, f(n))})^{-1}(t^{-1}_{(2)}...t^{-1}_{(2k-4)} t^{-1}_{(2k-2)})^{-1}\notag \\
&=d^*_{(f(n), n)}d^{*}_{(f(n-1), n-1)}...\ d^{*}_{(f(2k+1), 2k+1)}t_{(2k-2)}t_{(2k-4)}...t_{(2)}\notag \\
&= s_{f(n), n-1}s_{f(n-1), n-2}...\ s_{f(2k+1), 2k}s_{i_{2k-2}, 2k-2}s_{i_{2k-4}, 2k-4}...\ s_{i_{2}, 2}\notag
\end{align}
with $f(i) < f(i+1)$ for $2k+1 \leq i \leq n-1.$

Notice that the involution $*$ in $D_n(N)$ maps $d^*$ to the diagram $d$ which is of the form $e_{(k)} \omega^{-1}$
with
$$\omega^{-1} = s_{2, i_{2}}...\ s_{2k-4, i_{2k-4}}s_{2k-2, i_{2k-2}}s_{2k, f(2k+1)}...\ s_{n-2, f(n-1)}s_{n-1, f(n)}.$$
By Definition \ref{dn10}, the corresponding basis elements $g_{d^*}$ ($g_{d}$) in the $q$-Brauer algebra are
$$g_{d^*}=g_{\omega}e_{(k)} \text{ and } g_{d}= e_{(k)}g_{\omega^{-1}}.$$

\begin{ex}{In $D_{7}(N)$ we consider the following diagram $d^*$}$$
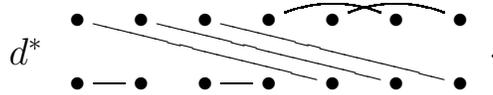
\begin{array}{c} 
\begin{xy}
\xymatrix@!=0.01pc{ \bullet \ar@{-}[rrrrd] & \bullet \ar@{-}[rrrrd] &\bullet \ar@{-}[rrrrd] &\bullet \ar@{-}@/^/[rr] & \bullet \ar@{-}@/^/[rr] &\bullet & \bullet  \\
    \ar@{}[u]^{\text{ \large $d^*$ \  }} \bullet \ar@{-}[r] & \bullet & \bullet \ar@{-}[r] & \bullet &\bullet &\bullet &\bullet }
\end{xy}\end{array}.$$

Step 1. Transform $d$ into $d'$
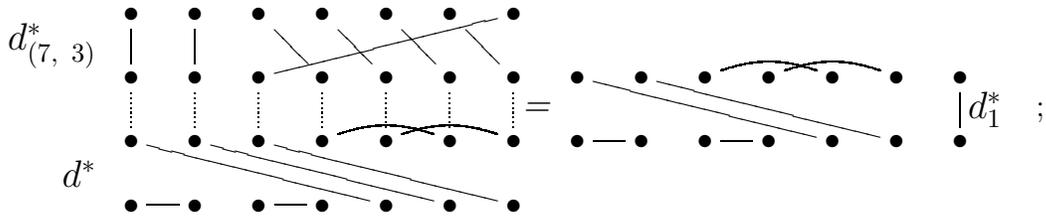
$$ \begin{array}{c} 
\begin{xy}
\xymatrix@!=0.01pc{ \bullet \ar@{-}[d] & \bullet \ar@{-}[d] &\bullet \ar@{-}[rd] &\bullet\ar@{-}[rd] & \bullet\ar@{-}[rd] &\bullet\ar@{-}[rd] & \bullet \ar@{-}[lllld] \\
  \ar@{}[u]^{\text{ \large $d^*_{(7,\ 3)}$ \  }}  \bullet \dta[d] & \bullet \dta[d] & \bullet \dta[d] & \bullet \dta[d] &\bullet \dta[d] &\bullet \dta[d] &\bullet \dta[d] \ar@{}[d]^{\text{\large {=} }}
&\bullet \ar@{-}[rrrrd] & \bullet \ar@{-}[rrrrd] &\bullet \ar@{-}@/^/[rr] &\bullet \ar@{-}@/^/[rr] & \bullet &\bullet & \bullet \ar@{-}[d] \ar@{}[d]^{\text{\large $d^*_{1}$ }}\\
      \bullet \ar@{-}[rrrrd] & \bullet \ar@{-}[rrrrd] &\bullet \ar@{-}[rrrrd] &\bullet \ar@{-}@/^/[rr] & \bullet \ar@{-}@/^/[rr] &\bullet & \bullet
& \bullet \ar@{-}[r] & \bullet & \bullet \ar@{-}[r] & \bullet &\bullet &\bullet &\bullet  \\
   \ar@{}[u]^{\text{ \large $d^*$ \  }} \bullet \ar@{-}[r] & \bullet & \bullet \ar@{-}[r] & \bullet &\bullet &\bullet &\bullet }
\end{xy}\end{array};$$

$$ \begin{array}{c} 
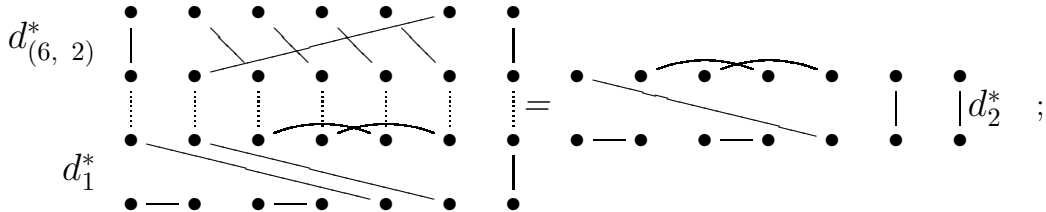
\begin{xy}
\xymatrix@!=0.01pc{ \bullet \ar@{-}[d] & \bullet \ar@{-}[rd] &\bullet \ar@{-}[rd] &\bullet\ar@{-}[rd] & \bullet\ar@{-}[rd] &\bullet\ar@{-}[lllld] & \bullet \ar@{-}[d] \\
  \ar@{}[u]^{\text{ \large $d^*_{(6,\ 2)}$ \  }}  \bullet \dta[d] & \bullet \dta[d] & \bullet \dta[d] & \bullet \dta[d] &\bullet \dta[d] &\bullet \dta[d] &\bullet \dta[d] \ar@{}[d]^{\text{\large {=} }}
&\bullet \ar@{-}[rrrrd] & \bullet \ar@{-}@/^/[rr] &\bullet \ar@{-}@/^/[rr] &\bullet & \bullet &\bullet\ar@{-}[d] & \bullet \ar@{-}[d] \ar@{}[d]^{\text{\large $d^*_{2}$ }}\\
      \bullet \ar@{-}[rrrrd] & \bullet \ar@{-}[rrrrd] &\bullet \ar@{-}@/^/[rr] &\bullet \ar@{-}@/^/[rr] & \bullet &\bullet & \bullet \ar@{-}[d]
& \bullet \ar@{-}[r] & \bullet & \bullet \ar@{-}[r] & \bullet &\bullet &\bullet &\bullet  \\
   \ar@{}[u]^{\text{ \large $d^*_{1}$ \  }} \bullet \ar@{-}[r] & \bullet & \bullet \ar@{-}[r] & \bullet &\bullet &\bullet &\bullet }
\end{xy}\end{array};$$

$$ \begin{array}{c} 
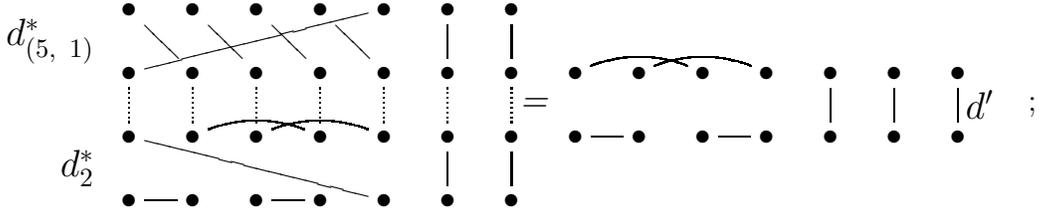
\begin{xy}
\xymatrix@!=0.01pc{ \bullet \ar@{-}[rd] & \bullet \ar@{-}[rd] &\bullet \ar@{-}[rd] &\bullet\ar@{-}[rd] & \bullet \ar@{-}[lllld] &\bullet \ar@{-}[d] & \bullet \ar@{-}[d] \\
  \ar@{}[u]^{\text{ \large $d^*_{(5,\ 1)}$ \  }}  \bullet \dta[d] & \bullet \dta[d] & \bullet \dta[d] & \bullet \dta[d] &\bullet \dta[d] &\bullet \dta[d] &\bullet \dta[d] \ar@{}[d]^{\text{\large {=} }}
&\bullet \ar@{-}@/^/[rr] & \bullet \ar@{-}@/^/[rr] &\bullet &\bullet & \bullet \ar@{-}[d] &\bullet \ar@{-}[d] & \bullet \ar@{-}[d] \ar@{}[d]^{\text{\large $d'$ }}\\
      \bullet \ar@{-}[rrrrd] & \bullet \ar@{-}@/^/[rr] &\bullet \ar@{-}@/^/[rr] &\bullet & \bullet &\bullet\ar@{-}[d] & \bullet \ar@{-}[d]
& \bullet \ar@{-}[r] & \bullet & \bullet \ar@{-}[r] & \bullet &\bullet &\bullet &\bullet  \\
   \ar@{}[u]^{\text{ \large $d^*_{2}$ \  }} \bullet \ar@{-}[r] & \bullet & \bullet \ar@{-}[r] & \bullet &\bullet &\bullet &\bullet }
\end{xy}\end{array};$$

Step 2. Transform $d'$ into $e_{(2)}$
$$ \begin{array}{c} 
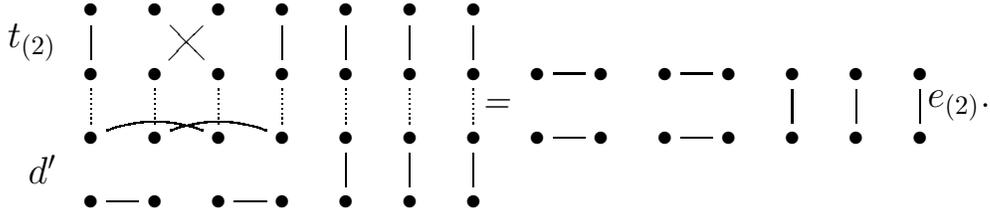
\begin{xy}
\xymatrix@!=0.01pc{ \bullet \ar@{-}[d] & \bullet \ar@{-}[rd] &\bullet \ar@{-}[ld] &\bullet\ar@{-}[d] & \bullet \ar@{-}[d] &\bullet \ar@{-}[d] & \bullet \ar@{-}[d] \\
  \ar@{}[u]^{\text{ \large $t_{(2)}$ \  }}  \bullet \dta[d] & \bullet \dta[d] & \bullet \dta[d] & \bullet \dta[d] &\bullet \dta[d] &\bullet \dta[d] &\bullet \dta[d] \ar@{}[d]^{\text{\large {=} }}
&\bullet \ar@{-}[r] & \bullet &\bullet \ar@{-}[r] &\bullet & \bullet \ar@{-}[d] &\bullet \ar@{-}[d] & \bullet \ar@{-}[d] \ar@{}[d]^{\text{\large $e_{(2)}.$ }}\\
      \bullet \ar@{-}@/^/[rr] & \bullet \ar@{-}@/^/[rr] &\bullet &\bullet & \bullet \ar@{-}[d] &\bullet \ar@{-}[d] & \bullet \ar@{-}[d]
& \bullet \ar@{-}[r] & \bullet & \bullet \ar@{-}[r] & \bullet &\bullet &\bullet &\bullet  \\
   \ar@{}[u]^{\text{ \large $d'$ \  }} \bullet \ar@{-}[r] & \bullet & \bullet \ar@{-}[r] & \bullet &\bullet &\bullet &\bullet }
\end{xy}\end{array}$$
\end{ex}
Now, the diagram $d^*$ is rewritten in the form $\omega e_{(2)},$ where
\begin{align*}
\omega =(d^{*}_{(5,\ 1)}d^{*}_{(6,\ 2)}d^{*}_{(7,\ 3)})^{-1}t_{(2)}
= s_{3, 6}s_{2, 5}s_{1, 4}s_{2}.
\end{align*}
The corresponding basis element with $d^*$ in the $q$-Brauer algebra $Br_{7}(r, q)$ is
$$g_{d^*}= g_{\omega}e_{(2)} = g^{+}_{3,6}g^{+}_{2,5}g^{+}_{1,4}g_{2}e_{(2)}.$$
Using the involution $*$ in the Brauer algebra $D_{n}(N)$ yields the resulting diagram $d$ in which
$$d =  e_{(2)}\omega^{-1} = e_{(2)}(s_{3, 6}s_{2, 5}s_{1, 4}s_{2})^{-1} =e_{(2)}s_{2}s_{4, 1}s_{5, 2}s_{6, 3}.$$
Hence,
$$g_{d}= e_{(2)}g_{\omega^{-1}} =e_{(2)}g_{2} g^{+}_{4,1}g^{+}_{5, 2}g^{+}_{6, 3}.$$
Observe that this result can also be obtained via applying the involution $i$ (see Proposition \ref{md12}) on the $q$-Brauer algebra, that is,
\begin{align*}
i(g_{d^*}) = i(g_{\omega}e_{(2)}) = i(g^{+}_{3,6}g^{+}_{2,5}g^{+}_{1,4}g_{2}e_{(2)})
=e_{(2)}g_{2} g^{+}_{4,1}g^{+}_{5, 2}g^{+}_{6, 3} = g_{d}.
\end{align*}

\begin{rem}\label{nx3}
1. Combining with Lemma \ref{bd2} the above algorithm implies that given a diagram $d^*$ of $\mathscr{D}^*_{k, n}$
there exists a unique element $\omega = t_{n-1}t_{n-2}...t_{2k}t_{2k-2}t_{2k-4}...t_{2} \in B^*_k$
with $t_{j} = s_{i_{j}, j}$ and $i_{j} < i_{j+1}$ for $2k+1 \leq j \leq n-1,$ such that
$d^*= \omega e_{(k)}$ and $\ell(d^*) = \ell(\omega).$ Let
\begin{align}\label{eq23}
B^*_{k, n} = \{\omega \in B^*_k\ |\ d^* = \omega e_{(k)} \text{ and } \ell(d^*) = \ell(\omega), \ d^* \in \mathscr{D}^*_{k, n} \}
\end{align}
and
\begin{align}\label{eq24}
B_{k, n} = \{\omega^{-1} |\ \omega \in B^*_{k, n}\ \}.
\end{align}
By Lemma \ref{bd3}(a), $B_{k, n} = \{ \omega^{-1} \in B_k |\ d = e_{(k)}\omega^{-1} \text{ and } \ell(d) = \ell(\omega^{-1}), \ d \in \mathscr{D}_{k, n} \}$,
where $\mathscr{D}_{k, n}$ is the set of all diagrams $d$ which are image of $d^* \in \mathscr{D}^*_{k, n}$ via the involution $*$.
The uniqueness of element $\omega\in B_{k, n}$ means that $|B^*_{k, n}| = |B_{k, n}| = |\mathscr{D}_{k, n}| = |\mathscr{D}^*_{k, n}|$.
Given a diagram $d^*$ in $\mathscr{D}^*_{k, n},$
since the number of diagrams $d^*$ is equal to the number of possibilities to draw k edges between n vertices on its top row,
it implies that,
$$|B^{*}_{k, n}| = |B_{k, n}| = \dfrac {n!} {2^{k}(n-2k)!k!}.$$

2. For an element $\omega = t_{n-1}t_{n-2}...t_{2k}t_{2k-2}t_{2k-4}...t_{2} \in B^*_{k, n}$
if $t_{j} = 1$ with $2k \leq j \leq n-1$ then $t_{j+1} = 1$.
Indeed, as in \eqref{eq22}, suppose that $t_{j}= s_{f(j+1), j}=1.$
This means the corresponding diagram $d^*_{(f(j+1), j+1)}=1$, that is, a diagram with all vertical edges and no intersection between any two vertical edges.
It implies that the $(j+1)$-st vertex in the bottom row of diagram $d^*$ joins the same vertex of the upper one,
that is, $f(j+1)= j+1$. By definition of $d^*$ in $\mathscr{D}^*_{k, n},$
the other vertical edges on the right side of the $(j+1)$-st vertex of the bottom row has no intersection.
This means the $(j+2)$-th vertex in the bottom row of diagram $d^*$ joins the $f(j+2)=(j+2)$-th vertex of the top row.
Hence, $d^*_{(f(j+2), j+2)} = d^*_{(j+2, j+2)}= 1$, that is, $t_{j+1} = s_{f(j+2), j+1} = s_{j+2, j+1} = 1$.
\end{rem}

\section{Cellular structure of the $q$-Brauer algebra $Br_{n}(r,q)$}
This section is devoted to establish cellularity of the $q$-Brauer algebra.
As usual, let $k$ be an integer, $0 \leq k \leq [n/2]$ and
denote $D_{n}(N)$ the Brauer algebra over ring $R$.
Define $V^*_{k, n}$ an $R$-vector space linearly spanned by $\mathscr{D}^*_{k, n}.$
This implies $V^*_{k, n}$ is an $R$-submodule of $V^*_k$.
Hence, by Lemma \ref{bd2}(a) a given basis diagram $d^*$ in $V^*_{k, n}$ has a unique reduced expression in the form $d^* =\omega e_{(k)}$
with $\ell(d^*) = \ell(\omega)$ and $\omega \in B^*_{k, n}.$
Similarly, let $V_{k, n}$ be an $R$-vector space linearly spanned by $\mathscr{D}_{k, n}$,
that is, a basis diagram $d$ in $V_{k, n}$ has exactly 2k horizontal
edges, its top row is the same as a row of $e_{(k)}$ and there is no intersection between any two vertical edges. \\
The following lemma is directly deduced from definitions.

\begin{lem}\label{bd11}
The $R$-module $V^*_{k, n}$ has a basis
$\{ d^*= \omega e_{(k)},\ \omega \in  B^*_{k, n} \}$.
Dually, the $R$-module $V_{k, n}$ has a basis
$\{ d= e_{(k)}\omega,\ \omega \in  B_{k, n} \}$.
Moreover,
$$dim_{R}V_{k, n} =dim_{R}V^{*}_{k, n} = \dfrac {n!} {2^{k}(n-2k)!k!}.$$
\end{lem}
\medskip
Statements below in Lemmas \ref{bd12}, \ref{bd13}, \ref{bd14} and Corollary \ref{hq3} are needed for later reference.

\begin{lem}\label{bd12}
Let $d^*$ be a basis diagram in $V^*_{k, n}$ and $\pi$ be a permutation in $S_{2k+1, n}.$
Then $d^* \pi$ is a basis diagram in $V^*_k$ satisfying $\ell(d^* \pi) = \ell(d^*) + \ell(\pi).$
\end{lem}

\begin{proof}
Since $d^* \in V^*_{k, n}$, it implies that there exists a unique element $\omega \in B^*_{k, n}$ such that $d^*=\omega e_{(k)}$.
Consider $\pi$ as a diagram in $D_{n}(N)$ and also observe that two diagrams $e_{(k)}$ and $\pi$ commute.
It follows from concatenation of diagrams $d^* \text{ and } \pi$ that the diagram $d^* \pi$ is a basis diagram in $V^*_k$.
This yields $d^* \pi = \omega e_{(k)} \pi = \omega \pi e_{(k)}.$
Since the basis diagram $d^*$ has no intersection between any two vertical edges,
the number of intersections of vertical edges in the resulting basis diagram $d^*\pi$ is equal to
that of the diagram $\pi.$ In fact, the number of intersections of vertical edges in the diagram $\pi$
is equal to its length. This produces $\ell(d^* \pi) = \ell(d^*) + \ell(\pi).$
\end{proof}

\begin{cor}\label{hq3}
Let $\omega$ be a permutation in $B^*_{k, n}$ and $\pi$ be a permutation in $S_{2k+1, n}.$
Then
$$\ell(\omega \pi) = \ell(\omega) + \ell(\pi) \text{ and }  \omega \pi \in B^*_k.$$
\end{cor}
\begin{proof}
Set $d^*=\omega e_{(k)}$, then $\ell(d^*) = \ell(\omega)$ since $\omega \in B^*_{k, n}$.
Lemma \ref{bd12} implies that  $\ell(d^* \pi)= \ell(d^*) + \ell(\pi) = \ell(\omega) + \ell(\pi)$ and $d^*\pi = \omega\pi e_{(k)}$ is a basis diagram in $V^*_k$.
The latter deduces that $\ell(d^*\pi) = \ell(\omega\pi e_{(k)}) \le \ell(\omega\pi)$, that is, $\ell(d^*\pi) \le \ell(\omega\pi)$.
Immediately, the first equality follows from applying the inequality $\ell(\omega) + \ell(\pi) \geq \ell(\omega \pi)$.

Let $d' = d^* \pi \in V^*_k$, then by Lemma \ref{bd2} there uniquely exists an element $\omega'$ in $B^*_k$
such that $d'= \omega'e_{(k)}$ and $\ell(d') = \ell (\omega').$
This implies that
\begin{align} \label{eq25}
d' = \omega\pi e_{(k)} = \omega'e_{(k)}
\end{align}
and
$$\ell(d') = \ell(\omega\pi) = \ell(\omega) + \ell(\pi) = \ell(\omega').$$
The equality (\ref{eq25}) yields $\pi e_{(k)} = \omega^{-1} \omega'e_{(k)}.$
Observe that the basis diagram $\pi e_{(k)}$ is a combination of two separate parts where
the first part includes all horizontal edges on the left side and the second consists of all vertical edges on the other side.
As $\pi$ and $e_{(k)}$ commute, $\pi$ corresponds one-to-one with the second part of the basis diagram $\pi e_{(k)}$.
This implies that the basis diagram $\pi e_{(k)}$ is uniquely presented as a concatenation of diagrams $\omega \in S_{2k+1, n}$ and $e_{(k)}$.
Now, the equality $\pi e_{(k)} = \omega^{-1} \omega'e_{(k)}$ implies that $\pi = \omega^{-1} \omega'$, that is,
$\omega\pi = \omega'$. So that $\omega\pi \in B^*_k$.
\end{proof}

\begin{lem}\label{bd13}
Let $\sigma$ be a permutation in $B^*_k$. Then there exists a unique pair $(\omega', \pi'),$
where $\omega' \in B^*_{k, n}$ and $\pi' \in S_{2k+1, n},$ such that
$\sigma = \omega'\pi'.$
\end{lem}
\begin{proof}
Observe that $\sigma e_{(k)}$ is a basis diagram in $V^*_k$ in which $\ell(\sigma e_{(k)}) = \ell(\sigma)$ by Lemma \ref{bd2}(a).
Now suppose that there are two pairs $(\omega', \pi')$  and $(\omega, \pi)$ satisfying
\begin{align}\label{eq26}
\sigma=\omega'\pi' = \omega \pi,
\end{align}
where $\omega, \ \omega' \in B^*_{k, n}$ and $\pi, \ \pi' \in S_{2k+1, n}.$
Suppose that $\omega \neq \omega'$ then $\omega e_{(k)} \neq \omega' e_{(k)}$ by definition of $B^*_{k, n}$.
This implies two diagrams $\omega e_{(k)}$ and $\omega' e_{(k)}$ differ by position of horizontal edges in their top rows.
By definition of diagrams in $\mathscr{D}^*_{k, n}$,
the position of horizontal edges in the top row of the basis diagram $\omega e_{(k)}$ (or $\omega' e_{(k)}$) is unchanged after concatenating it with an arbitrary element of $S_{2k+1, n}$
on the right. This implies that $\omega e_{(k)}\pi$ differs from $\omega' e_{(k)}\pi'$,
that is, $\omega \pi e_{(k)} \neq \omega'\pi' e_{(k)}.$ However, the equality (\ref{eq26}) yields $\omega\pi e_{(k)} = \omega'\pi' e_{(k)},$ a contradiction.
Thus $\omega = \omega'$, and hence, $\pi = \pi'$ by multiplying (\ref{eq26}) with $\omega^{-1}$ on the left.
\end{proof}

\begin{lem}\label{bd14}
Let $k,\ l$ are integers, $0\leq l,\ k\leq [n/2]$. Let $\omega$ be a permutation in $B^*_{l, n}$ and $\pi$ be a permutation in $S_{2l+1, n}$.
Then there exist $\omega' \in B^*_{k, n}$ and $\pi' \in S_{2k+1, n}$ such that
$$\omega\pi e_{(k)} = \omega'\pi' e_{(k)}.$$
\end{lem}
\begin{proof}
Observe that $\omega\pi e_{(k)}$ is a basis diagram in $V^*_k$ by concatenation of diagrams $\omega,\ \pi \text{ and } e_{(k)}$.
Using Lemma \ref{bd2}(a), there exists a unique element $\sigma \in B^*_k$ such that $\omega\pi e_{(k)}=\sigma e_{(k)}$
and $\ell(\omega \pi e_{(k)}) = \ell(\sigma).$ Lemma \ref{bd13} implies that $\sigma $ can be rewritten in the form $\sigma = \omega'\pi'$,
where $\omega'$ and $\pi'$ are elements uniquely determined in $B^*_{k, n}$ and $S_{2k+1, n},$ respectively.
Therefore, $\sigma e_{(k)} = \omega\pi e_{(k)} = \omega'\pi' e_{(k)}.$
\end{proof}

Note that given $\omega \in B^*_{l, n}$ and $\pi \in S_{2l+1, n}$, Lemma \ref{bd14} can be obtained via a direct calculation using the defining relations as follows:

By Corollary \ref{hq3}, $\omega\pi$ is in the form $\omega\pi = t_{n-1}t_{n-2}...t_{2l}t_{2l-2}t_{2l-4}...t_{2} \in B^*_l$,
where $t_{j} = 1$ or $t_{j} =s_{i_{j}, j}$ for $1 \le i_{j} \le j\le n-1.$
Concatenation of $\omega\pi$ and $e_{(k)}$ will transform $\omega\pi e_{(k)}$ into $\sigma e_{(k)}$,
where $\sigma = t'_{n-1}t'_{n-2}...t'_{2k}t'_{2k-2}t'_{2k-4}...t'_{2} \in B^*_k$
with $t'_{j} = 1$ or $t'_{j} =s_{i_{j}, j}$ for $1 \le i_{j} \le j \le n-1.$
This process is done by using defining relations on the Brauer algebra $D_{n}(N),$ including $(S_{0})$, $(S_{1})$, $(S_{2})$, (3) and (5).
In fact, two final relations yield the vanishing of some transpositions $s_{i}$ in $\omega\pi e_{(k)}$ and
the three first imply an rearrange $\omega\pi e_{(k)} \text{ into } \sigma e_{(k)}.$

\begin{ex}\label{ex6} We fix the element $\omega = s_{7}s_{5, 6}s_{4, 5}s_{1, 4}s_{2}$ in $B^*_{2, 8}$ and the element $\pi = s_{6, 7}s_{5}$ of $S_{5, 8}.$
In the Brauer algebra $D_{8}(N)$ the diagram $\omega \pi e_{(2)}$ corresponds to the diagram $d'= d^*\pi$ which is the result of concatenating $d^*$ and $\pi$ as follows:
$$ \begin{array}{c} 
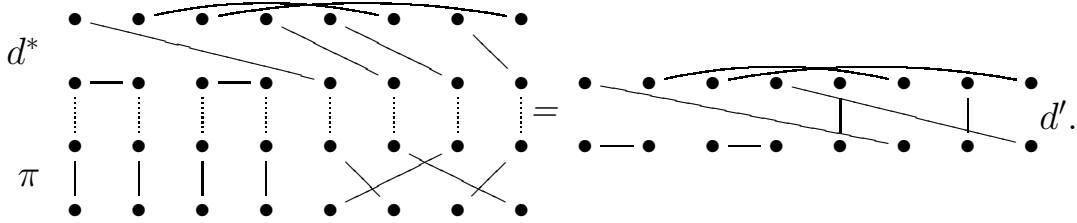
\begin{xy}
\xymatrix@!=0.01pc{\bullet \ar@{-}[rrrrd] & \bullet \ar@{-}@/^/[rrrr] & \bullet \ar@{-}@/^/[rrrrr] &\bullet \ar@{-}[rrd] &\bullet \ar@{-}[rrd] & \bullet &\bullet \ar@{-}[rd] & \bullet \\
  \ar@{}[u]^{\text{ \large $d^*$ \  }} \bullet \dta[d] \ar@{-}[r] & \bullet \dta[d] & \bullet \dta[d] \ar@{-}[r] & \bullet \dta[d] & \bullet \dta[d] &\bullet \dta[d] &\bullet \dta[d] &\bullet \dta[d] \ar@{}[d]^{\text{\large {=} }}
& \bullet \ar@{-}[rrrrrd] & \bullet \ar@{-}@/^/[rrrr] & \bullet \ar@{-}@/^/[rrrrr] &\bullet \ar@{-}[rrrrd] &\bullet \ar@{-}[d] & \bullet &\bullet \ar@{-}[d] & \bullet \ar@{}[d]^{\text{\large $d'.$ }}\\
      \bullet \ar@{-}[d] & \bullet \ar@{-}[d] & \bullet \ar@{-}[d] &\bullet \ar@{-}[d] &\bullet \ar@{-}[rd] & \bullet \ar@{-}[rrd] &\bullet \ar@{-}[lld] & \bullet \ar@{-}[ld]
& \bullet \ar@{-}[r] & \bullet & \bullet \ar@{-}[r] & \bullet & \bullet &\bullet &\bullet &\bullet  \\
   \ar@{}[u]^{\text{ \large $\pi$ \  }} \bullet & \bullet & \bullet & \bullet & \bullet &\bullet &\bullet &\bullet }
\end{xy}\end{array}$$
Using the algorithm in Subsection 3.3, we obtain the element $\sigma= s_{4, 7}s_{6}s_{1, 5}s_{3, 4}s_{2} \in B^*_2$
satisfying
$$d' = \sigma e_{(2)} \text{ and } \ell(d') = \ell(\sigma)= 13.$$
By direct calculation using relations $(S_{1}),$ $(S_{2})$ on the Brauer algebra $D_{n}(N)$,
it also transforms $\omega\pi$ into $\sigma$ as follows:
\begin{align*}
\omega \pi & = (s_{7}s_{5, 6}s_{4, 5}s_{1, 4}s_{2})( s_{6, 7}s_{5})
\overset{(S_{2})}{=} s_{7}s_{5, 6}s_{4, 5}s_{6, 7}s_{1, 5}s_{2} \\
&\overset{(S_{2})}{=} s_{7}s_{5}s_{4}(s_{6}s_{5}s_{6})s_{7}s_{1, 5}s_{2}
\overset{(S_{1})}{=} s_{7}s_{5}s_{4}(s_{5}s_{6}s_{5})s_{7}s_{1, 5}s_{2} \\
&\overset{(S_{1})}{=} s_{7}(s_{4}s_{5}s_{4})s_{6}s_{5}s_{7}s_{1, 5}s_{2}
\overset{(S_{2})}{=} s_{4}s_{5}s_{7}s_{4}(s_{6}s_{7})s_{5}s_{1, 5}s_{2} \\
&\overset{(S_{2})}{=} s_{4}s_{5}(s_{7}s_{6}s_{7})s_{4}s_{5}s_{1, 3}s_{4}s_{5}s_{2}
\overset{(S_{1})}{=} s_{4}s_{5}(s_{6}s_{7}s_{6})s_{4}s_{1, 3}(s_{5}s_{4}s_{5})s_{2} \\
&\overset{(S_{1})}{=} s_{4,7}s_{6}s_{4}s_{1, 3}(s_{4}s_{5}s_{4})s_{2}
\overset{(S_{2})}{=} s_{4,7}s_{6}s_{1, 2}(s_{4}s_{3}s_{4})s_{5}s_{4}s_{2} \\
&\overset{(S_{1})}{=} s_{4,7}s_{6}s_{1, 2}(s_{3}s_{4}s_{3})s_{5}s_{4}s_{2}
\overset{(S_{2})}{=} s_{4,7}s_{6}s_{1, 5}s_{3, 4}s_{2}.
\end{align*}
Thus
\begin{align}\label{eq27}
\omega \pi = s_{4,7}s_{6}s_{1, 5}s_{3, 4}s_{2} \in B^*_2.
\end{align}
Subsequently, concatenating two diagrams $d'$ and $e_{(3)}$ produces the diagram $d''$
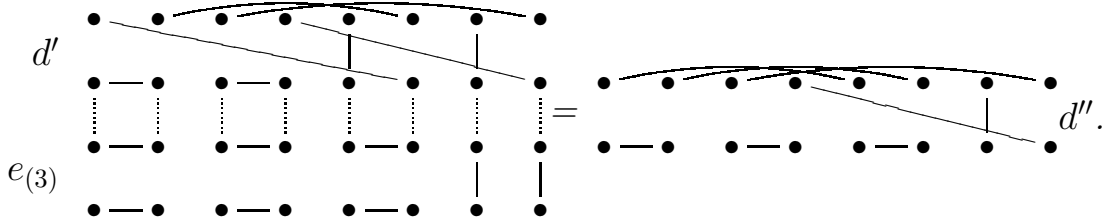
$$ \begin{array}{c} 
\begin{xy}
\xymatrix@!=0.01pc{\bullet \ar@{-}[rrrrrd] & \bullet \ar@{-}@/^/[rrrr] & \bullet \ar@{-}@/^/[rrrrr] &\bullet \ar@{-}[rrrrd] &\bullet \ar@{-}[d] & \bullet &\bullet \ar@{-}[d] & \bullet \\
  \ar@{}[u]^{\text{ \large $d'$ \  }} \bullet \ar@{-}[r] \dta[d] & \bullet \dta[d] & \bullet \dta[d] \ar@{-}[r] & \bullet \dta[d] & \bullet \dta[d] &\bullet \dta[d] &\bullet \dta[d] &\bullet \dta[d] \ar@{}[d]^{\text{\large {=} }}
& \bullet \ar@{-}@/^/[rrrr] & \bullet \ar@{-}@/^/[rrrr] & \bullet \ar@{-}@/^/[rrrrr] &\bullet \ar@{-}[rrrrd] &\bullet & \bullet &\bullet \ar@{-}[d] & \bullet \ar@{}[d]^{\text{\large $d''$. }}\\
      \bullet \ar@{-}[r] & \bullet & \bullet\ar@{-}[r] &\bullet &\bullet \ar@{-}[r] & \bullet &\bullet \ar@{-}[d] & \bullet \ar@{-}[d]
& \bullet \ar@{-}[r] & \bullet & \bullet \ar@{-}[r] & \bullet & \bullet \ar@{-}[r] &\bullet &\bullet &\bullet  \\
   \ar@{}[u]^{\text{ \large $e_{(3)}$ \  }} \bullet \ar@{-}[r] & \bullet & \bullet\ar@{-}[r] &\bullet &\bullet \ar@{-}[r] & \bullet &\bullet & \bullet }
\end{xy}\end{array}$$
The algorithm in Section 3.3 identifies a unique element $\sigma'= s_{4,7}s_{6}s_{3,4}s_{2} \in B^*_3$
satisfying $d'' = \sigma'e_{(3)}$ and $\ell(d'') = \ell(\sigma')= 8$.
In another way, a direct calculation using the defining relations $(S_0),\ (S_1),\ (S_2), (3) \text{ and } (5)$ also yields the same result. In detail,
\begin{align*}
\omega \pi e_{(3)} & = s_{4,7}s_{6}s_{1, 5}s_{3, 4}s_{2}e_{(3)}
\overset{(S_{2})}{=} s_{4,7}s_{6}s_{1, 4}s_{3}s_{2}(s_{5}s_{4}e_{(3)}) \\
&\overset{(5)}{=}  s_{4,7}s_{6}s_{1, 4}s_{3}s_{2}(s_{3}s_{4}e_{(3)})
\overset{(S_{1})}{=} s_{4,7}s_{6}s_{1, 4}(s_{2}s_{3}s_{2})s_{4}e_{(3)} \\
&\overset{(S_{2})}{=} s_{4,7}s_{6}s_{1, 3}s_{2}(s_{4}s_{3}s_{4})s_{2}e_{(3)}
\overset{(S_{1})}{=} s_{4,7}s_{6}s_{1, 3}s_{2}(s_{3}s_{4}s_{3})s_{2}e_{(3)} \\
&\overset{(S_{2})}{=} s_{4,7}s_{6}s_{1, 2}(s_{3}s_{2}s_{3})s_{4}(s_{3}s_{2}e_{(3)})
\overset{(S_{1}),\ (5)}{=} s_{4,7}s_{6}s_{1, 2}(s_{2}s_{3}s_{2})s_{4}(s_{1}s_{2}e_{(3)}) \\
&\overset{(S_{0}),\ (S_{2}) }{=} s_{4,7}s_{6}s_{1}s_{3}s_{4}(s_{2}s_{1}s_{2})e_{(3)}
\overset{(S_{1}),\ (S_{2}) }{=} s_{4,7}s_{6}s_{3}s_{4}s_{1}(s_{1}s_{2}s_{1})e_{(3)} \\
&\overset{(S_{0})}{=} s_{4,7}s_{6}s_{3, 4}s_{2}(s_{1}e_{(3)})
\overset{(3)}{=} s_{4,7}s_{6}s_{3, 4}s_{2}e_{(3)}
\overset{(S_1),\ (S_2)}{=}s_7 s_{4,6}s_{3, 4}s_{2}(s_7e_{(3)})
\end{align*}
Thus, we obtain
\begin{align}\label{eq28}
d'' = d' e_{(3)} = d^* \pi e_{(3)} = \omega \pi e_{(3)} = \sigma' e_{(3)} =\omega' \pi' e_{(3)},
\end{align}
where $\omega' = s_7 s_{4,6}s_{3, 4}s_{2}$ and $\pi' = s_7$.
\end{ex}
Notice that computing the diagrams $\sigma \text{ and } \sigma'$ using the algorithm in Section 3.3 is left for the reader.
\begin{lem}\label{bd15}
There exists a bijection between the $q$-Brauer algebra $Br_{n}(r, q)$ and
$R$-vector space
$$\oplus_{k=0}^{[n/2]}V^*_{k, n}\otimes_{R}V_{k, n}\otimes_{R}H_{2k+1,n}.$$
\end{lem}
\begin{proof}
For a value $k$ the dimension of each $V^*_{k, n}\otimes_{R}V_{k, n}\otimes_{R}H_{2k+1,n}$ is calculated by the formula:
$$dim_{R}V^*_{k, n}\otimes_{R}V_{k, n}\otimes_{R}H_{2k+1,n}= dim_{R}(V^*_{k, n})\cdot dim_{R}(V_{k, n}) \cdot dim_{R} H_{2k+1, n} = (\dfrac {n!} {2^{k}(n-2k)!k!})^{2}\cdot (n-2k)!.$$
In the Brauer algebra $D_{n}(N)$ the number of diagrams $d$ which has exactly 2k horizontal edges is $(\dfrac {n!} {2^{k}(n-2k)!k!})^{2}\cdot (n-2k)!.$
Hence
$$dim_{R}(\oplus_{k=0}^{[n / 2]}V^*_{k, n}\otimes_{R}V_{k, n}\otimes_{R}H_{2k+1,n})= dim_{R}D_{n}(N) = 1\cdot 3 \cdot 5 ... (2n -1).$$
Theorem 3.8(a) in \cite{W3} implies the dimension of the Brauer algebra and the q-Brauer algebra is the same. Therefore
\begin{center}
$dim_{R}\oplus_{k=0}^{[n/2]}V^*_{k, n}\otimes_{R}V_{k, n}\otimes_{R}H_{2k+1,n}= dim_{R}Br_{n}(r, q).$
\end{center}
\end{proof}
Now an explicit isomorphism will be given.
Suppose that $d$ is a diagram with a unique reduced expression $(\omega_{1}, \omega_{(d)}, \omega_{2}),$
where $\omega_{1} \in B^*_{k, n}$, $\omega_{2} \in B_{k, n}$ and $\omega_{(d)} \in S_{2k+1, n}.$
As indicated in Subsection 3.2, the partial diagrams $d_{1}= \omega_{1} e_{(k)}$ and $d_{2}= e_{(k)}\omega_{2}$ with
$\ell(d_{1})= \ell(\omega_{1})$ and $\ell(d_{2})= \ell(\omega_{2})$ are basis diagrams of $V^*_{k, n}$ ($V_{k, n}$), respectively.
Therefore, the diagram $d$ corresponds one-to-one to a basis element
$$d_{1} \otimes d_{2} \otimes g_{\omega_{(d)}}= \omega_{1}e_{(k)}\otimes e_{(k)}\omega_{2} \otimes g_{\omega_{(d)}}$$
of the $R$-vector space $V^*_{k, n}\otimes_{R}V_{k, n}\otimes_{R}H_{2k+1,n}.$
Now, the correspondence between an arbitrary diagram $d$ in the Brauer algebra and a basis element $g_{d}= g_{\omega_{1}}g_{\omega_{(d)}}e_{(k)}g_{\omega_{2}}$ of the q-Brauer algebra
shown in Theorem \ref{dl11}, implies a bijection from $Br_{n}(r, q)$ to $\oplus_{k=0}^{[n/2]}V^*_{k, n}\otimes_{R}V_{k, n}\otimes_{R}H_{2k+1,n}$ linearly spanned by the rule
$$ g_{d}=g_{\omega_{1}}g_{\omega_{(d)}}e_{(k)}g_{\omega_{2}} \longmapsto \omega_{1}e_{(k)}\otimes e_{(k)} \omega_{2}\otimes g_{\omega_{(d)}}.$$

\medskip
From now on, if no confusion can arise, we will denote by $g_{d}$ both a basis element of $B_{n}(r, q)$ and its corresponding representation in
$V^*_{k, n}\otimes_{R}V_{k, n}\otimes_{R}H_{2k+1,n}$.

\subsection{The \textbf{$R$}-bilinear form for \textbf{$V^*_{k, n}\otimes_{R}V_{k, n}\otimes_{R}H_{2k+1,n}$}}
Now we want to construct an $R$-bilinear form
$$ \varphi_{k} :  V_{k, n}\otimes_{R}V^*_{k, n}  \longrightarrow H_{2k+1,n}$$
for each $0\leq k \leq [n/2].$

Given elements $\omega_{1},\ \omega_{2} \in B^*_{k, n}$,
by Lemma \ref{bd15} we form the element $X_{j}:= d^*_{j}\otimes d_{j}\otimes {\bf 1}$
in $V^*_{k, n}\otimes_{R}V_{k, n}\otimes_{R}H_{2k+1,n}$ for $j= 1, 2$,
where $d^*_{j} = {\omega_{j}e_{(k)}}$ and $d_{j} = {e_{(k)}\omega_{j}^{-1}}$.
The corresponding basis element of $X_{j}$ in the q-Brauer algebra is ${X_{j}} = g_{\omega_{j}}e_{(k)}g_{\omega^{-1}_{j}}.$ Then
\begin{align}\label{eq29}
{X_{1}}{X_{2}} = (g_{\omega_{1}}e_{(k)}g_{\omega^{-1}_{1}})(g_{\omega_{2}}e_{(k)}g_{\omega^{-1}_{2}}).
\end{align}
Using Lemma \ref{bd8} for $j=k$ implies
$$e_{(k)}g_{\omega^{-1}_{1}}g_{\omega_{2}}e_{(k)} \in  H_{2k+1, n}e_{(k)} + \sum_{m \geq k+1} H_{n} e_{(m)}H_{n}.$$
Hence, $e_{(k)}g_{\omega^{-1}_{1}}g_{\omega_{2}}e_{(k)}$ can be rewritten
as an $R$-linear combination of the form
$$e_{(k)}g_{\omega^{-1}_{1}}g_{\omega_{2}}e_{(k)} =\sum_{j}a_{j}g_{\omega_{(c_{j})}}e_{(k)}\  +\ a',$$
where $a_{j} \in R$,
$g_{\omega_{(c_{j})}} \in H_{2k+1, n},$
and $a'$ is a linear combination of basis elements in $\sum_{m \geq k+1} H_{n} e_{(m)}H_{n}.$
This implies that
\begin{align*}
{X_{1}}{X_{2}}=g_{\omega_{1}} (\sum_{j}a_{j}g_{\omega_{(c_{j})}}e_{(k)})g_{\omega^{-1}_{2}}
\ + \ g_{\omega_{1}}a^{'}g_{\omega^{-1}_{2}}
= \sum_{j}a_{j}g_{\omega_{1}}g_{\omega_{(c_{j})}}e_{(k)}g_{\omega^{-1}_{2}} \ + \ a,
\end{align*}
where $a$ is an $R$-linear combination in $\sum_{m \geq k+1} H_{n} e_{(m)}H_{n}.$
By Definition \ref{dn10}, the elements $g_{\omega_{1}}g_{\omega_{(c_{j})}}e_{(k)}g_{\omega^{-1}_{2}}$,
denoted by $g_{c_{j}}$, are basis elements in the $q$-Brauer algebra,
and hence, the product $X_{1}X_{2}$ can be rewritten to be
$\sum_{j}a_{j}g_{c_{j}} +\ a.$
Using Lemma \ref{bd15}, $g_{c_{j}}$ can be expressed in the form
$$g_{c_{j}} = \omega_{1}e_{(k)}\otimes {e_{(k)}\omega^{-1}_{2}}\otimes g_{\omega_{(c_{j})}}.$$
Finally, $X_{1}X_{2}$ can be presented as
\begin{align}\label{eq30}
X_{1}X_{2} = \sum_{j} \omega_{1}e_{(k)} \otimes e_{(k)}\omega^{-1}_{2} \otimes a_{j}g_{\omega_{(c_{j})}} + a.
\end{align}
Subsequently, $\varphi_{k} : V_{k, n}\otimes_{R}V^*_{k, n}  \longrightarrow H_{2k+1,n}$
is an $R$-bilinear form defined by
\begin{align}\label{eq31}
\varphi_{k}(d_{1}, d^*_{2}) = \sum_{j} a_{j}g_{\omega_{(c_{j})}} \in H_{2k+1,n}.
\end{align}
In a particular case
\begin{align} \label{eq31'}
\varphi_{k}(e_{(k)}, e_{(k)}) = (\dfrac{r-1} {q-1})^{k} \in H_{2k+1,n}.
\end{align}
Also note that, the equality (\ref{eq30}) and Definition \ref{eq31} yield the product $X_{1}X_{2}$ in the \hspace{2cm} $q$-Brauer algebra $Br_{n}(r, q)$
\begin{align}\label{eq32}
{X_{1}}{X_{2}}=g_{\omega_{1}}\varphi_{k}(d_{1}, d^{*}_{2})e_{(k)}g_{\omega^{-1}_{2}} + a
=g_{\omega_{1}}e_{(k)}\varphi_{k}(d_{1}, d^{*}_{2})g_{\omega^{-1}_{2}} + a.
\end{align}

We define $J_{k}$ to be the $R-module$ generated by the basis elements $g_{d}$ in the q-Brauer algebra $Br_{n}(r,q)$,
where $d$ is a diagram whose number of vertical edges, say $\vartheta(d)$, are less than or equal $n-2k$.
It is clear that $J_{k+1}\subset J_{k}$
and $J_{k}$ is an ideal in $Br_{n}(r, q)$.
By Lemma \ref{bd9} (see the complete proof in \cite{W3}), $J_{k} = \sum_{m=k}^{[n/2]}H_{n}e_{(m)}H_{n}.$

\begin{lem}\label{bd16}
Let $g_{c}= c_{1}\otimes {c_{2}}\otimes g_{\omega_{(c)}}$
and $g_{d}= d_{1}\otimes d_{2}\otimes g_{\omega_{(d)}}$,
where $g_{\omega_{(c)}},\ g_{\omega_{(d)}} \in H_{2k+1, n}$, $c_{1}= \omega_{1}e_{(k)}$, $c_{2}=e_{(k)}\omega_{2}$,
$d_{1}=\delta_{1}e_{(k)}$, $d_{2} =e_{(k)}\delta_{2}$
with $\omega_{1}, \delta_{1} \in B^*_{k, n}$ and $\omega_{2}, \delta_{2} \in B_{k, n}$.
Then
$$g_{c}g_{d} =c_{1}\otimes d_{2}\otimes g_{\omega_{(c)}}\varphi_{k}(c_{2}, d_{1})g_{\omega_{(d)}} \ (mod \ J_{k+1}).$$
\end{lem}

\begin{proof}
Note that the basis diagram $c_{2}=e_{(k)}\omega_{2} \in V_{k, n}$ (similarly $d_{1}=\delta_{1}e_{(k)} \in V^*_{k, n}$)
can be seen as elements
$$g_{c_{2}} = e_{(k)} \otimes e_{(k)}\omega_{2} \otimes {\bf 1}
\text{ and } g_{d_{1}} = \delta_{1} e_{(k)} \otimes e_{(k)} \otimes {\bf 1}$$
in $V^*_{k, n}\otimes_{R}V_{k, n}\otimes_{R}H_{2k+1,n}.$ Lemma \ref{bd15} implies that
$g_{c_{2}} =e_{(k)}g_{\omega_{2}}$ and $g_{d_{1}}=g_{\delta_{1}}e_{(k)}$
are corresponding basis elements in the q-Brauer algebra $Br_{n}(r, q),$ respectively.
Applying Lemma \ref{bd8} for $j=k$, the product of $g_{c_{2}}$ and $g_{d_{1}}$ is
$$g_{c_{2}}g_{d_{1}} = e_{(k)}g_{\omega_{2}}g_{\delta_{1}}e_{(k)} \in  H_{2k+1, n}e_{(k)} + \sum_{m \geq k+1} H_{n} e_{(m)}H_{n}.$$
Therefore,
\begin{align*}
g_{\omega^{-1}_{2}}g_{c_{2}}g_{d_{1}}g_{\delta^{-1}_{1}} = (g_{\omega^{-1}_{2}} e_{(k)}g_{\omega_{2}})(g_{\delta_{1}}e_{(k)}g_{\delta^{-1}_{1}})
\overset{(\ref{eq32})}{=} g_{\omega^{-1}_{2}}\varphi_{k}(c_{2}, d_{1}) e_{(k)}g_{\delta^{-1}_{1}} + a,
\end{align*}
where $a$ is a linear combination of basis elements in $J_{k+1}$ and
$\varphi_{k}(c_{2}, d_{1})$ is the above defined bilinear form.
This means
\begin{align}\label{eq33}
g_{c_{2}}g_{d_{1}}= \varphi_{k}(c_{2}, d_{1}) e_{(k)} + a'
\end{align}
with $a'=  (g_{\omega^{-1}_{2}})^{-1} a (g_{\delta^{-1}_{1}})^{-1} \in J_{k+1}.$
As a consequence, $g_{c}g_{d}$ is formed as product of basis elements:
\begin{align*}
g_{c}g_{d} &= (g_{\omega_{1}}g_{\omega_{(c)}}e_{(k)}g_{\omega_{2}})(g_{\delta_{1}}g_{\omega_{(d)}}e_{(k)}g_{\delta_{2}})\\
&= g_{\omega_{1}}g_{\omega_{(c)}}(e_{(k)}g_{\omega_{2}}g_{\delta_{1}}e_{(k)})g_{\omega_{(d)}}g_{\delta_{2}}\\
&=g_{\omega_{1}}g_{\omega_{(c)}}(g_{c_{2}}g_{d_{1}})g_{\omega_{(d)}}g_{\delta_{2}}\\
&\overset{(\ref{eq33})}{=} g_{\omega_{1}}g_{\omega_{(c)}}(\varphi_{k}(c_{2}, d_{1}) e_{(k)} + a')g_{\omega_{(d)}}g_{\delta_{2}}\\
&=g_{\omega_{1}}g_{\omega_{(c)}}\varphi_{k}(c_{2}, d_{1}) e_{(k)}g_{\omega_{(d)}}g_{\delta_{2}}  + a''
\end{align*}
with $a'' \in J_{k+1}.$
Thus, by Lemma \ref{bd15}, $g_{c}g_{d}$ can be expressed as
\begin{center}
$ g_{c}g_{d} \equiv c_{1} \otimes d_{2} \otimes g_{\omega_{(c)}} \varphi_{k}(c_{2}, d_{1}) g_{\omega_{(d)}}  (mod \ J_{k+1}).$
\end{center}
\end{proof}

Lemma \ref{bd15} implies that $Br_{n}(r, q)$ has a decomposition as $R$-modules :
$$Br_{n}(r, q)= \oplus^{[n/2]}_{k=0}V^*_{k, n}\otimes_{R}V_{k, n}\otimes_{R}H_{2k+1,n}.$$

\begin{lem}\label{bd17}
Let $\omega$ be an arbitrary permutation in $B^*_{k, n}$ and $\pi \in S_{2k+1, n}$. Then
$g_{\omega}g_{\pi}e_{(k)}$ is a basis element in the q-Brauer algebra $Br_{n}(r, q)$.
\end{lem}
\begin{proof}
Corollary \ref{hq3} yields that the basis diagram $d=\omega\pi e_{(k)}$ satisfies
$$\ell(d) = \ell(\omega\pi) = \ell(\omega) + \ell(\pi).$$
This means the pair $(\omega, \pi)$ is a reduced expression of $d.$ Therefore, by Definition \ref{dn10} we get the precise statement.
\end{proof}

\begin{lem}\label{bd18}
Let $k,\ l$ are integers, $0 \leq l, k \leq [n/2]$.
Let $\omega$ be a permutation in $B^*_{l, n}$ and $\pi$ a permutation in $S_{2l+1, n}$.
Then
$$g_{\omega}g_{\pi} e_{(k)} = \sum_{j} a_{j} g_{\omega_{j}'}g_{\pi_{j}'} e_{(k)},$$
where $a_{j} \in R,$ $\omega_{j}' \in B^*_{k, n}$, and $\pi_{j}' \in S_{2k+1, n}$.
\end{lem}

\begin{proof}
Corollary \ref{hq3} implies $\ell(\omega\pi) = \ell(\omega) + \ell(\pi)$.
As a consequence, $g_{\omega}g_{\pi} = g_{\omega \pi}$ by applying Lemma \ref{bd4}(i).
The remainder of proof follows from the correspondence between the Brauer algebra $D_{n}(N)$
and the q-Brauer algebra $Br_{n}(r, q)$ in the following way:

Using the common properties of the Brauer algebra $D_{n}(N)$ in Section 2 and of the \hspace{2cm} $q$-Brauer algebra $Br_{n}(r, q)$ shown in whole Section 3,
hence, the effect of the basis element $g_{\omega \pi}$ on $e_{(k)}$ on the left(right) is similar to this of permutation $\omega \pi$ with respect to diagram $e_{(k)}$, respectively.
in the left (right).

In fact, the operations used to move $\omega \pi e_{(k)}$ into $\omega' \pi' e_{(k)}$ in Lemma \ref{bd14} are $(S_{0}),$ $(S_{1}),$ $(S_{2}),$ (3) and (5) on the Brauer algebra $D_{n}(N).$
In the same way, the product $g_{\omega \pi} e_{(k)}$ transforms into the form $\sum_{j} a_{j} g_{\omega_{j}'}g_{\pi_{j}'} e_{(k)}$
via using corresponding relations $g^{2}_{i} = (q-1) g_{i} + q$ in Definition \ref{dn5}(iii), $(H_{1}),$ $(H_{2})$ on the Hecke algebra of type $A_{n-1}$
as well as two relations in Lemmas \ref{bd7}(b) and \ref{bd6}(c).
\end{proof}

\begin{ex} Continue considering the same diagram as in Example \ref{ex6} with\\
$\omega =  s_{7}s_{5, 6}s_{4, 5}s_{1, 4}s_{2} \in B^*_{2, 8}$ and $\pi = s_{6, 7}s_{5} \in S_{5, 8}.$

As in \ref{eq27}, the product of $g_{\omega}$ and $g_{\pi}$ in the q-Brauer algebra $Br_{n}(r, q)$ is
$$g_{\omega} g_{\pi} \overset{L\ref{bd4}(i)}=  g_{\omega \pi} \overset{(4.3)}{=} g^{+}_{4, 7}g_{6}g^{+}_{1, 5}g^{+}_{3, 4}g_{2}.$$
Hence,
\begin{align*}
&g_{\omega \pi} e_{(3)} = g^{+}_{4,7}g_{6}g^{+}_{1, 5}g^{+}_{3, 4}g_{2}e_{(3)}
\overset{(H_{2})}{=} g^{+}_{4,7}g_{6}g^{+}_{1, 4}g_{3}g_{2}(g_{5}g_{4}e_{(3)}) \\
&\overset{L\ref{bd4}(c)}{=}  g^{+}_{4,7}g_{6}g^{+}_{1, 4}g_{3}g_{2}(g_{3}g_{4}e_{(3)})
\overset{(H_{1})}{=} g^{+}_{4,7}g_{6}g^{+}_{1, 4}(g_{2}g_{3}g_{2})g_{4}e_{(3)} \\
&\overset{(H_{2})}{=} g^{+}_{4,7}g_{6}g^{+}_{1, 3}g_{2}(g_{4}g_{3}g_{4})g_{2}e_{(3)}
\overset{(H_{1})}{=} g^{+}_{4,7}g_{6}g_{1, 3}g_{2}(g_{3}g_{4}g_{3})g_{2}e_{(3)} \\
&\overset{(H_{2})}{=} g^{+}_{4,7}g_{6}g^{+}_{1, 2}(g_{3}g_{2}g_{3})g_{4}(g_{3}g_{2}e_{(3)})
\overset{(H_{1}),\ L3.3(b)}{=} g^{+}_{4,7}g_{6}g^{+}_{1, 2}(g_{2}g_{3}g_{2})g_{4}(g_{1}g_{2}e_{(3)}) \\
&\overset{Def\ref{dn5}(iii),\ (H_{2})}{=} g^{+}_{4,7}g_{6}g_{1}((q-1)g_{2}+ q)g_{3}g_{4}(g_{2}g_{1}g_{2})e_{(3)}\\
&\overset{(H_{1})}{=} g^{+}_{4,7}g_{6}g_{1}((q-1)g_{2}+ q)g_{3}g_{4}(g_{1}g_{2}g_{1})e_{(3)} \\
&\overset{L\ref{bd7}(b)}{=} qg^{+}_{4,7}g_{6}g_{1}((q-1)g_{2}+ q)g_{3}g_{4}g_{1}g_{2}e_{(3)}\\
&\overset{}{=} q(q-1)g^{+}_{4,7}g_{6}g_{1}g_{2}g_{3}g_{4}g_{1}g_{2}e_{(3)} + q^{2}g^{+}_{4,7}g_{6}g_{1}g_{3}g_{4}g_{1}g_{2}e_{(3)}\\
&\overset{(H_{2})}{=} q(q-1)g^{+}_{4,7}g_{6}g^{+}_{1, 4}g^{+}_{1, 2}e_{(3)} + q^{2}g^{+}_{4,7}g_{6}g^{+}_{3, 4}g_{1}g_{1}g_{2}e_{(3)}\\
&\overset{Def\ref{dn5}(iii)}{=} q(q-1)g^{+}_{4,7}g_{6}g^{+}_{1, 4}g^{+}_{1, 2}e_{(3)} + q^{2}g^{+}_{4,7}g_{6}g^{+}_{3, 4}((q-1)g_{1} +q)g_{2}e_{(3)}\\
& = q(q-1)g^{+}_{4,7}g_{6}g^{+}_{1, 4}g^{+}_{1, 2}e_{(3)} + q^{2}(q-1)g^{+}_{4,7}g_{6}g^{+}_{3, 4}g^{+}_{1, 2}e_{(3)} +q^{3}g^{+}_{4,7}g_{6}g^{+}_{3, 4}g_{2}e_{(3)}
\end{align*}
Thus, in the q-Brauer algebra $Br_{n}(r, q)$ the element $g_{\omega \pi} e_{(3)}$ is rewritten
as an $R$- linear combination of elements $g_{\omega_{j}}e_{(3)}$ ($1 \le j \le 3$), where $\omega_{j} \in B^*_3$.
Now, using Lemmas \ref{bd14} and \ref{bd17}, each element $\omega_{j}$ of  $B^*_3$ can be uniquely expressed in the form $\omega_{j}= \omega'_{j}\pi'_{j}$,
where $\omega'_{j} \in B^*_{3, 8}$ and $\pi'_{j} \in S_{7, 8}$, as follows:
\begin{align}\label{eq34}
g_{\omega \pi} e_{(3)} = \sum_{j =1}^{3} a_{j}g_{\omega'_{j}}g_{\pi'_{j}}e_{(3)}
= &q(q-1)g_{7}g^{+}_{4,6}g^{+}_{1, 4}g^{+}_{1, 2}(g_{7}e_{(3)}) \\
&+ q^{2}(q-1)g_{7}g^{+}_{4,6}g^{+}_{3, 4}g^{+}_{1, 2}(g_{7}e_{(3)}) +q^{3}g_{7}g^{+}_{4,6}g^{+}_{3, 4}g_{2}(g_{7}e_{(3)}).\notag
\end{align}
\end{ex}

Note that if fixing $r= q^{N}$ and $q \ra 1$, then $Br_{n}(r, q) \equiv D_{n}(N)$ (see Remark 3.1, \cite{W3}).
In this case $g_{i}$ becomes the transposition $s_{i}$ and the element $e_{(k)}$ can be identified with the diagram $e_{(k)}.$
Hence, the last Lemma coincides with Lemma \ref{bd14}. This means the equality (\ref{eq34}) in the above example recovers the equality (\ref{eq28}) in Example \ref{ex6}.

The next statement shows how to get an ideal in $Br_{n}(r, q)$ from an ideal in Hecke algebras.

\begin{prop}\label{md13}
Let I be an ideal in $H_{2k+1,n}$.
Then $J_{k+1} + V^*_{k, n}\otimes_{R}V_{k, n}\otimes I$
is an ideal in $Br_{n}(r, q).$
\end{prop}

\begin{proof}
Given two elements $g_{c} = c_{1}\otimes c_{2}\otimes g_{\omega_{(c)}}$
with $c \in D_{n}(N)$ and $\vartheta(c)= n-2l$, and \hspace{2cm}
$g_{d} = d_{1}\otimes d_{2}\otimes g_{\omega_(d)}$
with $d \in D_{n}(N)$ and $\vartheta(d) = n-2k$, we need to prove out that:
$$(c_{1}\otimes c_{2}\otimes g_{\omega_{(c)}})(d_{1}\otimes d_{2}\otimes g_{\omega_{(d)}})\equiv b\otimes d_{2}\otimes ag_{\omega_{(d)}} \ (mod \ J_{k+1})$$
for some $b \in V^*_{k}$,
and $a$ is an element in $H_{2k+1,n}$
which is independent of $g_{\omega_{(d)}}$.

This property is shown via considering the multiplication of basis elements
of the q-Brauer algebra $Br_{n}((r, q)$ as in the proof of Lemma \ref{bd15}.
Assume that
$$g_{c} = g_{\omega_{1}}e_{(l)}g_{\omega_{(c)}}g_{\omega_{2 }}
\text{ and } g_{d} = g_{\delta_{1}e_{(k)}}g_{\omega_{(d)}}g_{\delta_{2 }}$$
are basis elements on $Br_{n}(r, q)$, where $\omega_{(d)} \in S_{2k+1, n}$; $\omega_{(c)} \in S_{2l+1, n}$; $\omega_{1},\ \omega^{-1}_{2 } \in B^*_{l, n}$
and $\delta_{1}, \delta^{-1}_{2} \in B^*_{k, n}$.
Then, it implies
\begin{equation*}
g_{c}g_{d} = g_{\omega_{1}}(e_{(l)}g_{\omega_{(c)}}g_{\omega_{2}}g_{\delta_{1}}e_{(k)})g_{\omega_{(d)}}g_{\delta_{2}}. 
\end{equation*}
In the following we consider two separate cases of $l$ and $k$. \\
Case 1. If $l > k,$ then Corollary \ref{hq13}(f) implies that
$$e_{(l)}g_{\omega_{(c)}}g_{\omega_{2 }} g_{\delta_{1}}e_{(k)}  \in e_{(l)}H_{2k+1, n} + \sum_{m \geq l+1} H_{n} e_{(m)}H_{n},$$
and hence
\begin{align*}
g_{c}g_{d} = g_{\omega_{1}}(e_{(l)}g_{\omega_{(c)}}g_{\omega_{2}}g_{\delta_{1}}e_{(k)})g_{\omega_{(d)}}g_{\delta_{2}}
&\in   g_{\omega_{1}}e_{(l)}H_{2k+1, n}g_{\omega_{(d)}}g_{\delta_{2}} + \sum_{m \geq l+1} H_{n} e_{(m)}H_{n}\\
& \subseteq   H_{n}e_{(l)}H_{n} + \sum_{m \geq l+1} H_{n} e_{(m)}H_{n} \\
&= \sum_{m \geq l} H_{n} e_{(m)}H_{n} = J_{l} \overset{l > k}{\subseteq} J_{k+1}.
\end{align*}
Thus, in this case we obtain $g_{c}g_{d} \equiv 0 \ (mod \ J_{k+1})$.\\
Case 2. If $l \leq k$, then by Lemma \ref{bd8},
$$e_{(l)}g_{\omega_{(c)}}g_{\omega_{2 }} g_{\delta_{1}}e_{(k)}  \in H_{2l+1, n}e_{(k)} + \sum_{m \geq k+1} H_{n} e_{(m)}H_{n}$$
and
\begin{align*}
g_{c}g_{d} = g_{\omega_{1}}(e_{(l)}g_{\omega_{(c)}}g_{\omega_{2}}g_{\delta_{1}}e_{(k)})g_{\omega_{(d)}}g_{\delta_{2}}
\in   g_{\omega_{1}}H_{2l+1, n}e_{(k)}g_{\omega_{(d)}}g_{\delta_{2}} + \sum_{m \geq k+1} H_{n} e_{(m)}H_{n}.
\end{align*}
Without loss of generality we may assume that
\begin{align*}
g_{c}g_{d}  = g_{\omega_{1}}(\sum_{i}b_{i} g_{\omega_{(c_{i})}}e_{(k)})g_{\omega_{(d)}}g_{\delta_{2}} + b'
= \sum_{i}b_{i} (g_{\omega_{1}} g_{\omega_{(c_{i})}}e_{(k)})g_{\omega_{(d)}}g_{\delta_{2}} + b',
\end{align*}
where $b_{i} \in R,$ $g_{\omega_{(c_{i})}} \in H_{2l +1, n},$ and $ b' \in \ \sum_{m \geq k+1} H_{n} e_{(m)}H_{n}$.\\
Using Lemma \ref{bd18} with respect to $\omega_{1} \in B_{l}$ and $\omega_{(c_{i})} \in S_{2l +1, n}$, it implies that
$$g_{\omega_{1}} g_{\omega_{(c_{i})}}e_{(k)} = \sum_{j} a_{(i,j)} g_{\omega_{(i,j)}^{'}}g_{\pi_{(i,j)}^{'}} e_{(k)},$$
where $a_{(i,j)} \in R$, $\omega_{(i,j)}^{'} \in B^*_{k, n}$, and $\pi_{(i,j)}^{'} \in S_{2k+1, n}$ for $l \le k.$

Finally, $g_{c}g_{d}$ can be rewritten as an $R$ - linear combination of basis elements in the $q$-Brauer algebra $Br_{n}(r, q)$ as follows:
\begin{align*}
g_{c}g_{d} &= \sum_{i}b_{i} (g_{\omega_{1}} g_{\omega_{(c_{i})}}e_{(k)})g_{\omega_{(d)}}g_{\delta_{2}} + b'\\
&= \sum_{i}b_{i} \big{(} \sum_{j} a_{(i,j)} g_{\omega_{(i,j)}^{'}}g_{\pi_{(i,j)}^{'}} e_{(k)}\big{)} g_{\omega_{(d)}}g_{\delta_{2}} + b'\\
&= \sum_{i} \sum_{j} b_{i}a_{(i,j)} (g_{\omega_{(i,j)}'}g_{\pi_{(i,j)}'} e_{(k)} g_{\omega_{(d)}}g_{\delta_{2}}) + b'\\
&= \sum_{i,\ j} b_{i}a_{(i,j)} g_{\omega_{(i,j)}'}g_{\pi_{(i,j)}'}g_{\omega_{(d)}} e_{(k)} g_{\delta_{2}} + b',
\end{align*}
where $\pi_{(i,j)}' \text{ and  $\omega{(d)}$ are in } S_{2k+1, n}$; $b_{i}a_{(i,j)} \in R$; $\omega_{(i,j)}' \in B^*_{k, n}$ and $ \delta_{2} \in B_{k, n}.$
Lemma \ref{bd15} implies that the corresponding element of $g_{c}g_{d}$ in $V^*_{k, n}\otimes_{R}V_{k, n}\otimes_{R}H_{2k+1,n}$ is
\begin{align*}
g_{c}g_{d} &\equiv \sum_{i,\ j} b_{i}a_{(i,j)}(\omega_{(i,j)}'e_{(k)} \otimes e_{(k)}\delta_{2} \otimes g_{\pi_{(i,j)}'} g_{\omega_{(d)}}) \text{ mod } J_{k+1}\\
&\equiv \sum_{i,\ j} (\omega_{(i,j)}'e_{(k)} \otimes e_{(k)}\delta_{2} \otimes (b_{i}a_{(i,j)}) g_{\pi_{(i,j)}'} g_{\omega_{(d)}}) \text{ mod } J_{k+1}\\
& \equiv b\otimes d_{2}\otimes ag_{\omega_{(d)}} \ (mod \ J_{k+1}),
\end{align*}
where $b=\sum_{i,\ j} \omega_{(i,j)}'e_{(k)} \in V^*_{k}$, $a = \sum_{i,\ j} b_{i}a_{(i,j)}g_{\pi_{(i, j)}'} \in H_{2k +1, n},$ and $d_{2} = e_{(k)} \delta_{2}.$
\end{proof}
The following lemma describes the effect of the involution $i$ of the $q$-Brauer algebra on
$$V^*_{k, n}\otimes_{R}V_{k, n}\otimes_{R}H_{2k+1,n}.$$

\begin{lem}\label{bd19}
(a) If $g_{d} = \ d_{1} \otimes d_{2} \otimes g_{\omega_{(d)}}$
with a reduced expression $d = (\omega_{1},\ \omega_{(d)},\ \omega_{2})$ in $D_{n}(N)$. Then
$$i(g_{d}) = d^{-1}_{2} \otimes d^{-1}_{1} \otimes i(g_{\omega_{(d)}})$$
with  $d^{-1}_{2}=\omega^{-1}_{2}e_{(k)} \in V^*_{k, n}$ and $d^{-1}_{1}=e_{(k)}\omega^{-1}_{2}\in V_{k, n}.$

(b) The involution $i$ on $H_{2k+1, n}$ and the $R-bilinear$ form $\varphi_{k}$
have the following property:
$$i \varphi_{k}(c, d) = \varphi_{k}(d^{-1}, c^{-1})$$
for all $c \in V_{k, n}$ and $d \in V^*_{k, n}.$
\end{lem}
\begin{proof}
Part (a). The basis element corresponding to the reduced expression $d = (\omega_{1},\ \omega_{(d)},\ \omega_{2})$ is $g_{d} = g_{\omega_{1}}e_{(k)}g_{\pi_{(d)}}g_{\omega_{2}}.$
As shown in (\ref{eq19}) of Proposition \ref{md12}, the image of $g_{d}$ via involution $i$ is
\begin{align*}
i(g_{d}) = g_{\omega^{-1}_{2}}e_{(k)}g_{\omega^{-1}_{(d)}}g_{\omega^{-1}_{1}}
\end{align*}
with $\ell(d^{-1}_{2})= \ell(\omega^{-1}_{2}e_{(k)})= \ell(\omega^{-1}_{2})$ and
$\ell(d^{-1}_{1})= \ell(e_{(k)} \omega^{-1}_{1})= \ell(\omega^{-1}_{1})$.
Definitions of $V^*_{k, n}$ and $V_{k, n}$ imply that $d^{-1}_{1} \in V_{k, n}$ and $d^{-1}_{2} \in V^*_{k, n}.$
Hence, by Lemma \ref{bd12}, $i(g_{d})$ can be rewritten as
$$i(g_{d}) = d^{-1}_{2} \otimes d^{-1}_{1} \otimes i(g_{\omega_{(d)}})$$
in $V^*_{k, n}\otimes_{R}V_{k, n}\otimes_{R}H_{2k+1,n}$.\\

Part (b). Using the same argument as in the proof of Lemma \ref{bd16}, $c$ and $d$ can be expressed as basis elements in
$V^*_{k, n}\otimes_{R}V_{k, n}\otimes_{R}H_{2k+1,n}$
with $g_{c} = e_{(k)} \otimes c \otimes \textbf{1}$ and $g_{d} = d \otimes e_{(k)} \otimes \textbf{1}$,
where $\textbf{1}$ is identity in $H_{2k +1, n}$.
As a consequence,
$$g_{c}g_{d} = (e_{(k)} \otimes c \otimes \textbf{1}) (d \otimes e_{(k)} \otimes \textbf{1}) \overset{L\ref{bd16}}{=} e_{(k)} \otimes e_{(k)} \otimes \varphi_{k} (c, d) \ (mod \ J_{k+1}).$$
This means the corresponding basis element in the $q$-Brauer algebra is $g_{c}g_{d}=e_{(k)}\varphi_{k} (c, d) + a$ with $a \in J_{k+1}$.
And hence,
\begin{align*}
i(g_{c}g_{d}) &= i(e_{(k)}\varphi_{k} (c, d) +a)
= i(\varphi_{k} (c, d)) i(e_{(k)}) + i(a)\\
&\overset{P\ref{md12}}{=} i(\varphi_{k} (c, d)) e_{(k)} + i(a)
\overset{Rem\ref{rem1}(1)}{=}e_{(k)} i(\varphi_{k} (c, d)) + i(a).
\end{align*}
Thus, the corresponding element in $V^*_{k, n}\otimes_{R}V_{k, n}\otimes_{R}H_{2k+1,n}$ is of the form
\begin{align}\label{eq35}
i(g_{c}g_{d}) =i(e_{(k)} \otimes e_{(k)} \otimes \varphi_{k} (c, d)) = e_{(k)} \otimes e_{(k)} \otimes i\varphi_{k} (c, d) \ (mod \ J_{k+1})
\end{align}
since $i(a) \in J_{k+1}.$\\
In another way, the part (1) implies that $i(g_{c}) = c^{-1} \otimes e_{(k)} \otimes \textbf{1}$ and $i(g_{d}) = e_{(k)} \otimes d^{-1} \otimes \textbf{1}$.
Therefore,
\begin{align}\label{eq36}
i(g_{d})i(g_{c}) = (e_{(k)} \otimes d^{-1} \otimes \textbf{1})(c^{-1} \otimes e_{(k)} \otimes \textbf{1}) \overset{L\ref{bd16}}{=} e_{(k)} \otimes e_{(k)} \otimes \varphi_{k} (d^{-1}, c^{-1}) \ (mod \ J_{k+1}).
\end{align}
Now the equality $i(g_{c}g_{d}) = i(g_{d})i(g_{c})$ shows that
$$e_{(k)} \otimes e_{(k)} \otimes i\varphi_{k} (c, d)
= e_{(k)} \otimes e_{(k)} \otimes \varphi_{k} (d^{-1}, c^{-1}),$$
that is, $i\varphi_{k} (c, d) = \varphi_{k} (d^{-1}, c^{-1}).$
\end{proof}

The below statement gives an $'iterated\ inflation'$ structure for the $q$-Brauer algebras $Br_{n}(r, q)$,
and the proof of this comes from above results.

\begin{prop}\label{md14} The $q$-Brauer algebra $Br_{n}(r, q)$ is an iterated inflation of Hecke algebras of type $A_{n-1}$.
More precisely: as a free $R$-module, $Br_{n}(r, q)$ is equal to
$$V^*_{0, n}\otimes_{R}V_{0, n}\otimes_{R}H_{n}
    \oplus V^*_{1, n}\otimes_{R}V_{1, n}\otimes_{R} H_{3, n}
    \oplus V^*_{2, n}\otimes_{R}V_{2, n}\otimes_{R} H_{5, n} \oplus ... ,$$
and the iterated inflation starts with $H_{n}$, inflates it along $V^*_{k, n}\otimes_{R}V_{k, n}$ and so on,
ending with an inflation of $R = H_{n,n}$ or $R = H_{n+1,n}$ as bottom layer (depending on whether n is odd or even),
where $H_{2k+1, n}$ is Hecke algebra with  generators $g_{2k + 1}$, $g_{2k + 2}$... $g_{n -1}.$
\end{prop}

\begin{cor}\label{md15} The statement in Proposition \ref{md14} holds on $Br_n(N)$.
\end{cor}
\smallskip
Before giving the main result we need the following lemma.
This lemma is shown in \cite{X1} as a condition to ensure that an algebra has cellular structure.

\begin{lem}\label{bd20}(\cite{X1}, Lemma 3.3)
Let A be a $\varLambda-algebra$ with an involution $i$.
Suppose there is a decomposition
$$ A = \oplus^{m}_{j=1}V_{(j)}
\otimes_{\varLambda}V_{(j)}
\otimes_{\varLambda}B_{j} \quad
{(direct \ sum \ of \ {\varLambda-modules})}$$
where $V_{(j)}$ is a free $\varLambda-modules$
of finite rank and $B_{j}$ is a cellular $\varLambda-algebra$
with respect to an involution $\delta_{j}$
and a cell chain
$J_{1}^{j} \subset J_{2}^{j} ... \subset J_{s_{j}}^{j} = B_{j}$ for each $j$.
Define
$J_{t} = \oplus^{t}_{j=1}V_{(j)}
\otimes_{\varLambda}V_{(j)}
\otimes_{\varLambda}B_{j}$.
Assume that the restriction of $i$ on $\oplus^{t}_{j=1}V_{(j)}\otimes_{\varLambda}V_{(j)}\otimes_{\varLambda}B_{j}$
is given by $w\otimes v\otimes b \longrightarrow v\otimes w\otimes\delta_{j}(b)$.
If for each j there is a bilinear from
$\varphi_{j} : V_{(j)}\otimes_{\varLambda}V_{(j)} \longrightarrow B_{j}$
such that $\delta_{j}(\varphi_{j}(w, v)) = \varphi_{j}(v, w)$ for all $w, v \in V_{(j)}$
and that the multiplication of two elements in $V_{(j)}\otimes_{\varLambda}V_{(j)}\otimes_{\varLambda}B_{j}$
is governed by $\phi_{j}$ modulo $J_{j-1}$;
that is, for $x, y, u, v \in V_{(j)},$ and $b, c \in B_{j}$,
we have
$$ (x\otimes y \otimes b)(u \otimes v \otimes c)
= x \otimes v \otimes b\varphi_{j}(y, u)c$$
modulo the ideal $J_{j-1}$,
and if $V_{(j)}\otimes_{\varLambda}V_{(j)}\otimes_{\varLambda}J^{j}_{l} + J_{j-1}$
is an ideal in A for all l and j, then A is a cellular algebra.
\end{lem}

\begin{thm}\label{dl16}
Suppose that $\varLambda$ is a commutative noetherian ring
which contains $R$ as a subring with the same identity.
If $q$, $r$ and $(r-1)/(q-1)$ are invertible in $\varLambda$,
then the $q$-Brauer algebra $Br_{n}(r, q)$
over the ring $\varLambda$ is cellular
with respect to the involution $i$.
\end{thm}

\begin{proof}
From the above construction, we know that
$V^*_{k, n}$ and $V_{k, n}$ has the same finite rank.
Now apply Lemma \ref{bd20}
to the $q$-Brauer algebra $Br_{n}(r, q)$ with $j=k$.
Set $B_{n-2k} = H_{2k+1, n}$
with $0 \leq k \leq [n/2]$,
then the $q$-Brauer algebra $Br_{n}(r, q)$ has a decomposition
$$Br_{n}(r, q) = V^*_{0, n}\otimes_{\varLambda}V_{0, n}\otimes_{\varLambda}H_{n}
    \oplus V^*_{1, n}\otimes_{\varLambda}V_{1, n}\otimes_{\varLambda} H_{3, n}
    \oplus V^*_{2, n}\otimes_{\varLambda}V_{2, n}\otimes_{\varLambda} H_{5, n} \oplus $$
$$...\oplus V^*_{[n/2]-1, n}\otimes_{\varLambda}V_{[n/2]-1, n}
\otimes_{\varLambda} H_{2[n/2]-1, n}
\oplus V^*_{[n/2], n}\otimes_{\varLambda}V_{[n/2], n}\otimes_{\varLambda} H_{2[n/2]+1, n}, $$
where $H_{n,n} = H_{n+1, n} = \varLambda$ and $V^*_{0, n}= V_{0, n}= \varLambda$.
Lemmas \ref{bd12}, \ref{bd15}, \ref{bd16}, \ref{bd19}, Proposition \ref{md13} and Theorem \ref{dl1} show that
the conditions of the lemma \ref{bd20} applied for $q$-Brauer algebra  are satisfied.
Thus, the q-Brauer algebra $Br_{n}(r, q)$ is a cellular algebra.
\end{proof}

\begin{cor}\label{dl17}
Suppose that $\varLambda$ is a commutative noetherian ring
which contains $R$ as a subring with the same identity.
If $q$, $r$ and $[N]$ are invertible in $\varLambda$,
then the $q$-Brauer algebra $Br_{n}(N)$
over the ring $\varLambda$ is cellular
with respect to the involution $i$.
\end{cor}

\begin{rem}
By Remark \ref{remofqBr}(2), in the case $q=1$ the statements in Corollaries \ref{md15} and \ref{dl17} recover these of the Brauer algebra $D_n(N)$ due to Koenig and Xi (see \cite{KX4}, Theorem 5.6).
\end{rem}


As a consequence of Theorem \ref{dl16}, we have the following parametrization of cell modules for $q$-Brauer algebra.
Given a natural number n, denote by I the set $\{ (n-2k, \la)|$ k is a non negative integer
with $0 \leq k \leq [n/2],$ $\la$ - a partition of $(n - 2k)$ with the shape $(\la_{1}, \la_{2},..., \la_{l})\}.$

\begin{cor}\label{hq17} Let $r$, $q$ and $(r-1)/(q-1)$ are invertible over a commutative Noetherian ring $R$.
The $q$-Brauer algebra $Br_{n}(r, q)$ over $R$ has the set of cell modules
$$ \{\Delta_{k}(\la) = V^*_{k, n}\ten d_{k} \ten \Delta(\la) | (n-2k, \la) \in I \}$$
where $d_{k}$ is non-zero elements in $V_{k, n}$
and $\Delta(\la)$ are cell modules of the Hecke algebra $H_{2k+1, n}$ corresponding to the partition $\la$ of $(n - 2k)$.
\end{cor}

Let us describe here what is the $R$-bilinear form, say $\Phi_{(k, \lambda)}$, on cellular $q$-Brauer algebra $Br_{n}(r, q)$.
For a general definition of the $R$-bilinear form induced by cellular algebras we refer the reader to \cite{GL}.

\begin{defn} \label{dn8} Keep the above notations. Define
$$ \Phi_{(k, \lambda)}:  \Delta_{k}(\la) \times \Delta_{k}(\la) \lra R$$
determined by
$$\Phi_{(k, \lambda)}(d^*_{k}\otimes b  \otimes x,\ c \otimes d_{k} \otimes y):= \phi_{(k, \lambda)}(x\varphi_{k}(b,\ c),\ y)$$
is an $R$-bilinear form, where $(d^*_{k}\otimes b  \otimes x)\ \text{and} \ (c \otimes d_{k} \otimes y)$ are in $\Delta_{k}(\la)$
and $\phi_{(k, \lambda)}$ is the symmetric $R$-bilinear form on the Hecke algebra $H_{2k+1, n}.$
\end{defn}

In the following  we assume that $R$ is a field which contains non-zero elements $q$, $r$.
Let $e(q)$ be the least positive integer $m$ \st $[m]_{q}=1 + q + q^{2}...+ q^{m-1} = 0$, if that exists, let
$ e(q) = \infty$ otherwise. Notice that, $q$ can be also seen as an $e(q)-th$ primitive root of unity
and $ e(q) = \infty$ if $q$ is not an $m-th$ root of unity for all m.
Recall that a partition $\la = (\la_{1}, \la_{2},..., \la_{l})$ is called e(q)-restricted if $\la_{j}-\la_{j+1} < e(q),$ for all j.

\begin{thm}
Let $Br_{n}(r, q)$ be the q-Brauer algebra over an arbitrary field $R$ with characteristic $p \geq 0$. More assume that
$q$, $r$ and $(r-1)/(q-1)$ are invertible in $R$.
Then the non-isomorphic simple $Br_{n}(r, q)$-modules are parametrized by the set
$\{ (n-2k, \la) \in I |$ $\la$ is an $e(q)$-restricted partition of (n - 2k)\}.
\end{thm}

\begin{proof}
By Theorem \ref{dl16}, the $q$-Brauer algebra $Br_{n}(r, q)$ is cellular.
As a consequence, it follows from Corollary \ref{hq17} above and (\cite{GL}, Theorem 3.4)
that the simple $Br_{n}(r, q)$-modules are parametrized by
the set $\{(n-2k, \la) \in I\ |\ \Phi_{(k,\la)} \neq 0 \}.$ If $n-2k \neq 0,$ then
given a pair of basis elements $g_{c}= e_{(k)} \otimes e_{(k)} \otimes g_{\omega_{(c)}}$, $g_{d}= e_{(k)} \otimes e_{(k)} \otimes g_{\omega_{(d)}}$
in $V^*_{k, n}\otimes_{R}V_{k, n}\otimes_{R}H_{2k+1,n}$ with $g_{\omega_{(c)}}$ and $g_{\omega_{(d)}}$ are in $\Delta(\lambda)$,
it follows from Definition \ref{dn8} and Lemma \ref{bd16} that
\begin{align*}
\Phi_{(k, \lambda)}(g_{c}, g_{d}) &=\Phi_{(k, \lambda)} (e_{(k)}\otimes e_{(k)} \otimes g_{\omega_{(c)}},\ e_{(k)}\otimes e_{(k)}\otimes g_{\omega_{(d)}})\\
&= \phi_{(k, \lambda)}(g_{\omega_{(c)}}\varphi_{k}(e_{(k)}, e_{(k)}),\ g_{\omega_{(d)}})\\
&\overset{(\ref{eq31'})}= (\dfrac {r-1} {q-1})^{k}\phi_{(k, \lambda)}(g_{\omega_{(c)}},\ g_{\omega_{(d)}}).
\end{align*}
The last formula implies that $\Phi_{(k,\la)} \neq 0$ if and only if the corresponding linear form
$\phi_{(k,\la)}$ for the cellular algebra $H_{2k+1, n}$ is not zero.
By using a result of Dipper and James (\cite{DJ1}, Theorem 7.6) that states that $\phi_{(k, \la)} \neq 0$ if and only if $\la$ is an $e(q)$-restricted partition of $(n-2k)$
it yields $\Phi_{(k, \lambda)} \neq 0$ if and only if the partition $\lambda$ of n-2k is $e(d)$-restricted.
If $n-2k = 0$ then the last formula above implies $\Phi_{(k, \lambda)} = (\dfrac {r-1} {q-1})^{[n/2]} \neq 0.$
Hence, we obtain the precise statement.
\end{proof}

\begin{rem} The last theorem points out that all simple $Br_{n}(r,q)$-modules are labeled by Young diagrams $[\lambda]$ with n, n-2, n-4, ... boxes,
where each Young diagram $[\lambda]$ with $e(q)$-restricted partition $\lambda$ of $n-2k$ indexes a simple module $\Delta(\lambda)$ of the Hecke algebra $H_{2k+1, n}$
\end{rem}

The next consequence follows from applying Theorem \ref{dl16} and a result in (\cite{KX3}, Proposition 3.2)   on cellular algebras:
\begin{cor}
Let $Br_{n}(r, q)$ be the $q$-Brauer algebra over a field $R$ with $r$, $q$, and  $(r-1)/(q-1)$ are invertible elements.
Then the determinant of the the Cartan matrix C
of the $q$-Brauer algebra is a positive integer,
where the entries of C are by definition the multiplicities of composition factors in indecomposable projective modules.
\end{cor}

As another consequence of Theorem \ref{dl16} and (\cite{GL}, Theorem 3.8),
we obtain the following corollary.

\begin{cor}
Under the assumption of the Theorem \ref{dl16}, the q-Brauer algebra is semisimple \iff
the cell modules are simple and pairwise non-isomorphic.
\end{cor}

\section{Quasi-heredity of $Br_{n}(r, q)$ over a field}
Let $R$ be a field and $q$, $r$ are non-zero elements in R.
It is well-known that in some explicit cases the cellularity of algebras can provide the quasi-hereditary structure on themselve.
More precisely, Koenig and Xi \cite{KX1, KX3} and Xi \cite{X1, X2} pointed out that with suitable choices of parameters
Brauer algebras, Partition algebras, Temperley-Lieb algebras of type A and BMW algebras are quasi-hereditary.
These enable us to be interested in the existence of the quasi-hereditary structure of the $q$-Brauer algebra.
The question is that what conditions the $q$-Brauer algebra is quasi-hereditary.
Let us first recall the definition of quasi-hereditary algebras introduced in \cite{CPS} due to Cline, Parshall and Scott.
The examples of quasi-hereditary algebras are given in \cite{PS, CPS}.

\begin{defn}
Let $A$ be a finite dimensional $R$-algebra.
An ideal $J$ in $A$ is called a {\it hereditary ideal} if $J$ is idempotent,
$J(rad(A))J = 0 $ and $J$ is a projective left (or right) $A$-module.
The algebra $A$ is called {\it quasi-hereditary} provided there is a finite chain
$0 = J_{0} \subset J_{1} \subset J_{0} \subset...\subset J_{n} = A$
of ideals in $A$ \st $J_{i+1}/J_{i}$ is a hereditary ideal in $A/J_{i}$ for all i.
Such a chain is then called a heredity chain of the quasi-hereditary~algebra~$A$.
\end{defn}
Subsequently, we give an positive answer with respect to q-Brauer algebras.
We state below when the $q$-Brauer algebra is quasi-hereditary.
It is well-known that the Hecke algebra $H_{n}$ is semisimple \iff
$e(q) > n$. This is shown by Dipper and James (see \cite{DJ1} or \cite{DJ2} for more details).

\begin{thm}\label{qhthm}
Let $R$ be any field with $0 \neq q,\ r \in R$.
Assume more that  $(r-1)/(q-1)\neq0$.
Then the $q$-Brauer algebra $Br_{n}(r, q)$ is quasi-hereditary \iff
$e(q)$ is strictly bigger than n.
\end{thm}
\begin{proof}
\smallskip
Suppose that the $q$-Brauer algebra is quasi-hereditary. For any choices of non-zero parameter,
the $q$-Brauer algebra has a quotient $H_{n}$, the Hecke algebra of type $A_{n-1}$, and
this quotient actually arises as $Br_{n}(r, q)/J_{1}$ for some ideal in the cell chain.
Furthermore, we know by \cite{KX3} that any cell chain is a hereditary chain and
also note that a self-injective algebra is quasi-heredity \iff it is semisimple.
Thus, as a consequence, the Hecke algebra $H_{n}$ is semisimple, that is, $e(q) > n$.

Conversely, if $e(q)$ is strictly bigger than n, then using Theorem 7.6 \cite{DJ1}, $H_{n}$ is semisimple and so, all subalgebras $H_{2k+1, n}$ of $H_{n}$ are semisimple.
To prove that under our assumption the algebra $Br_{n}(r, q)$ is quasi-hereditary, we need to show by \cite{KX3} that the square
$(V^*_{k, n}\otimes_{R}V_{k, n}\otimes_{R}H_{2k+1,n})2$ is not zero modulo $J_{k+1}$ for all $0 \leq k \leq [n/2]$.
Proceed as in \cite{KX3},
let $\{ g^{\lambda}_{S, T} \ \mid \ \lambda$ is a partition of the form $(\lambda_{1}, \lambda_{2},..., \lambda_{l})$ and $S, T$ are standard tableaux of shape $\lambda  \}$
be a cellular basis of the semisimple cellular algebra $H_{2k+1, n}$.
Then there are two elements $g^{\lambda}_{S, T}$ and $g^{\lambda}_{U, V}$ such that $g^{\lambda}_{S, T}g^{\lambda}_{U, V}$
is not zero modulo the span of all $g^{\mu}_{S^{'}, T^{'}}$, where $\mu$ strictly smaller than $\lambda$
and $S',\ T'$ are standard tableau of shape $\mu$ .
Take the element $e_{(k)}$ in $V^*_{k, n}$ ($V_{k, n}$) and consider the product of
$e_{(k)} \otimes e_{(k)} \otimes g^{\lambda}_{S, T}$ and $e_{(k)}\otimes e_{(k)} \otimes g^{\lambda}_{U, V}.$
By Lemma \ref{bd16}
\begin{align*}
x&\overset{}{:=} (e_{(k)} \otimes e_{(k)} \otimes g^{\lambda}_{S, T})(e_{(k)}\otimes e_{(k)} \otimes g^{\lambda}_{U, V}) \\
&\overset{}{\equiv} e_{(k)} \otimes e_{(k)} \otimes g^{\lambda}_{S, T} \varphi_{k}(e_{(k)}, e_{(k)})g^{\lambda}_{U, V}\\
&\overset{(\ref{eq31'})}{\equiv} e_{(k)} \otimes e_{(k)} \otimes g^{\lambda}_{S, T} (\dfrac {r-1} {q-1})^{k} g^{\lambda}_{U, V}\\
&\overset{}{\equiv} e_{(k)} \otimes e_{(k)} \otimes (\dfrac {r-1} {q-1})^{k}(g^{\lambda}_{S, T}g^{\lambda}_{U, V})\ (mod \ J_{k+1}).
\end{align*}
Since $\dfrac {r-1} {q-1} \neq 0 $, x is non-zero modulo $J_{k+1}$.
\end{proof}

As a consequence of the last theorem and (Theorem 3.1, \cite{KX3}), we obtain the following corollary.

\begin{cor}
Let $R$ be any field with $q$, $r$, and  $(r-1)/(q-1)$ are simultaneously non-zero in $R.$
If $e(q) > n$, then

(a)  $Br_{n}(r, q)$ has finite global dimension.

(b) The Cartan determinant of  $Br_{n}(r, q)$  is 1.

(c) The simple  $Br_{n}(r, q)- module$ can be parametrized by the set of all pairs $(n-2k, \lambda)$
with $0 \leq k \leq [n/2]$ and $\lambda$ is an $e(q)$-restricted partition of (n-2k) of the form $(\lambda_{1}, \lambda_{2},..., \lambda_{l}).$
\end{cor}
We obtain the similar results for the version $Br_n(N)$ as follows:

\begin{cor} \label{hqofBrN1}
Let $R$ be any field with $q$, $r$, and $[N]$ are simultaneously non-zero in $R$.
Then the $q$-Brauer algebra $Br_{n}(N)$ is quasi-hereditary \iff
$e(q)$ is strictly bigger than n.
\end{cor}

\begin{cor} \label{hqofBrN2}
Let $R$ be any field with $q$, $r$, and $[N]$ are simultaneously non-zero in $R$.
If $e(q) > n$, then

(a)  $Br_{n}(N)$ has finite global dimension.

(b) The Cartan determinant of  $Br_{n}(N)$  is 1.

(c) The simple  $Br_{n}(N)- module$ can be parametrized by the set of all pairs $(n-2k, \lambda)$
with $0 \leq k \leq [n/2]$ and $\lambda$ is an $e(q)$-restricted partition of (n-2k) of the form $(\lambda_{1}, \lambda_{2},..., \lambda_{l}).$
\end{cor}

\begin{rem} Note that if $q=1$ then $[m]_{q}=1 + 1 + 1^{2}...+ 1^{m-1} = m = 0$
and hence $e(q)$ is equal to the characteristic of $R$.
As a consequence, the result in Corollary \ref{hqofBrN1} recovers this of the Brauer algebra $D_n(N)$ with the nonzero parameter due to Koenig and Xi (see \cite{KX3}, Theorem 1.3).
\end{rem}
\bigskip

{\bf{Acknowledgments}}

I would like to thank Professor Hans Wenzl for useful discussions. I am specially grateful to my supervisor, Professor Steffen Koenig, for constant support and valuable advice during this work.
The research work is financially supported by the Project MOET-322 of the Training and Education Ministry of Vietnam
and (partly) by the DFG Priority Program SPP-1489. I would also like to express my gratitude for this.

\end{document}